\documentclass[a4paper]{amsart}

\usepackage{amsmath,amssymb,amsthm,stmaryrd}

\usepackage{color,graphicx}
\usepackage[dvipsnames]{xcolor}
\usepackage{accents}
\usepackage{enumerate}

\usepackage[textsize=footnotesize,color=yellow!40, bordercolor=white]{todonotes}

\newtheorem{theorem}{Theorem}[section]

\newtheorem{lemma}[theorem]{Lemma}
\newtheorem{fact}[theorem]{Fact}

\newcounter{cl}[theorem]
\newtheorem{claim}[cl]{Claim}
\newcounter{scl}[cl]
\newtheorem{subclaim}[scl]{Subclaim}
\newtheorem*{claim*}{Claim}
\newtheorem*{subclaim*}{Subclaim}
\newtheorem{corollary}[theorem]{Corollary}

\theoremstyle{definition}
\newtheorem{definition}[theorem]{Definition}
\newtheorem*{definition*}{Definition}
\newcounter{ca}[cl]
\newtheorem{case}[ca]{Case}
\newcounter{sca}[ca]
\newtheorem{subcase}[sca]{Subcase}

\newcounter{ssca}[sca]

\theoremstyle{remark}
\newtheorem{remark}[theorem]{Remark}
\newtheorem*{remark*}{Remark}

\usepackage{tikz}
\usetikzlibrary{decorations.pathmorphing,shapes}

\usepackage[
    colorlinks=true, 
    citecolor=blue,
    linkcolor=blue,
    urlcolor=blue]{hyperref}
    
\newcommand{\ZFC}{\ensuremath{\operatorname{ZFC}} }
\newcommand{\ZF}{\ensuremath{\operatorname{ZF}} }

\newcommand{\Ord}{\ensuremath{\operatorname{Ord}} }
\newcommand{\HOD}{\ensuremath{\operatorname{HOD}} }
\newcommand{\HC}{\ensuremath{\operatorname{HC}} }
\newcommand{\OD}{\ensuremath{\operatorname{OD}} }

\newcommand{\rng}{\ensuremath{\operatorname{rng}} }

\newcommand{\AD}{\ensuremath{\operatorname{AD}} }
\newcommand{\crit}{\ensuremath{\operatorname{crit}} }

\newcommand{\Col}{\ensuremath{\operatorname{Col}} }
\newcommand{\lh}{\ensuremath{\operatorname{lh}} }
\newcommand{\Ult}{\ensuremath{\operatorname{Ult}} }
\newcommand{\ran}{\ensuremath{\operatorname{ran}} }
\newcommand{\dom}{\ensuremath{\operatorname{dom}} }
\newcommand{\pred}{\ensuremath{\operatorname{pred}} }

\newcommand{\id}{\ensuremath{\operatorname{id}} }

\newcommand{\dirlim}{\ensuremath{\operatorname{dirlim}} }
\newcommand{\str}{\ensuremath{\operatorname{str}} }

\newcommand{\BS}{{}^\omega\omega}

\newcommand{\cQ}{\mathcal{Q}}
\newcommand{\cP}{\mathcal{P}}
\newcommand{\cM}{\mathcal{M}}
\newcommand{\cN}{\mathcal{N}}
\newcommand{\cT}{\mathcal{T}}
\newcommand{\cU}{\mathcal{U}}

\newcommand{\cF}{\mathcal{F}}
\newcommand{\cI}{\mathcal{I}}
\newcommand{\cR}{\mathcal{R}}
\newcommand{\cS}{\mathcal{S}}

\newcommand{\cC}{\mathcal{C}}
\newcommand{\cW}{\mathcal{W}}

\newcommand{\cO}{\mathcal{O}}
\newcommand{\cL}{\mathcal{L}}

\newcommand{\bR}{\mathbb{R}}
\newcommand{\bP}{\mathbb{P}}

\newcommand{\bC}{\mathbb{C}}

\newcommand{\ind}{\ensuremath{\operatorname{index}}}

\newcommand{\Pot}{\mathcal{P}}

\newcommand{\pwimg}{\, "}
\newcommand{\dbar}[1]{\overline{\overline{#1}}}

\begin{document}
\title[The strength of determinacy when all sets are universally Baire]{The consistency strength of determinacy when all sets are universally Baire}



\author{Sandra M\"uller} \address{Sandra M\"uller, Institut f\"ur Diskrete Mathematik und Geometrie, TU Wien, Wiedner Hauptstra{\ss}e 8-10/104, 1040 Wien, Austria.}
\email{sandra.mueller@tuwien.ac.at}
\thanks{The author gratefully acknowledges funding from L'OR\'{E}AL Austria, in collaboration with the Austrian UNESCO Commission and in cooperation with the Austrian Academy of Sciences - Fellowship \emph{Determinacy and Large Cardinals}. In addition, this research was funded in part by the Austrian Science Fund (FWF) [10.55776/V844, 10.55776/Y1498, 10.55776/I6087]. For open access purposes, the author has applied a CC BY public copyright license to any author accepted manuscript version arising from this submission.
   Moreover, the author would like to thank Grigor Sargsyan for helpful discussions at the beginning of this project and the support through his NSF Career Award DMS-1352034 during a research visit at Rutgers University in the fall of 2019. In addition, the author would like to thank Farmer Schlutzenberg, Benjamin Siskind, and Lena Wallner for helpful comments on an earlier version of this article. Finally, the author is grateful for the very helpful comments and suggestions provided by the anonymous referee.}




\begin{abstract}
  It is known that the large cardinal strength of the Axiom of Determinacy when enhanced with the hypothesis that all sets of reals are universally Baire is much stronger than the Axiom of Determinacy itself. Sargsyan conjectured it to be as strong as the existence of a cardinal that is both a limit of Woodin cardinals and a limit of strong cardinals. Larson, Sargsyan and Wilson used a generalization of Woodin's derived model construction to show that this conjectured result would be optimal. In this paper we introduce a new translation procedure for hybrid mice extending work of Steel, Zhu and Sargsyan and apply it to prove Sargsyan's conjecture.
\end{abstract}
\maketitle
\setcounter{tocdepth}{1}

\section{Introduction}

A central theme in inner model theory is the deep connection between determinacy principles and inner models with large cardinals. One hierarchy of determinacy principles is given by imposing additional structural properties on a model of the Axiom of Determinacy ($\AD$). Examples of such structural properties are ``$\theta_0 < \Theta$'', ``all sets of reals are Suslin'', ``$\Theta$ is regular'', or the Largest Suslin Axiom, see for example \cite{St09, St_ADR, Sa15, SaTr}. Woodin used his famous derived model construction to show that if there is a cardinal $\lambda$ that is a limit of Woodin cardinals and a limit of ${<}\lambda$-strong cardinals, then there is a model of $\text{``}\AD^+ + \text{ all sets of reals are Suslin''},$ see \cite{St09}. Approximations to the converse were proved by Woodin,
and the full converse was proved by Steel \cite{St_ADR}, using Woodin's work together with precursors of the translation procedures we use in this paper. By results of Martin and Woodin, see \cite[Theorems 9.1 and 9.2]{St09}, assuming $\AD$, the statement ``all sets of reals are Suslin'' is equivalent to the Axiom of Determinacy for games on reals ($\AD_\bR$).

Being universally Baire as defined by Feng, Magidor and Woodin \cite{FMW92} is a non-trivial strengthening of being Suslin. Larson, Sargsyan and Wilson showed in \cite{LSW} via an extension of Woodin's derived model construction that if there is a cardinal $\lambda$ that is a limit of Woodin cardinals and a limit of strong cardinals, then there is a model of $\text{``}\AD +$ all sets of reals are universally Baire''. One challenge in their argument is that no model of the form $L(\Pot(\bR))$ is a model of $\text{``}\AD + $ all sets of reals are universally Baire''. So the resulting model of determinacy needs to have non-trivial structure above $\Pot(\bR)$. Sargsyan conjectured that a cardinal $\lambda$ that is a limit of Woodin cardinals and a limit of strong cardinals is the optimal hypothesis to construct a model of $\text{``}\AD + $ all sets of reals are universally Baire''. We prove here that this conjecture is true.
More precisely, we extend Sargsyan's translation procedure from \cite{Sa17} to prove the following theorem. 

\begin{theorem}\label{thm:main}
  Suppose there is a proper class model of the theory ``$\AD +$ all sets of reals are universally Baire''. Then there is a transitive model $\cM$ of $\ZFC$ containing all ordinals such that $\cM$ has a cardinal $\lambda$ that is a limit of Woodin cardinals and a limit of strong cardinals.
\end{theorem}


Translation procedures can be used to transform a hybrid strategy mouse (a \emph{hod mouse}) into an ordinary mouse while generating additional large cardinals from the strategy predicate. They originate from constructions in the \emph{core model induction} method due to Woodin. The idea behind this method is to construct new mice with iteration strategies of high desciptive set theoretic complexity. Once this method reaches infinitely many Woodin cardinals, a key step in every core model induction is to pass from a hybrid strategy mouse to an ordinary mouse without loosing its descriptive set theoretic complexity. Here is where translation procedures start coming into play: They translate hybrid strategy mice into ordinary mice while ensuring that the derived models of the original strategy mouse and the resulting ordinary mouse are the same. This is what we mean by \emph{preserving the descriptive set theoretic complexity} here.
In \cite{St08dm}, Steel develops a translation procedure that can be used to prove Woodin's result that \[ \ZF + \AD^+ + \; \theta_0 < \Theta \] is equiconsistent with \begin{eqnarray*} & \ZFC \text{ + } \text{ there is a pair of cardinals } \kappa, \lambda \text{ such that } \\ & \lambda \text{ is a limit of Woodin cardinals and } \kappa \text{ is }{<}\lambda\text{-strong.} \end{eqnarray*}
Steel's translation procedure in \cite{St08dm} originates from work of Closson, Neeman, and Steel and was later extended by Zhu in \cite{Zh15}. Steel's and Zhu's constructions use $\AD^+$ theory and work in a specific hod mouse, more specifially, they use a real parameter for the translation. Therefore, their constructions have the disadvantage that the translation cannot be done in $\HOD$ itself but only in $\HOD$ of a real.

This was improved by Sargsyan in \cite{Sa17}, where he developed a translation procedure that works locally and constructs a model with infinitely many Woodin cardinals and one strong cardinal. Sargsyan's construction does not depend on the choice of a specific hod mouse or a fixed real parameter and can therefore be done in $\HOD$ itself. However, it was open how to continue Sargsyan's translation after the first level reaching a strong cardinal. We extend Sargsyan's work and push it through infinitely many levels by developing a new translation procedure that allows us to construct a model with a cardinal that is both a limit of Woodin cardinals and a limit of strong cardinals, cf.~Definition \ref{def:translationprocedure}. The model $\cM$ we construct will in fact be a countably iterable premouse. Note that the premouse obtained from Sargsyan's translation in \cite{Sa17} was not known to be iterable.

\begin{theorem}\label{thm:maintranslation}
  Let $\cW$ be a translatable structure as in Definition \ref{def:translatablestructure} and let $\cM$ be the result of an $\cM$-construction in $\cW$ as in Definition \ref{def:translationprocedure}. Then $\cM$ is a countably iterable proper class premouse with a cardinal $\lambda$ that is a limit of Woodin cardinals and a limit of strong cardinals.
\end{theorem}

The amount of iterability we obtain for the translated premouse $\cM$ is in fact a bit stronger than what is stated in Theorem \ref{thm:maintranslation}, see Lemma \ref{lem:iterability}.

By combining our work here with the results in \cite{LSW} we obtain the following corollary.

\begin{corollary}
  The following theories are equiconsistent.
  \begin{description}
  \item[$T_1$] $\ZF + \AD +$ all sets of reals are universally Baire.
  \item[$T_2$] $\ZFC +$ there is a limit of Woodin cardinals that is a limit of strong cardinals.
  \end{description}
\end{corollary}

As a corollary to our proof we obtain a proof of Steel's unpublished result \cite{St_ADR} on the exact consistency strength lower bound for $\AD_\bR$ mentioned at the beginning of this introduction. We thank John Steel for his permission to include it here.

\begin{corollary}[Steel]
    Suppose there is a proper class model of $\AD_\bR$. Then the $\AD_\bR$-hypothesis is consistent, i.e., there is a transitive model $\cM$ of $\ZFC$ containing all ordinals such that $\cM$ has a cardinal $\lambda$ that is a limit of Woodin cardinals and a limit of ${<}\lambda$-strong cardinals.
\end{corollary}

The corollary follows from our argument by observing that Sargsyan's \cite{Sa09, Sa15} yields a hod mouse with a limit of Woodin cardinals $\lambda$ and strategies up to $\lambda$ from $\AD_\bR$. Restricting the translation procedure of this article to strategies below $\lambda$ results in a transitive model $\cM$ of $\ZFC$ containing all ordinals such that $\cM$ has a cardinal $\lambda$ that is a limit of Woodin cardinals and a limit of ${<}\lambda$-strong cardinals. Steel's original unpublished proof was different from ours. He, assuming $\AD_\bR$, Prikry-forced a hod mouse (in the sense of \cite{Sa09}) $H$ whose derived model in $V$ and then translated it into a model of the $\AD_\bR$-hypothesis whose derived model is still $V$. His proof that the translation works used that certain mouse operators are equivalent on a cone, and that the bases of those cones have been made generic over $H$ at the appropriate places. 

\section{Preliminaries and notation}

We start with the definition of universally Baire sets of reals; see, for example, \cite[Definition 32.21]{Je03} or \cite[Definition 8.6]{Sch14}.

  \begin{definition}
    Let $(S,T)$ be trees on $\omega \times \kappa$ for some ordinal $\kappa$ and let $Z$ be any set. We say \emph{$(S,T)$ is $Z$-absolutely complementing} iff \[ p[S] = \BS \setminus p[T] \] in every $\Col(\omega,Z)$-generic extension of $V$.
  \end{definition}
  \begin{definition}[Feng-Magidor-Woodin, \cite{FMW92}]\label{def:uB}
    A set of reals $A$ is \emph{universally Baire (uB)} if for every $Z$, there are $Z$-absolutely complementing trees $(S,T)$ with $p[S] = A$.
  \end{definition}

  We assume that the reader is familiar with the basic concepts of inner model theory and refer to \cite{St10} for an overview. But following \cite{Sa17} we use a slight variant of Mitchell-Steel indexing for the extenders on the sequence of our premice. 
  We say an ordinal $\xi$ is a \emph{$0$-weak cutpoint} in a premouse $\cM$ iff $\xi$ is a cutpoint in $\cM$, i.e., there is no extender $E_\alpha^\cM$ on the $\cM$-sequence such that $\crit(E_\alpha^\cM) < \xi < \alpha$. Furthermore, for $i<\omega$, we recursively say $\xi$ is an \emph{$i+1$-weak cutpoint} in $\cM$ iff the only extenders on the $\cM$-sequence overlapping $\xi$ have critical points $\kappa_j$, $j \leq i$, that are $j$-weak cutpoints and limits of $j$-weak cutpoints in $\cM$. More precisely, $\xi$ is an $i+1$-weak cutpoint iff for every extender $E_\alpha^\cM$ on the sequence of $\cM$, if $\crit(E_\alpha^\cM) < \xi < \alpha$, then $\crit(E_\alpha^\cM)$ is a $j$-weak cutpoint that is a limit of $j$-weak cutpoints in $\cM$ for some $j \leq i$.
  
  Analogous to the indexing in \cite{Sa17}, the following extenders are indexed differently than in Mitchell-Steel indexing: Suppose $E$ is an extender on the sequence of a premouse $\cM$ with critical point $\kappa$ that is a $j$-weak cutpoint and a limit of $j$-weak cutpoints for some $j<\omega$. Then $E$ is indexed at the successor of the least $j$-weak cutpoint above $\kappa$ of its own ultrapower, for the minimal such $j$. More precisely, a fine extender sequence in this article is defined as follows: We use the notation in \cite{St10}. In addition, if $E$ is an extender on the sequence of $\cM = J_\beta^{\vec E}$ such that $\crit(E)$ is a $j$-weak cutpoint and a limit of $j$-weak cutpoints in $\cM$ for some $j<\omega$, we say \emph{$E$ is of type S}\footnote{Here ``S'' stands for Sargsyan as this is a direct generalization of the indexing suggested by Sargsyan in \cite{Sa17}.}. Another difference is that, similar as in \cite{Schl15}, we do not require coherence above $\lambda$, the limit of Woodin and strong cardinals we aim to obtain. The reason for this is that in the model we construct there will be extenders indexed throughout the ordinals but no extenders on the sequence have critical point $\geq \lambda$. So coherence cannot hold above $\lambda$.

  \begin{definition}
      For an ordinal $\lambda$, a \emph{fine extender sequence up to $\lambda$} in this article is a sequence $\vec E$ such that for each $\alpha \in \dom(\vec E)$, $\vec E$ is acceptable at $\alpha$, and either $(\vec E)_\alpha = \emptyset$, or $E_\alpha = (\vec E)_\alpha$ is a $(\kappa, \alpha)$-pre-extender over $J_\alpha^{\vec E}$ for some $\kappa$ such that $J_\alpha^{\vec E} \models$ ``$\kappa^+$ exists'', and
      \begin{enumerate}
          \item \begin{enumerate}
              \item if $E_\alpha$ is not of type S, $E_\alpha$ is the trivial completion of $E_\alpha\upharpoonright\nu(E_\alpha)$, hence $\alpha = (\nu(E_\alpha)^+)^{\Ult(J_\alpha^{\vec E}, E_\alpha)}$, and $E_\alpha$ is not of type Z, and we let $\nu^*(E_\alpha) = \nu(E_\alpha)$,
              \item if $E_\alpha$ is of type S, $\alpha = (\xi^+)^{\Ult(J_\alpha^{\vec E},E_\alpha)}$, where $\xi$ is the least $j$-weak cutpoint of $\Ult(J_\alpha^{\vec E}, E_\alpha)$ above $\kappa$ for the smallest $j<\omega$ such that $\kappa$ is a $j$-weak cutpoint and a limit of $j$-weak cutpoints, and we let $\nu^*(E_\alpha) = \xi$,
          \end{enumerate}
          \item (Coherence) for $\alpha < \lambda$, $i(\vec E \upharpoonright \kappa) \upharpoonright \alpha = \vec E \upharpoonright \alpha$ and $i(\vec E \upharpoonright \kappa)_\alpha = \emptyset$, where $i \colon J_\alpha^{\vec E} \rightarrow \Ult(J_\alpha^{\vec E}, E_\alpha)$ is the canonical embedding, and
          \item (Closure under initial segment) for any $\eta$ such that $(\kappa^+)^{J_\alpha^{\vec E}} \leq \eta < \nu^*(E_\alpha)$, $\eta = \nu^*(E_\alpha \upharpoonright \eta)$, and $E_\alpha \upharpoonright \eta$ is not of type Z, one of the following holds:
          \begin{enumerate}
              \item there is a $\gamma < \alpha$ such that $E_\gamma$ is the $(\kappa, \gamma)$-extender derived from $E_\alpha \upharpoonright \eta$, or
              \item $E_\eta \neq \emptyset$ and letting $j \colon J_\eta^{\vec E} \rightarrow \Ult(J_\eta^{\vec E}, E_\eta)$ be the canonical embedding with $\mu = \crit(j)$, there is a $\gamma < \alpha$ such that $j(\vec E \upharpoonright \mu)_\gamma$ is the $(\kappa, \gamma)$-extender derived from $E_\alpha \upharpoonright \eta$.
          \end{enumerate}
      \end{enumerate}
  \end{definition} 
  
  As the points where this indexing differs from the usual Mitchell-Steel indexing are determined by the critical points of the extenders in question, the usual fine structural properties can be shown via straightforward case distinctions. The reason why this indexing is more convenient for our construction than the usual Mitchell-Steel indexing becomes apparent in the definition of the operators, cf. Section \ref{subsec:operators}.
  We will work below a model of ``$\AD_\bR + \; \Theta$ is regular'' and refer to hybrid strategy mice as in \cite{Sa15}. 
  For similar results concering strategy mice and comparison of iteration strategies but for a different notion of hod pair, the interested reader may also consult \cite{St22,St23}.

\section{Getting a translatable structure}

In this section, we work in a model of ``$\AD_\bR + \text{ all sets of reals are uB}$'' in order to define a hybrid structure that we are going to translate into a model with a cardinal $\lambda$ that is both a limit of Woodin cardinals and a limit of strong cardinals in the rest of this paper. Recall that by a result of Martin and Woodin, cf. \cite[Theorem 9.1]{St09}, $\AD_\bR$ holds in every model of ``$\AD + \text{ all sets of reals are uB}$'' as clearly every universally Baire set of reals is Suslin. The arguments in this section will come from descriptive inner model theory and will look familiar to readers who have seen \emph{$\HOD$ computations} (cf., e.g. \cite{StW16, MS21, Sa15, Tr14}).

All definitions in Subsections \ref{subsec:hodpm}, \ref{subsec:suitablepm}, and \ref{subsec:propitstr} will make sense in models of $\ZFC$ as well as in models of $\AD^+$.\footnote{The ``$+$'' in $\AD^+$ will not be essential to follow the arguments in this article. So we just note that $\AD$ plus ``all sets of reals are Suslin'' (and hence obviously our hypothesis $\AD$ plus ``all sets of reals are universally Baire'') implies $\AD^+$ and refer the interested reader to \cite{La23, La20} for details on $\AD^+$.} The reason is that we will not put any restrictions on the iterability of the mice we include in lower part models, they will all be fully iterable throughout the ordinals. The notions introduced in the next subsections will be used later in this article in models of $\ZF + \AD^+$ and in strategy mice. At that point we will specify the model in which we are working, so we do not fix a model at this point.

\subsection{Hod premice}\label{subsec:hodpm}


We consider \emph{hod premice} as in \cite[Definition 1.34]{Sa15} and just briefly recall some important notions here. Let $\cP$ be a hod premouse. Then, as in \cite{Sa15} we let $\lambda^\cP$ be the order type of the Woodin cardinals of $\cP$ and their limits\footnote{In the notation of \cite{Sa15}, $\lambda^\cP = o.t.(Y^\cP)$.} and $(\delta_\alpha^\cP \mid \alpha \leq \lambda^\cP)$ be the sequence of Woodin cardinals of $\cP$ and their limits enumerated in increasing order. All $\delta_\alpha^\cP$ will be strong cutpoints\footnote{An ordinal $\gamma$ is called a strong cutpoint of a (hod) premouse $\cP$ if there is no extender on the sequence of $\cP$ such that $\crit(E) \leq \gamma < \lh(E)$.} of $\cP$ and our hod premice will always satisfy \[\lambda^\cP \leq \omega,\] so they will for example look like \cite[Figure 1.7.2]{Sa15} (an ordinary premouse, $\lambda^\cP = 0$), \cite[Figure 1.7.3]{Sa15} ($\lambda^\cP = 1$), or \cite[Figure 1.7.4]{Sa15} ($\lambda^\cP = \omega$). As in \cite[Definition 1.34]{Sa15}, we use
\begin{enumerate}
\item $\cP(i)$ to denote $\cP$ up to its $i$'s level\footnote{So, in our situation, $\cP(i) = \cP | (\delta_i^\cP)^{+\omega}$.} for $i \leq \lambda^\cP$ and
\item let $\cP(\lambda^\cP) = \cP$.
\end{enumerate}
Moreover, we let
\begin{enumerate}
\item $\Sigma^\cP_i$ denote the internal strategy of $\cP(i)$ for $i<\lambda^\cP$ and
\item if $\lambda^\cP = \omega$, we let $\Sigma^\cP = \Sigma^\cP_{{<}\omega} = \bigoplus_{i<\omega} \Sigma_i^\cP.$ 
\end{enumerate}
Whenever it does not lead to confusion, we drop the superscript $\cP$ and just write $\Sigma_i$ and $\Sigma$.
For a hod premouse $\cP$ such that $\lambda^\cP = i < \omega$, we write \[ \cP^- = \cP | (\delta_{i-1}^{+\omega})^\cP. \] If $\lambda^\cP = \omega$, we write \[ \cP^- = \cP|\delta^\cP_\omega. \]

We will consider strategies for hod premice that satisfy hull condensation (cf., \cite[Definition 1.31]{Sa15}).

\begin{definition}\label{def:hullcondensation}
  Let $\cP$ be a hod premouse and let $\Sigma$ be an $(\omega_1, \omega_1+1)$-iteration strategy for $\cP$. Then we say $\Sigma$ has \emph{hull condensation} if for any two stacks $\vec\cT$ and $\vec\cU$ on $\cP$ such that $\vec\cT$ is according to $\Sigma$ and $(\cP,\vec\cU)$ is a hull of $(\cP,\vec\cT)$ in the sense of \cite[Definition 1.30]{Sa15}, $\vec\cU$ is according to $\Sigma$ as well.
\end{definition}

Recall the definition of a hod pair (cf., \cite[Definition 1.36]{Sa15}).

\begin{definition}
  We say $(\cP, \Sigma)$ is a \emph{hod pair} iff
  \begin{enumerate}
  \item $(\cP, \Sigma)$ is an lsm pair (cf., \cite[Definition 1.23]{Sa15}),
  \item $\cP$ is a hod premouse, and
  \item $\Sigma$ is an $(\omega_1, \omega_1+1)$-iteration strategy for $\cP$ with hull condensation.
  \end{enumerate}
\end{definition}

As in \cite[Definition 2.26]{Sa15}, except that we require full iterability, we define the lower part closure relative to $\Sigma$ as follows. 
Note that this is denoted by the $\mathcal{W}$-operator in \cite{Sa17} and other articles.

\begin{definition}
  Let $\Sigma$ be an iteration strategy with hull condensation for some premouse $\cM_\Sigma$ and $a$ a transitive self well-ordered\footnote{Following \cite{SchlToperators}, we say a transitive set $a$ is \emph{self well-ordered} iff $a = b \cup \{b,<\}$ for some transitive set $b$ and a well-order $<$ of $b$. This induces a natural well-ordering $<_a \in L_1(a)$ on $a$ extending $<$.} set such that $\cM_\Sigma \in a$. Then we let
  \begin{enumerate}
  \item $Lp_0^{\Sigma,\infty}(a) = a$, 
  \item $Lp^{\Sigma,\infty}(a) = Lp_1^{\Sigma,\infty}(a) = \bigcup \{ \cM \mid \cM$ is an $\Ord$-iterable sound $\Sigma$-mouse over $a$ such that $\rho_\omega(\cM) = a \}$,
  \item $Lp_{\alpha+1}^{\Sigma,\infty}(a) = Lp^{\Sigma,\infty}(Lp_\alpha^{\Sigma,\infty}(a))$, and
  \item $Lp_\lambda^{\Sigma,\infty}(a) = \bigcup_{\alpha<\lambda} Lp_\alpha^{\Sigma,\infty}(a)$, for limit ordinals $\lambda$.
  \end{enumerate}
\end{definition}

The following is a special case of $Lp^{\Sigma,\infty}$ for ordinary mice. 

\begin{definition}
  Let $a$ be a transitive self well-ordered set. Then we let
  \begin{enumerate}
  \item $Lp^\infty_0(a) = a$,
  \item $Lp^\infty(a) = Lp^\infty_1(a) = \bigcup \{ \cM \mid \cM$ is an $\Ord$-iterable sound mouse over $a$ such that $\rho_\omega(\cM) = a \}$,
  \item $Lp^\infty_{\alpha+1}(a) = Lp^\infty(Lp^\infty_\alpha(a))$, and
  \item $Lp^\infty_\lambda(a) = \bigcup_{\alpha<\lambda} Lp^\infty_\alpha(a)$, for limit ordinals $\lambda$.
  \end{enumerate}
\end{definition}

We define fullness preservation analogous to \cite[Definition 2.27]{Sa15}.

\begin{definition}
  Let $(\cP,\Sigma)$ be a hod pair. Then $\Sigma$ is \emph{fullness preserving} if whenever $\vec\cT$ is a stack on $\cP$ via $\Sigma$ with last model $\cQ$ such that $\pi^{\vec\cT}$ exists, $i +1 \leq \lambda^\cQ$ and $\eta \geq \delta_i^\cQ$ is a strong cutpoint of $\cQ(i+1)$, then
  \[ \cQ | (\eta^+)^{\cQ(i+1)} = Lp^{\Sigma_{\cQ(i),\vec\cT}, \infty}(\cQ|\eta). \]
\end{definition}

Following \cite[Section 4.3]{Sa15}, we can now define direct limits of hod pairs as follows. For a hod pair $(\cP, \Sigma)$ such that $\Sigma$ has branch condensation (cf., \cite[Definition 2.14]{Sa15}) and is fullness preserving, let 
\begin{align*}
  pB(\cP,\Sigma) = \{ \cQ & \mid \exists \, \vec\cT \text{ stack on } \cP \text{ via } \Sigma \text{ with last model $\cR$ such that} \\ & \text{$\cP$-to-$\cR$ does not drop and $\cQ$ is a hod initial segment of $\cR$,} \\ & \text{i.e., $\cQ = \cR(\alpha)$ for some $\alpha < \lambda^\cR$} \}
\end{align*}
and
\begin{align*}
  pI(\cP, \Sigma) = \{ \cQ & \mid \exists \, \vec\cT \text{ stack on } \cP \text{ via } \Sigma \text{ with last model $\cQ$} \\ & \text{such that $\cP$-to-$\cQ$ does not drop} \}.
\end{align*}

We let $\cM_\infty(\cP,\Sigma)$ be the direct limit of all $\Sigma$-iterates of $\cP$ under the iteration maps. That means
\begin{align*}
   \cM_\infty(\cP, \Sigma) = \dirlim((\cQ, &\Sigma_{\cQ}), \pi^\Sigma_{\cQ,\cR} \mid \cQ, \cR \in pB(\cP, \Sigma) \; \wedge \\ & \exists \alpha \leq \lambda^\cR (\cR(\alpha) \in pI(\cQ,\Sigma_{\cQ})). 
\end{align*}

\subsection{Suitable premice}\label{subsec:suitablepm}

We define suitable premice and related concepts as in \cite[Section 4.1]{Sa15} but we will recall most notions for the reader's convenience. The reader can consult \cite{SaTr14,Tr14} for similar definitions. 

We first define suitability at $\kappa_0$. The role of $\kappa_0$ will become clear later in the definition of the translation procedure. For now we could also call this notion $0$-suitability\footnote{This is different from $0$-suitability as defined in \cite{MSW} or \cite{MS21}, so to avoid confusion we decided to call it $\kappa_0$-suitability here.} and it suffices to notice that these mice are relevant for the first level of our construction. Note that $\kappa_0$-suitable mice are ordinary and not hybrid.

\begin{definition}\label{def:kappa0suitable}
  We say a premouse $\cP$ is \emph{suitable at $\kappa_0$} (or $\kappa_0$-suitable) iff the following holds:
  \begin{enumerate}
  \item There is some ordinal $\delta$ such that \[ \cP \vDash \text{``}\delta \text{ is a Woodin cardinal''} \] and $\cP \cap \Ord = \sup_{n<\omega}(\delta^{+n})^\cP$.
  \item $\cP$ is a hod premouse with $\lambda^\cP = 0$, so $\cP$ is an ordinary premouse.
  \item For any $\cP$-cardinal $\eta$, if $\eta$ is a strong cutpoint, then \[ \cP | (\eta^+)^\cP = Lp^\infty(\cP|\eta). \]
  \end{enumerate}
\end{definition}

In general, our suitable premice are hybrid and defined as follows. Again, the role of $\kappa_i$ will become clear later and we could have called these premice $i$-suitable here.

\begin{definition}\label{def:kappaisuitable}
  Let $i > 0$. We say a hod premouse $\cP$ is \emph{$\Sigma$-suitable at $\kappa_i$} (or $(\kappa_i,\Sigma)$-suitable) iff the following holds:
  \begin{enumerate}
  \item There is some ordinal $\delta$ such that \[ \cP \vDash \text{``}\delta \text{ is a Woodin cardinal''} \] and $\cP \cap \Ord = \sup_{n<\omega}(\delta^{+n})^\cP$.
  \item $\cP$ is a hod premouse with $\lambda^\cP = i$.
  \item $(\cP^-,\Sigma)$ is a hod pair such that $\Sigma$ has branch condensation and is fullness preserving.
  \item $\cP$ is a $\Sigma$-mouse above $\cP(i-1)$.
  \item For any $\cP$-cardinal $\eta > \delta^\cP_{i-1}$, if $\eta$ is a strong cutpoint, then \[ \cP | (\eta^+)^\cP = Lp^{\Sigma,\infty}(\cP|\eta). \]
  \end{enumerate}
  If it is clear which strategy $\Sigma$ we consider, we will also just call $\cP$ suitable at $\kappa_i$ or $\kappa_i$-suitable. We denote the unique $\delta$ in Definitions \ref{def:kappa0suitable} and \ref{def:kappaisuitable} by $\delta^\cP$.
\end{definition}

To unify the notation we sometimes call a premouse $\Sigma$-suitable at $\kappa_0$ as well and mean suitable at $\kappa_0$. In this case $\Sigma$ could be empty or undefined.


We define $(\kappa_i, \Sigma)$-short tree, $(\kappa_i, \Sigma)$-maximal tree and $(\kappa_i, \Sigma)$-correctly guided finite (mixed) stack , the last model of a $(\kappa_i, \Sigma)$-correctly guided finite mixed stack, $S(\kappa_i, \Sigma)$, $F(\kappa_i, \Sigma)$, and $f$-iterability  as in \cite[Definitions 4.3 - 4.7]{Sa15}, adapted to our definition of suitability. We recall the concepts here for the reader's convenience. In what follows, let $\cP$ be a hod premouse that is $\Sigma$-suitable at some $\kappa_i$, $i<\omega$, for some fullness preserving iteration strategy $\Sigma$. Write 
\begin{eqnarray*}
    \mathcal{O}^{\cP}_\eta = \bigcup\{ \cN \unlhd \cP \mid \cP | \eta \unlhd \cN, \rho_\omega(\cN) \leq \eta, Y^\cN \subseteq \eta, \text{ and }
    \text{there is no} \\ \text{ extender } E \text{ on the sequence of } \cN \text{ with }\crit(E) \leq \eta < \lh(E) \}. 
\end{eqnarray*} 

Following the notation in \cite{Sa15}, we say an iteration tree $\cT$ on $\cP$ above $\cP^-$ is \emph{nice} if $\cT$ has no fatal drops, i.e., there is no $\alpha < \lh(\cT)$ and $\eta < \cM_\alpha^\cT \cap \Ord$ such that $\cT_{\geq \cM_\alpha^\cT}$ is a normal iteration tree on $\mathcal{O}_\eta^{\cM_\alpha^\cT}$ that is above $\eta$. Here $\cT_{\geq \cM_\alpha^\cT}$ denotes the part of $\cT$ starting with $\cM_\alpha^\cT$.

\begin{definition}
    Let $\cT$ be a nice iteration tree on the $(\kappa_i, \Sigma)$-suitable hod premouse $\cP$. Then we say
    \begin{enumerate}
        \item $\cT$ is $(\kappa_i, \Sigma)$-\emph{correctly guided} if for every limit ordinal $\lambda < \lh(\cT)$, $\cQ(\cT \upharpoonright \lambda)$ exists and $\cQ(\cT \upharpoonright \lambda) \unlhd Lp^{\Sigma, \infty}(\cM(\cT \upharpoonright \lambda)),$
        \item $\cT$ is $(\kappa_i, \Sigma)$-\emph{short} if $\cT$ is $(\kappa_i, \Sigma)$-correctly guided and either $\cT$ has a last model, or $\cT$ has limit length and $\delta(\cT)$ is not Woodin in $Lp^{\Sigma, \infty}(\cM(\cT)),$ and
        \item $\cT$ is $(\kappa_i, \Sigma)$-\emph{maximal} if $\cT$ is a nice $(\kappa_i, \Sigma)$-correctly guided tree and not $(\kappa_i, \Sigma)$-short.
    \end{enumerate}
\end{definition}

\begin{definition}
    A sequence $(\cT_j, \cP_j \mid j \leq m)$ for $m < \omega$ is called a $(\kappa_i, \Sigma)$-\emph{correctly guided finite stack on $\cP$} if $\cP_0 = \cP$ and
    \begin{enumerate}
        \item for every $j \leq m$, $\cP_i$ is $(\kappa_i, \Sigma)$-suitable and $\cT_j$ is a nice $(\kappa_i, \Sigma)$-correctly guided tree on $\cP_j$,
        \item for every $j < m$, either 
        \begin{enumerate}
            \item $\cT_j$ has last model $\cP_{j+1}$ and the iteration embedding $\pi^{\cT_j}$ exists, and if $\lh(\cT_j) = \lambda + 1$ for some limit ordinal $\lambda$, then $\cQ(\cT_j^-)$ exists and $\cQ(\cT_j^-) \unlhd \cM^{\cT_j}_\lambda$, or
            \item $\cT_j$ is $(\kappa_i, \Sigma)$-maximal and $\cP_{j+1} = Lp^{\Sigma, \infty}_{\omega}(\cM(\cT_j))$,
        \end{enumerate}
        where $\cT_j^-$ denotes the iteration tree $\cT_j$ without its last branch.
    \end{enumerate}
\end{definition}

\begin{definition}A sequence $(\cT_j^k, \cP_j^k, \vec\cT^k \mid k \leq m, j \leq n_k)$ for $m < \omega$ and $n_k<\omega$ for all $k \leq m$ is called a $(\kappa_i, \Sigma)$-\emph{correctly guided finite mixed stack on $\cP$} if $\cP_0^0 = \cP$ and
\begin{enumerate}
    \item for every $k \leq m$, $\cP_0^k$ is $\Sigma_{(\cP_0^k)^-}$-suitable at $\kappa_i$ and $(\cT_j^k, \cP_j^k \mid j \leq n_k)$ is a $(\kappa_i, \Sigma_{(\cP_0^k)^-})$-correctly guided finite stack on $\cP_0^k$,
    \item for every $k < m$, $\vec\cT^k$ is a stack on $(\cP_{n_k}^k)^-$ via $\Sigma_{(\cP_{n_k}^k)^-}$ such that $\cP_0^{k+1}$ is the last model of $\vec\cT^k$ when it is regarded as a stack on $\cP_{n_k}^k$ and $(\cP_{n_k}^k)^-$-to-$(\cP_0^{k+1})^-$ does not drop, and
    \item $\vec\cT_m$ is a stack on $(\cP_{n_m}^m)^-$ via $\Sigma_{(\cP_{n_m}^m)^-}$.
\end{enumerate}
\end{definition}

\begin{definition}
Suppose $(\cT_j^k, \cP_j^k, \vec\cT^k \mid k \leq m, j \leq n_k)$ for $m < \omega$ and $n_k<\omega$ for all $k \leq m$ is a $(\kappa_i, \Sigma)$-correctly guided finite mixed stack on $\cP$. Then we say $\cR$ is the \emph{last model of $(\cT_j^k, \cP_j^k, \vec\cT^k \mid k \leq m, j \leq n_k)$} if 
\begin{enumerate}
    \item $\vec\cT_m$ is non-empty and has a last model, which is $\cR$, when $\vec\cT_m$ is regarded as a stack on $\cP_{n_m}^m$,
    \item $\vec\cT_m$ is empty and $\cT^m_{n_m}$ has a last model which is $\cR$,
    \item $\vec\cT_m$ is empty, $\cT^m_{n_m}$ has limit length and is $(\kappa_i, \Sigma_{(\cP_0^m)^-})$-short, and there is a cofinal well-founded branch $b$ of $\cT^m_{n_m}$ such that $\cQ(b, \cT^m_{n_m})$ exists, \[ \cQ(b, \cT^m_{n_m}) \unlhd Lp^{\Sigma_{(\cP_0^m)^-}, \infty}(\cM(\cT^m_{n_m})) \] and $\cR = \cM_b^{\cT^m_{n_m}}$, or
    \item $\vec\cT_m$ is empty, $\cT^m_{n_m}$ has limit length and is $(\kappa_i, \Sigma_{(\cP_0^m)^-})$-maximal, $\cR$ is $(\kappa_i, \Sigma_{(\cP_0^m)^-})$-suitable and \[ \cR = Lp_\omega^{\Sigma_{(\cP_0^m)^-}, \infty}(\cM(\cT^m_{n_m})). \]
\end{enumerate}
In this case, we say $\cR$ is a \emph{$(\kappa_i, \Sigma)$-correct iterate of $\cP$}.
\end{definition}

Before we define $f$-iterability, we introduce some more notation from \cite{Sa15}. Let
\[ S(\kappa_i, \Sigma) = \{ \cQ \mid \cQ^- \in pI(\cP, \Sigma) \text{ and } \cQ \text{ with strategy } \Sigma_{\cQ^-} \text{ is } (\kappa_i, \Sigma)\text{-suitable} \} \]
and 
\begin{eqnarray*}
    F(\kappa_i, \Sigma) = \{ f \colon S(\kappa_i, \Sigma) \rightarrow \Pot(S(\kappa_i, \Sigma)) \mid f \text{ function such that for all } \cQ \in S(\kappa_i, \Sigma), \\ f(\cQ) \subseteq \cQ \text{ and for every } X \in \cQ, X \cap f(\cQ) \in \cQ \}. \;\;\;\;\;\;\;\;\;
\end{eqnarray*}
Moreover, for $\cQ \in S(\kappa_i, \Sigma)$ and $f \in F(\kappa_i, \Sigma)$ let, for $n<\omega$, \[ f^n(\cQ) = f(\cQ) \cap \cQ|((\delta^\cQ)^{+n})^\cQ. \] Then $f(\cQ) = \bigcup_{n<\omega} f^n(\cQ)$ and we can define $\pi(f(\cQ)) = \bigcup_{n<\omega} \pi(f^n(\cQ))$ for any elementary embedding $\pi \colon \cQ \rightarrow \cR$. Let
\[ \gamma_f^\cQ = \sup(Hull_1^\cQ(\cQ^- \cup \{ f^n(\cQ) \mid n < \omega \}) \cap \delta^\cQ), \] 
where $Hull_1^\cQ$ denotes an uncollapsed $\Sigma_1$-hull in $\cQ$, and note that
\[ \gamma_f^\cQ = \sup(Hull_1^\cQ(\gamma_f^\cQ \cup \{ f^n(\cQ) \mid n < \omega \}) \cap \delta^\cQ). \] 
Moreover, let
\[ H_f^\cQ = Hull_1^\cQ(\gamma_f^\cQ \cup \{ f^n(\cQ) \mid n < \omega \}). \]
Now, we can define $f$-iterability.

\begin{definition}\label{def:fiterability}
    For $\cQ \in S(\kappa_i, \Sigma)$ and $f \in F(\kappa_i, \Sigma)$ we say that $\cQ$ is \emph{$f$-iterable} if for every $(\kappa_i, \Sigma)$-correctly guided finite mixed stack $(\cT_j^k, \cQ_j^k, \vec\cT^k \mid k \leq m, j \leq n_k)$ for $m < \omega$ and $n_k<\omega$ for all $k \leq m$ on $\cQ$ with last model $\cR$ there is a sequence \[ (b_j^k \mid k \leq m, j \leq n_k) \] with the following properties.
    \begin{enumerate}
        \item For $k \leq m-1$ and $j \leq n_k$ as well as for $k = m$ and $j \leq n_m - 1$, $b_j^k = \emptyset$, if $\cT_j^k$ has successor length, and $b_j^k$ is a cofinal well-founded branch through $\cT_j^k$ with $\cM_{b_j^k}^{\cT_j^k} = \cQ_j^k$, if $\cT_j^k$ is a $(\kappa_i, \Sigma_{(\cP_0^k)^-})$-maximal tree.
        \item For the last tree:
        \begin{enumerate}
            \item If $\cT_{n_m}^m$ has successor length, then $b_{n_m}^m = \emptyset$.
            \item If $\cT_{n_m}^m$ is a $(\kappa_i, \Sigma_{(\cP_0^m)^-})$-short tree, then $b_{n_m}^m$ is the unique cofinal well-founded branch through $\cT_{n_m}^m$ such that $\cQ(b_{n_m}^m, \cT_{n_m}^m)$ exists and $\cQ(b_{n_m}^m, \cT_{n_m}^m) \unlhd Lp_\omega^{\Sigma_{(\cP_0^m)^-}, \infty}(\cM(\cT^m_{n_m}))$.
            \item If $\cT_{n_m}^m$ is a $(\kappa_i, \Sigma_{(\cP_0^m)^-})$-maximal tree, then $b_{n_m}^m$ is a cofinal well-founded branch.
        \end{enumerate}
        \item For $k \leq m$ and $j \leq n_k$, let 
        \[ \pi_j^k = \begin{cases}
            \pi^{T_j^k} & \text{ if } T_j^k \text{ has successor length, and}\\
            \pi^{T_j^k}_{b_j^k} & \text{ if } T_j^k \text{ has limit length.} 
        \end{cases} \]
        Moreover, let \[ \pi_k = \pi^{\vec\cT^k} \circ \pi^k_{n_k} \circ \pi^k_{n_k -1} \circ \dots \circ \pi^k_0 \] and
        \[ \pi = \pi_{m} \circ \pi_{m-1} \circ \dots \circ \pi_0. \]
        Then $\pi(f(\cQ)) = f(\cR).$ \label{eq:fiterabilityclause3}
    \end{enumerate}
\end{definition}

In the setting of Definition \ref{def:fiterability} we say that the sequence $\vec b = (b_j^k \mid k \leq m, j \leq n_k)$ witnesses $f$-iterability for $\vec\cT = (\cT_j^k, \cQ_j^k, \vec\cT^k \mid k \leq m, j \leq n_k)$ and write $\pi_{\vec\cT, \vec b}$ for the embedding $\pi$ in Definition \ref{def:fiterability}\eqref{eq:fiterabilityclause3}.

As usual, $f$-iterability embeddings for a given $\vec\cT$ are independent of the choice of $\vec b$ (see \cite[Lemma 4.8]{Sa15}). Moreover, we can define a strengthening, called \emph{strong $f$-iterability}, where this uniqueness also does not depend on the choice of the stack $\vec\cT$ (see \cite[Definition 4.10]{Sa15}).

In the previous definitions we omit the reference to $\kappa_i$ and $\Sigma$ if $i = 0$ and $\Sigma = \emptyset$ or it is clear from the context to which $i < \omega$ and strategy $\Sigma$ we are referring. 
As in \cite[Section 4.2]{Sa15} we can define when a $(\kappa_i, \Sigma)$-suitable premouse is $A$-iterable\footnote{Sargsyan calls this $B$-iterability.} for an $\OD_\Sigma$ set of reals $A$:

Recall that for a $(\kappa_i, \Sigma)$-suitable premouse $\cP$, some uncountable $\cP$-cardinal $\gamma$, and a set of reals $A$, a term $\tau \in \cP^{\Col(\omega,\gamma)}$ for a set of reals \emph{locally term captures $A$ at $\gamma$} if whenever $g$ is $\Col(\omega, \gamma)$-generic over $\cP$, \[ \tau_g = A \cap \cP[g]. \]

\begin{definition}\label{def:Aiterable}
    Suppose $A$ is an $\OD_\Sigma$ set of reals that is locally term captured for comeager many set generics over a $(\kappa_i, \Sigma)$-suitable premouse $\cP$. Let $\tau_{A,\gamma}^{\cP, \Sigma}$ be the term in $\cP$ capturing $A$ at $\gamma$ for comeager many set generics and let $f_A \in F(\kappa_i,\Sigma)$ be the function given by \[ f_A(\cQ) = \bigoplus_{\gamma < \cQ \cap \Ord} \tau_{A,\gamma}^{\cQ, \Sigma_{\cQ^-}}, \] for all $\cQ \in S(\kappa_i, \Sigma)$. Then $\cP$ is \emph{$A$-iterable} iff it is $f_A$-iterable.
\end{definition}

In the setting of Definition \ref{def:Aiterable}, we write \[ \gamma_A^{\cQ, \Sigma_{\cQ^-}} = \gamma^{\cQ}_{f_A} \] and \[ H_A^{\cQ, \Sigma_{\cQ^-}} = H^{\cQ}_{f_A} \] for any $\cQ \in S(\kappa_i, \Sigma)$.
For a $(\kappa_i, \Sigma)$-correctly guided finite mixed stack $\vec\cT$ we say that $\vec b$ is an \emph{$A$-correct branch} iff $\vec b$ witnesses $f_A$-iterability for $\cT$.


\subsection{Properties of iteration strategies}\label{subsec:propitstr}

We will need the following standard properties of iteration strategies. The first one is super fullness preservation. This notion goes back to \cite[Definition 3.33]{Sa15} and is also used in \cite[Definition 2.3]{Sa17}. We simply adapt the definition to our context. Recall that we will specify the models in which we apply the definitions from this section later. In this section we sometimes explicitly mention a model $M$ when we discuss generic closure of this model under a strategy and related notions. This will later be applied to models $M$ that are strategy mice.

\begin{definition}
  Let $(\cP,\Sigma)$ be a hod pair such that $\Sigma$ has branch condensation and is fullness preserving. Then, for some $\eta$, we say $\Sigma$ is \emph{$\eta$-super fullness preserving} iff,
  letting $\delta = \delta^\cP$ and $\delta^n = (\delta^{+n})^\cP$, for every $n<\omega$, there is a term $\tau \in \cP^{\Col(\omega, \delta^n)}$ such that
  \begin{enumerate}
      \item $\cP \Vdash_{\Col(\omega, \delta^n)} \tau \subseteq \bR^2$,
      \item for every $\Col(\omega, \delta^n)$-generic $g$ over $\cP$,
      \begin{align*}
      \tau_g \subseteq \{ (x,y) \in \bR^2 \mid y &\text{ codes a sound $\Sigma$-premouse $\cN$ over $x$ such that } \\ & \cN \unlhd Lp^{\Sigma,\infty}(x) \text{ and } \rho_\omega(\cN) = \omega \},
      \end{align*}
      \item for every iteration $j \colon \cP \rightarrow \cQ$ according to $\Sigma \upharpoonright V_\eta$, every self well-ordered set $X \in V_\eta$, and every $\Col(\omega, j(\delta^n))$-generic $g$ over $\cQ$ such that $X \in \HC^{\cQ[g]}$, letting $x \in \bR^{\cQ[g]}$ be generic over $L_\omega[X]$ and coding $X$,\footnote{For the definition of $\cP$-constructions for hod premice see $\cS$-constructions in \cite[Section 3.8]{Sa15}. As such constructions are due to Steel, they are called $\cS$-constructions in \cite{Sa15}. We nevertheless decided to use the more common terminology $\cP$-construction here.} 
      \begin{align*}
      Lp^{\Sigma,\infty}(X) = \bigcup&\{ \cM \mid \cM \text{ is a $\Sigma$-premouse over $X$, there exists a real } y \text{ with } \\ 
      & (x,y) \in (j(\tau))_g \text{ and } \text{$y$ codes the result of a $\cP$-construction} \\
      & \text{relative to $\Sigma$ translating $\cM$ into a $\Sigma$-premouse above $x$} \}.
      \end{align*}
  \end{enumerate}
  Moreover, we say $\Sigma$ is \emph{super fullness preserving} iff it is $\eta$-super fullness preserving for every $\eta$.
\end{definition}

Now, following \cite[Definition 2.4]{Sa17}, we can define that a model $M$ is generically closed under $\Sigma$ if $\Sigma$ is super fullness preserving in all generic extensions of $M$. Formally, this is defined as follows.

\begin{definition}
  Let $\Sigma$ be a super fullness preserving strategy in some model $M$. Then $M$ is \emph{generically closed under $\Sigma$} iff for all $M$-cardinals $\eta$ and all ${\leq}\eta$-generics $g$ over $M$, there is a unique $((\eta^+)^M, (\eta^+)^M)$-strategy $\Sigma^{g,\eta} \in M[g]$ such that $\Sigma^{g,\eta} \upharpoonright M$ is the $((\eta^+)^M, (\eta^+)^M)$-fragment of $\Sigma$ and \[ M[g] \vDash \text{``}\Sigma^{g,\eta} \text{ is $\eta$-super fullness preserving''}. \]
In this case we write $\Sigma^g = \bigcup_{\eta \in \Ord} \Sigma^{g,\eta}$. 
\end{definition}

\begin{definition}
  Let $(\cP,\Sigma)$ be a hod pair and suppose $\Sigma$ is super fullness preserving in some model $M$ that is generically closed under $\Sigma$. Let $i<\omega$ and let $\tau$ be a $\Col(\omega, \cP(i))$-term relation for a set of reals. Then \emph{$\Sigma$ guides $\tau$ correctly} iff for any $g$ generic over $M$ and any $\Sigma^g$-iterates $\cQ$ and $\cR$ of $\cP$ such that $\cP$-to-$\cQ$ and $\cP$-to-$\cR$ do not drop, whenever $x \in \bR^{M[g]}$, $g_\cQ$ is $\Col(\omega, \cQ(i))$-generic over $\cQ$, $g_\cR$ is $\Col(\omega, \cR(i))$-generic over $\cR$, and $x \in \cQ[g_\cQ] \cap \cR[g_\cR]$, then \[ x \in i(\tau)_{g_\cQ} \text{ iff } x \in j(\tau)_{g_\cR}, \] where $i \colon \cP \rightarrow \cQ$ and $j \colon \cP \rightarrow \cR$ denote the iteration embeddings.

  In this case, we write $A^g_{\Sigma,\tau}$ for the set of reals in $M[g]$ determined by $\Sigma^g$ and $\tau$.
\end{definition}

Now, the following definition adapts \cite[Definition 2.6]{Sa17} to our context (cf., \cite[Definition 4.14]{Sa15}). 

\begin{definition}
  Let $\cP$ be $\Sigma$-suitable at $\kappa_i$ for some $i<\omega$ and suppose $(\cP,\Lambda)$ is a hod pair such that $\Lambda$ is super fullness preserving in some model $M$ that is generically closed under $\Lambda$. Let $\vec\tau = (\tau_k \mid k<\omega)$ be a sequence of terms in $M$. Then we say \emph{$\Lambda$ is guided by $\vec\tau$} iff for any $g$ generic over $M$ the following holds. 
  \begin{enumerate}
  \item For every $k<\omega$ there is some $n<\omega$ such that $\tau_k$ is a $\Col(\omega, ((\delta^\cP)^{+n})^\cP)$-term relation for a set of reals and $\Lambda^g$ guides $\tau_k$ correctly.
  \item For each $\Lambda^g$-iterate $\cQ$ of $\cP$ such that $\cP$-to-$\cQ$ does not drop with iteration embedding $\pi \colon \cP \rightarrow \cQ$, \[ \cQ = \bigcup_{k<\omega} H_{B_k}^{\cQ, \Sigma_{\cQ^-}}, \] for $B_k = A^g_{\Lambda_\cQ, \pi(\tau_k)}$.
  \item For each $\Lambda^g$-iterate $\cQ$ of $\cP$ such that $\cP$-to-$\cQ$ does not drop with iteration embedding $\pi \colon \cP \rightarrow \cQ$ and each normal tree $\cT$ in the domain of $\Lambda^g_\cQ$, letting $b = \Lambda^g_\cQ(\cT)$, if $\pi_b^\cT$ exists, then $b$ is the unique branch through $\cT$ such that for every $k<\omega$, \[ \pi_b^\cT(\pi(\tau_k)) \text{ locally term captures } A^g_{\Lambda_\cQ,\pi(\tau_k)}. \]
  \item For each $\eta$ that is a limit of Woodin cardinals in $M[g]$ such that $\cP \in V_\eta^{M[g]}$, 
  if $h$ is $\Col(\omega, {<}\eta)$-generic over $M[g]$, then for every $k<\omega$, $A_{\Lambda,\tau_k}^{g*h}$ is an $\OD_\Sigma$ set of reals in the derived model of $M[g]$ as computed by $h$.
  \end{enumerate}
Moreover, we say \emph{$\Lambda$ is strongly guided by the sequence $\vec\tau = (\tau_k \mid k<\omega)$} iff for any $\Lambda^g$-iterate $\cQ$ of $\cP$ such that $\cP$-to-$\cQ$ does not drop with iteration embedding $\pi \colon \cP \rightarrow \cQ$ and any premouse $\cR$ such that there are embeddings $\sigma \colon \cP \rightarrow \cR$ and $l \colon \cR \rightarrow \cQ$ with $\pi = l \circ \sigma$, it follows that $\cR$ is $\Sigma$-suitable at $\kappa_i$ and for every $k<\omega$, $l^{-1}(\pi(\tau_k))$ locally term captures $A^g_{\Lambda,\tau_k}$.
\end{definition}


\subsection{Translatable structures}

Work in a model of \[ \AD_\bR + \text{ all sets of reals are universally Baire.} \]
Recall that we may and will assume\footnote{A model of ``$\AD_\bR + \; \Theta$ is regular'' implies the existence of a model of $\ZFC$ with a cardinal $\delta$ that is an inaccessible limit of Woodin cardinals and ${<}\delta$-strong cardinals. This is more than what we are aiming to construct in Theorem \ref{thm:main}. See \cite{Sa13, Zh15}.} that we work below a model of ``$\AD_\bR + \; \Theta$ is regular'' and hence the techniques from \cite{Sa15} apply. 

Suppose first that $\Theta > \theta_\omega$. Then, by \cite[Lemma 0.23]{Sa15}, there is a hod pair $(\cP, \Sigma_\cP)$ such that $\Sigma_\cP$ has branch condensation,\footnote{For branch condensation in the sense of \cite[Definition 2.14]{Sa15} it is essential that we are working with hod mice under the minimality assumption of \cite{Sa15}. This suffices for our aim but we recommend that the reader interested in a more general setting consults \cite{St22}.} is fullness-preserving and $V_{\theta_\omega}^{\HOD}$ is a shortening of the direct limit of all $\Sigma_\cP$-iterates of $\cP$, i.e., $\cM_\infty(\cP,\Sigma_\cP)|\theta_\omega = V_{\theta_\omega}^{\HOD}.$ In this case, we let \[ \cM_\infty = \cM_\infty(\cP,\Sigma_\cP)|\theta_\omega \] and let $\Sigma$ be the corresponding tail strategy of $\Sigma_\cP$.

In case $\Theta = \theta_\omega$, let
\begin{align*}
  \cM_\infty = \bigcup\{\cM_\infty&(\cP,\Sigma_\cP)|\theta_i \mid (\cP,\Sigma_\cP) \text{ is a countable hod pair such }  \\ & \text{that } \Sigma_\cP \text{ has branch condensation and is fullness} \\ & \;\;\;\text{preserving and } w(Code(\Sigma_\cP)) = \theta_i \}. 
\end{align*}
Here $w(Code(\Sigma_\cP))$ denotes the Wadge-rank of the canonical code of $\Sigma_\cP$ as a set of reals. Note that $\cM_\infty$ is well-defined as the direct limit models $\cM_\infty(\cP, \Sigma_\cP)$ line up and do not depend on the choice of $(\cP, \Sigma_\cP)$ (cf. \cite[Lemma 0.23 and Theorem 4.24]{Sa15}). In fact, by the results of \cite{Sa15}, $V_\Theta^{\HOD} = \cM_\infty$ and there is a canonical iteration strategy, which we call $\Sigma$, for $\cM_\infty$ given by the join of the strategies for its hod initial segments (see also \cite[Section 3]{Tr14} on how $\cM_\infty$ can be represented as a direct limit of a directed system of hod pairs). 

As usual, each $\theta_i$ in our background universe is a Woodin cardinal in $\cM_\infty$ and we write $(\delta_i \mid i<\omega)$ for the sequence of Woodin cardinals in $\cM_\infty$.

Note that the Woodin cardinals $\delta_i$, $i<\omega$, in $\cM_\infty$ are cutpoints, so the strategy $\Sigma$ of $\cM_\infty$ naturally splits into strategies $\Sigma_i$ between consecutive Woodin cardinals $\delta_{i-1}$ and $\delta_i$, letting $\delta_{-1} = 0$. In particular, $\cM_\infty$ can see its own iteration strategy.

Universal Baireness ensures that we can canonically extend each $\Sigma_i$ to an iteration strategy acting on stacks of normal iteration trees of length $\lambda$ for any ordinal $\lambda$. See \cite[Lemma 1.21]{Sa13} for this argument and in particular for how iteration strategies need to be coded as sets of reals to make the coding sufficiently absolute. We also write $\Sigma_i$ for this canonical extension of the original strategy $\Sigma_i$ of $\cM_\infty$ between $\delta_{i-1}$ and $\delta_i$ as well as $\Sigma$ for the resulting extension of the original strategy $\Sigma$ of $\cM_\infty$. 

The following lemma now easily follows from the fact that we are working in a model of ``all sets of reals are universally Baire''.

\begin{lemma}\label{lem:OrdOrdStrategyMinfty}
  $\Sigma$ is an $(\Ord, \Ord)$-iteration strategy for $\cM_\infty$.
\end{lemma}

Moreover, by our choice of the strategies $\Sigma_i$ (in particular, by the fact that they have branch condensation in the sense of \cite[Definition 2.14]{Sa15}) and generic interpretability for hod mice (see \cite[Theorem 3.10]{Sa15}) we have the following lemma.

\begin{lemma}\label{lem:MinftyclosedunderLambda}
  $\cM_\infty$ is generically closed under $\Sigma$.
\end{lemma}



Now we are ready to define translatable structures. This is crucial for the rest of this article as a translatable structure is what we want to translate into a model with a cardinal that is both a limit of Woodin cardinals and a limit of strong cardinals.

\begin{definition}\label{def:translatablestructure}
  Let $\cW$ be a proper class hybrid strategy premouse in the sense of \cite[Chapter 1]{Sa15}. Then $\cW$ is a \emph{translatable structure} iff there is a sequence of ordinals $(\delta_i \mid i<\omega)$ such that $\cW$ satisfies the following conditions.
  \begin{enumerate}
  \item $\cW \vDash \text{``} \delta_i$ is a Woodin cardinal and a cutpoint for every $i<\omega$ and these are the only Woodin cardinals''.
  \item In $\cW$, let $\cP^0 = Lp_\omega^\infty(\cW|\delta_0)$ and $\cP^i = Lp_\omega^{\Sigma_{i-1}, \infty}(\cW|\delta_{i})$ for $i > 0$, where $\Sigma_i$ denotes the strategy for $\cP^i$. Then $\cP^0$ is suitable at $\kappa_0$ and $\cP^i$ is $\Sigma_{i-1}$-suitable at $\kappa_i$ for every $i > 0$.
  \item $\Sigma_i$ is a super fullness preserving iteration strategy for $\cP^i$ with hull condensation such that $\cW$ is generically closed under $\Sigma_i$ and there is a set of term relations $\vec\tau^{(i)} \subseteq \cP^i$ such that whenever $g$ is $\Col(\omega, \cP^i)$-generic over $\cW$ and $(\tau_k^{(i)} \mid k<\omega)$ is a generic enumeration of $\vec\tau^{(i)}$ in $\cW[g]$, then $(\Sigma_i)^g$ is strongly guided by $(\tau_k^{(i)} \mid k < \omega)$.
  \item $\cW$ is internally $(\Ord, \Ord)$-iterable.
  \item \label{cl:5_def:translatablestructure} $\cW \vDash$ ``there is no inner model with a superstrong cardinal''.\footnote{In this article, we are aiming for the consistency of a theory weaker than the existence of a superstrong cardinal, so we will never consider any inner models not satisfying \eqref{cl:5_def:translatablestructure} and this condition is redundant here.}
  \item Let $\delta_\omega $ be the limit of $\delta_i$, $i<\omega$. Let $G$ be $\Col(\omega,{<}\delta_\omega)$-generic over $\cW$ and let $M$ be the derived model computed in $\cW(\bR^*)$. Let $\Phi = (\Sigma^\cW)^G \upharpoonright \HC^M$ for $\Sigma^\cW = \bigoplus_{i<\omega} \Sigma_i$. Then $\Phi \in M$ and in $M$, $Lp_\omega^\infty(\cW | \delta_i)$ is a $\kappa_i$-suitable $\emptyset$-iterable $\Phi$-premouse (cf. \cite[Definition 2.7]{Sa17} for the definition of $\emptyset$-iterability). \label{eq:translatablestructure6}
  \end{enumerate}
\end{definition}

\begin{remark}
We analogously define set-sized translatable structures $\cW$ that satisfy a sufficiently large fragment of $\ZFC$.
\end{remark}

Recall that $(\cM_\infty, \Sigma)$ is a hod pair. It follows from the fullness of $\cM_\infty$ with respect to $\Sigma$ that constructing relative to the strategy $\Sigma$ above $\cM_\infty$ does not project across $\cM_\infty \cap \Ord$.

\begin{lemma}\label{lem:ConstructionDoesNotProjectAcrossMinfty}
   The construction of $L^\Sigma[\cM_\infty]$ above $\cM_\infty$ does not project across $\cM_\infty \cap \Ord$.
\end{lemma}
\begin{proof}
    Suppose otherwise, i.e., there is a $\xi \geq 1$ such that $J_\xi^\Sigma[\cM_\infty]$ projects across $\cM_\infty \cap \Ord$. Then $J_\xi^\Sigma[\cM_\infty]$ is an anomalous hod premouse of type III (see Definition 3.23 in \cite{Sa15}) and using fullness preservation of $\Sigma$ we can follow the proof of Theorem 6.1 in \cite{Sa15} (see also the argument for Claim 2 in the proof of Lemma 6.23 in \cite{Sa15}) to get a contradiction.
\end{proof}

 Lemmas \ref{lem:OrdOrdStrategyMinfty}, \ref{lem:MinftyclosedunderLambda}, and \ref{lem:ConstructionDoesNotProjectAcrossMinfty}, together with general properties of hod mice from \cite{Sa15}, imply that $L^\Sigma[\cM_\infty]$ is a translatable structure.

\begin{corollary}
Working in a model of ``$\AD_\bR \, +$ all sets of reals are universally Baire'' and defining $\cM_\infty$ and $\Sigma$ as above, $L^\Sigma[\cM_\infty]$ is a translatable structure.
\end{corollary}


\section{The translation procedure}

For the rest of this article, fix a translatable structure $\cW$ and let $\Sigma = \bigoplus_{i<\omega} \Sigma_i$ be the associated iteration strategy. The extenders we want to use to witness that certain cardinals are strong in the translated structure are of the following form:

\begin{definition}\label{def:genctblecompleteness}
  We say a $(\kappa, \lambda)$-extender $E$ is \emph{generically countably complete} if it is countably complete in all ${<}\kappa$-generic extensions, i.e., whenever $\bP$ is a partial order with $|\bP| < \kappa$ and $G$ is $\bP$-generic over $V$, then in $V[G]$, for every sequence $(a_i \mid i < \omega)$ of sets in $[\lambda]^{<\omega}$ and every sequence $(A_i \mid i < \omega)$ of sets $A_i \in E_{a_i}$ there is a function $\tau \colon \bigcup_i a_i \rightarrow \kappa$ such that $\tau \pwimg a_i \in A_i$ for each $i < \omega$.
\end{definition}


It is a well-known fact for countably complete extenders that they provide realizations, see for example \cite[Lemma 10.64]{Sch14}. We will crucially use (a generalization of) this in the proof of Lemma \ref{lem:iterability} to obtain iterability for our translated structures.

\begin{fact}\label{fact:realization}
  Let $\cM$ be a premouse and $\bar\cM$ the transitive collapse of a countable elementary substructure of $\cM$. Let $\sigma \colon \bar\cM \rightarrow \cM$ be the uncollapse map. Then for every extender $E \in \bar\cM$ such that $\sigma(E)$ is countably complete, if $\pi_E \colon \bar\cM \rightarrow \Ult(\bar\cM, E)$ denotes the ultrapower embedding, there is a realization map $\sigma^* \colon \Ult(\bar\cM,E) \rightarrow \cM$, cf. Figure \ref{fig:RealizationDiagram}.
\end{fact}

\begin{figure}[htb]
      \begin{tikzpicture}
        \draw[->] (0,-1.55) -- node[above] {$\sigma$} (3.75,0) node[right]
        {$\cM$}; 

        \draw[->] (0,-1.75) node[left] {$\bar\cM$}-- node[below] {$\pi_E$} (3.3,-1.75) node[right] {$\Ult(\bar\cM,E)$};
       
        \draw[->] (4.1,-1.4) -- node[right] {$\sigma^*$} (4.1,-0.3);
      \end{tikzpicture}
      \caption{Realizing a countably complete extender.}\label{fig:RealizationDiagram}
    \end{figure}

Now we can define our translation procedure. We aim to construct a premouse with strong cardinals $\kappa_i$, $i < \omega$, inside $\cW$. The cardinals $\kappa_i$ are going to be the least ${<}\delta_{i+1}$-strong cardinals in $\cW$. The idea is that extenders with critical point $\kappa \neq \kappa_i$ for all $i < \omega$ come from a standard fully backgrounded extender construction as in \cite{MS94} (without the smallness assumption). In addition, we add extra extenders witnessing that $\kappa_i$ is strong for each $i<\omega$. These extenders with critical point $\kappa_i$ are still fully backgrounded in case they are added in the construction below the next Woodin cardinal $\delta_{i+1}$. These extenders will not overlap $\delta_{i+1}$. At levels of the construction past $\delta_{i+1}$ all extenders that are added with critical point $\kappa_i$ are generically countably complete (and these extenders will overlap $\delta_{i+1}$). This will later allow us to prove that certain elementary substructures of the translated premouse are in fact sufficiently iterable (cf. Section \ref{subsec:iterability}) to show that the construction converges and the cardinals $\kappa_i$ are indeed fully strong\footnote{``Fully strong'' simply means ``strong'' but we decided to highlight this here as well as at some places below in order to avoid confusion with the statement that $\kappa_i$ is ${<}\delta_{i+1}$-strong.} (cf. Section \ref{subsec:constructionworks}). 


\begin{definition}\label{def:translationprocedure}
  Let $\cW$ be a translatable structure with Woodin cardinals $(\delta_i \mid i < \omega)$. Let $(\kappa_i \mid i < \omega)$ be a sequence of cardinals such that for each $i < \omega$, $\kappa_i$ is the least ${<}\delta_{i+1}$-strong cardinal above $\delta_i$ in $\cW$. We say $(\cM_\xi, \cN_\xi \mid \xi \leq \Xi)$ is the \emph{mouse construction} (or short \emph{$\cM$-construction}) in $\cW$ for $\Xi \leq \Ord$ if it is given inductively as follows.
  \begin{enumerate}
  \item $\cN_0$ is the result of a fully backgrounded extender construction in $\cW | \delta_1$.
  \item For $\xi+1 < \Xi$, if $\cM_\xi = (J_\alpha^{\vec E}, \in, \vec E, \emptyset)$ is a passive premouse, we define $\cN_{\xi+1}$ as follows. 
    \begin{enumerate}
    \item \label{item:backgroundext_critnotkappai} If there is an extender $F^*$ on the sequence of $\cW$ and an extender $F$ over $\cM_\xi$, both with critical point $\kappa \notin \{\kappa_i \mid i < \omega\}$,\footnote{In this case, we just follow the standard fully backgrounded construction in \cite{MS94} without the smallness assumption.} such that $(J_\alpha^{\vec E}, \in, \vec E, F)$ is a premouse and for some $\nu < \alpha$, $V_{\nu+\omega}^\cW \subseteq \Ult(\cW,F^*)$ and \[ F \upharpoonright \nu = F^* \cap ([\nu]^{<\omega} \times J_\alpha^{\vec E}), \] we let\footnote{In \cite{MS94}, the authors choose $F^*$ and $F$ as above with the minimal $\nu$ to get uniqueness of the next extender in the construction. This condition is omitted in the constructions in \cite{Sa15} and \cite{Sa17}. Farmer Schlutzenberg showed that this condition is not needed for uniqueness of the next extender.} \[\cN_{\xi+1} = (J_{\alpha}^{\vec E}, \in, \vec E, F).\] 
    \item \label{item:backgroundext_critkappai} If \eqref{item:backgroundext_critnotkappai} fails and there is an extender $F^*$ on the sequence of $\cW$ and an extender $F$ over $\cM_\xi$, both with critical point $\kappa = \kappa_i$ for some $i < \omega$ with $\alpha < \delta_{i+1}$,\footnote{In this case, we still follow the fully backgrounded construction as the new extender will not overlap any of the Woodin cardinals $\delta_i$.} such that $(J_\alpha^{\vec E}, \in, \vec E, F)$ is a premouse and for some $\nu < \alpha$, $V_{\nu+\omega}^\cW \subseteq \Ult(\cW,F^*)$ and \[ F \upharpoonright \nu = F^* \cap ([\nu]^{<\omega} \times J_\alpha^{\vec E}), \] we let \[\cN_{\xi+1} = (J_{\alpha}^{\vec E}, \in, \vec E, F).\] 
    \item \label{item:strongext} If \eqref{item:backgroundext_critnotkappai} and \eqref{item:backgroundext_critkappai} fail and there is a generically countably complete extender $F$ in $\cW$ with critical point $\kappa_i$ for some $i < \omega$ such that $\alpha > \delta_{i+1}$ and
      $(J_{\alpha}^{\vec E}, \in, \vec E, F)$ is a premouse, we let \[\cN_{\xi+1} = (J_{\alpha}^{\vec E}, \in, \vec E, F).\] 
    \item If \eqref{item:backgroundext_critnotkappai}, \eqref{item:backgroundext_critkappai} and \eqref{item:strongext} fail, we let \[\cN_{\xi+1} = (J_{\alpha+1}^{\vec E}, \in, \vec E, \emptyset).\] 
    \end{enumerate}
    In all cases, we 
    let $\cM_{\xi+1} = \bC_\omega(\cN_{\xi+1})$, if it is defined.
  \item For $\xi+1 < \Xi$, if $\cM_\xi = (J_\alpha^{\vec E}, \in, \vec E, F)$ is active, we let \[\cN_{\xi+1} = (J_{\alpha+1}^{\vec E^\prime}, \in, \vec E^{\prime}, \emptyset)\] for $\vec E^\prime = \vec E^\frown F$. Moreover, 
    we let $\cM_{\xi+1} = \bC_\omega(\cN_{\xi+1})$, if it is defined.
  \item If $\lambda \leq \Xi$ is a limit ordinal or $\lambda = \Ord$, we let $\cN_\lambda = \cM_\lambda$ be the unique passive premouse such that for all ordinals $\beta$, $\omega\beta \in \cN_\lambda \cap \Ord$ iff $J_\beta^{\cN_\alpha}$ is defined and eventually constant as $\alpha$ converges to $\lambda$, and for all such $\omega\beta \in \cN_\lambda \cap \Ord$, $J_\beta^{\cN_\lambda}$ is given by the eventual value of $J_\beta^{\cN_\alpha}$ as $\alpha$ converges to $\lambda$.
  \end{enumerate}
  We say the $\cM$-construction breaks down and stop the construction if $\bC_\omega(\cN_\xi)$ is not defined for some $\xi \in \Ord$.
  Otherwise we say the $\cM$-construction converges and $\cM = \cM_{\Ord}$ is the result of the construction.
\end{definition}

\begin{remark}
We can similarly perform $\cM$-constructions in set-sized translatable structures that satisfy a sufficiently large fragment of $\ZFC$.
\end{remark}

Recall that we fixed a translatable structure $\cW$. The purpose of the rest of this paper is to show that an $\cM$-construction inside a translatable structure produces the desired output, i.e., 
\begin{enumerate}[(a)]
    \item it results in a proper class model $\cM$ (i.e., $\Xi = \Ord$),
    \item in $\cM$, all $\kappa_i$ for $i<\omega$ are strong cardinals.
\end{enumerate}
That all $\delta_i$ for $i<\omega$ are Woodin cardinals in $\cM$ follows from the arguments in \cite{MS94}. 

The argument in the rest of this paper is organized inductively on $i<\omega$. That means, at stage $i$ of the argument we inductively assume that if $\cM^{(i)}$ denotes the result of an $\cM$-construction where Clause \eqref{item:strongext} is restricted to extenders with critical point in $\{\kappa_j \mid j<i\}$,
\begin{description}
    \item[$(\mathsf{IH}.1)_i$] for every $j<i$ and every $\bar\cW$ the transitive collapse of a hull of $\cW|\Omega$ for some sufficiently large $\Omega$ such that $\cW|\kappa_j \subseteq \bar\cW$, if \[\pi \colon \bar\cW \rightarrow \cW\] denotes the uncollapse embedding, $\bar\cM$ the collapse of the result of an $\cM$-construction where Clause \eqref{item:strongext} is restricted to extenders with critical point in $\{\kappa_j \mid j<i\}$ in $\cW|\Omega$ and $(\bar\kappa_l, \bar\delta_l \mid l<\omega)$ the preimages of $(\kappa_l, \delta_l \mid l<\omega)$, then \[\bar\cM \text{ is } (\delta_{j+1}, \delta_{j+1})\text{-iterable} \] with respect to extenders with critical point above $\kappa_j$. 
\end{description}
By running the proofs of Lemmas \ref{lem:constructionconverges} and \ref{lem:kappaiarestrong} this allows us to inductively assume that
\begin{description}
    \item[$(\mathsf{IH}.2)_i$] the construction of $\cM^{(i)}$ does not break down, and
    \item[$(\mathsf{IH}.3)_i$] all $\kappa_j$ for $j<i$ are strong in $\cM^{(i)}$.
\end{description}
For each $i<\omega$, write $(\mathsf{IH})_i$ for the conjuction of $(\mathsf{IH}.1)_i$, $(\mathsf{IH}.2)_i$, and $(\mathsf{IH}.3)_i$.   

\section{Characterizing extenders}\label{sec:directlimitsystems}

In order to prove that the result of an $\cM$-construction in a translatable structure has a limit of strong cardinals, we need a better understanding of the extenders on its sequence. To characterize the extenders with critical points that are the different ``future strong cardinals'' $\kappa_i$, $i<\omega$, we first define direct limit systems and internalize them to $\cM$.

\subsection{Direct limit systems at $\kappa_0$}\label{subsec:i=0}

For simplicity, we start with a direct limit system at $\kappa_0$, the first cardinal that is strong up to $\delta_1$ in $\cW$.\footnote{Note that by a short argument due to Farmer Schlutzenberg in this setting $\kappa_0$ will also be the first cardinal that is strong up to $\delta_1$ in $\cM^{(0)}$ since the fully backgrounded construction will add extenders witnessing this.} The general case for $\kappa_i$, $i>0$, will be handled in the next section. For this case with $i=0$ we do not require any inductive hypothesis (note that $(\mathsf{IH})_0$ is an empty condition). Recall, that $\cM^{(0)}$ denotes the result of an $\cM$-construction where Clause \eqref{item:strongext} is restricted to extenders with critical point $\kappa_0$. Note that $\cM^{(0)}$ is not necessarily a proper class model at this point. Showing that the construction does not break down and $\cM^{(0)}$ is indeed a proper class model is one of the applications of the results in this subsection.

We start working in $\cW$. Let $\cP^0_\infty$ be the direct limit of all $\Sigma_0$-iterates of $\cP^0 = Lp^\infty_\omega(\cW | \delta_0)$ that are countable in $\cW[g]$ for some generic $g \subseteq \Col(\omega, {<}\kappa_0)$. By standard arguments this direct limit is well-defined and well-founded, see, for example, \cite{StW16}. Let $\delta_\infty^0$ be the image of $\delta_0$ in $\cP_\infty^0$ and write \[ \cP_\infty^{0,-} = \cP_\infty^0 | \delta_\infty^0. \]

 Note that $\cM^{(0)}|\delta_1 = \cM|\delta_1$ is equal to the result of a fully backgrounded extender construction in $\cW | \delta_1$ (cf. \cite[Lemma 2.21]{MSW}). Therefore, the $\cM$-construction as well as the construction of $\cM^{(0)}$ succeed up to $\delta_1$ and $\cM^{(0)}|\delta_1 = \cM|\delta_1$ are well-defined. Also, if both constructions succeed far enough, $\cM^{(0)}|\delta_2 = \cM |\delta_2$.
 By the argument in \cite[Lemma 2.14]{Sa17}, $\kappa_0$ is a limit of Woodin cardinals in $\cM^{(0)}$. 
 We aim to identify $\cP^{0}_\infty$ inside $\cM^{(0)}$ as a direct limit of models such that densely many of these models can be obtained via certain fully backgrounded constructions as follows. 
 For any ordinal $\delta^* < \kappa_0$ that is a cutpoint in $\cM^{(0)}|\kappa_0$, let
\[ (\cL_\eta \mid \eta \leq \eta^*) \] be the result of a fully backgrounded $L[E]$-construction over $\emptyset$ in $\cM^{(0)} | \delta_1$ using extenders with critical point above $\delta^*$, in other words, $\cL_{\eta^*} = (L[E]^{\geq \delta^*}(\emptyset))^{\cM^{(0)}|\delta_1}$.
 Let $\gamma \leq \eta^*$ be the least ordinal such that \[ (Lp_\omega^{\infty}(\cL_\gamma))^\cW \models \text{``$\cL_\gamma \cap \Ord$ is Woodin'',} \] let $\cN_{\delta^*} = \cL_\gamma$ and write \[ \cN_{\delta^*}^+ = (Lp_\omega^{\infty}(\cL_\gamma))^\cW. \]
 As $\kappa_0$ is a limit of Woodin cardinals in $\cM^{(0)}$ (cf. \cite[Lemma 2.14]{Sa17})\footnote{We will prove a weaker statement that still suffices to show that these constructions reach Woodin cardinals for the general case $i>0$ in the next subsection.} and the least ${<}\delta_1$-strong cardinal in $\cW$, there are such Woodin cardinals in $(Lp_\omega^{\infty}(\cL_\gamma))^\cW$ cofinally below $\kappa_0$. In particular, for any cutpoint $\delta^* < \kappa_0$ in $\cM^{(0)}$ and $\gamma \leq \eta^*$ as above, $\cL_\gamma \cap \Ord < \kappa_0$. We start with arguing that these models $\cN_{\delta^*}$ and $\cN_{\delta^*}^+$ are in fact inside $\cM^{(0)}$. This is a consequence of the following more general lemma, see also \cite[Lemma 2.12]{Sa17}.

 \begin{lemma}\label{lem:Lpatkappa0inM}
 The function $X \mapsto (Lp^{\infty}(X))^\cW$ for $X \in \cM^{(0)}|\kappa_0$ is definable over $\cM^{(0)}|\delta_1$.
 \end{lemma}
 \begin{proof}
    Let $\delta^* < \kappa_0$ be a cutpoint in $\cM^{(0)}|\delta_1$ with $X \cap \Ord < \delta^*$ and let $\cL_{\eta^*} = (L[E]^{\geq \delta^*}(X))^{\cM^{(0)}|\delta_1}$ be the result of a fully backgrounded $L[E]$-construction over $X$ in $\cM^{(0)} | \delta_1$ using extenders with critical point above $\delta^*$.
    Let
    \begin{eqnarray*}
    \mathcal{O}^{\cL_{\eta^*}}_\eta = \bigcup\{ \cN \unlhd \cL_{\eta^*} \mid \cL_{\eta} \unlhd \cN, \rho_\omega(\cN) \leq \cL_\eta \cap \Ord, \text{ and }
    \text{there is no}\text{ extender } E \\ \text{ on the sequence of } \cN \text{ with }\crit(E) \leq \cL_\eta \cap \Ord < \lh(E) \}. 
    \end{eqnarray*} 
    We first show the following claim.

    \begin{claim}
      Let $\eta > X \cap \Ord$ be a cardinal in $\cM^{(0)}|\kappa_0$. Then
        \[ \mathcal{O}^{\cL_{\eta^*}}_\eta = (Lp^{\infty}(\cL_\eta))^{\cW}. \]
    \end{claim}
    \begin{proof}
The direction $\mathcal{O}^{\cL_{\eta^*}}_\eta \unlhd (Lp^{\infty}(\cL_\eta))^{\cW}$ is clear.
For the other direction, fix an arbitrary sound $\cN \unlhd Lp^{\infty}(\cL_\eta)$ in $\cW$ with $\cL_\eta \unlhd \cN$ and $\rho_\omega(\cN) = \cL_\eta \cap \Ord$.
Let $E$ be the index-least extender on the extender sequence of $\cL_{\eta^*}$ such that $\cL_\eta \cap \Ord$ is a cutpoint in $\Ult(\cL_{\eta^*}, E)$.
It suffices to show that $\cN \unlhd \Ult(\cL_{\eta^*}, E)$.
Note that both $\cN$ and $\Ult(\cL_{\eta^*}, E)$ are fully iterable in $\cW$ above $\cL_\eta \cap \Ord$ as the extenders used to construct $\cL_{\eta^*}$ are fully backgrounded by extenders in $\cM^{(0)}|\delta_1$ which are themselves fully backgrounded by extenders in $\cW$. This uses the observation that the extenders in $\cM^{(0)}$ with index below $\delta_1$ are all added by the fully backgrounded part of the $\cM$-construction.

By universality of fully backgrounded constructions (cf. \cite[Lemma 2.13]{Sa15}, this uses that there are no inner models with a superstrong cardinal in $\cW$), $\Ult(\cL_{\eta^*},E)$ wins the coiteration with $\cN$. As $\rho_\omega(\cN) = \cL_\eta \cap \Ord$ and the coiteration is above $\cL_\eta \cap \Ord$, the $\cN$-side cannot move as otherwise it would have to drop. Suppose the $\Ult(\cL_{\eta^*},E)$-side moves to an iterate $\cW^*$ of $\Ult(\cL_{\eta^*},E)$ such that $\cN \unlhd \cW^*$. If the iteration from $\Ult(\cL_{\eta^*},E)$ to $\cW^*$ drops on the main branch, $\cW^*$ is not fully sound, so $\cN \lhd \cW^*$. As $\rho_\omega(\cN) = \cL_\eta \cap \Ord$, $\cN$ is in fact a subset of $\cL_\eta \cap \Ord$ in $\cW^*$. But iterations above $\cL_\eta \cap \Ord$ do not add any new subsets of $\cL_\eta \cap \Ord$, thus $\cN \unlhd \Ult(\cL_{\eta^*},E)$, as desired. If the iteration from $\Ult(\cL_{\eta^*},E)$ to $\cW^*$ does not drop on the main branch, $\cW^* \cap \Ord = \cL_{\eta^*} \cap \Ord = \delta_1$ and again $\cN \lhd \cW^*$. So we can finish the argument as in the dropping case.     
  \end{proof}
 The lemma now follows from the next claim, which is proved like Claim 2.13 in the proof of \cite[Lemma 2.12]{Sa17}. It uses Clause \eqref{eq:translatablestructure6} in Definition \ref{def:translatablestructure} and generic interpretability of the involved iteration strategies with branch condensation.

 \begin{claim}
      For any $X \in \cM^{(0)}|\kappa_0$, $\cN \unlhd (Lp^{\infty}(X))^\cW$ if, and only if, there is some $\cR \unlhd \mathcal{O}^{\cL_{\eta^*}}_\eta$ such that 
      \begin{enumerate}[(1)]
\item $\rho_\omega(\cR) = \cL_\eta \cap \Ord$,
\item $\cR$ has $\omega$ Woodin cardinals with supremum $\nu$,
\item $\cN \in \cR$, and
\item $\cR \models \text{``}\cN \text{ is } \omega_1\text{-iterable in the derived model at }\nu\text{''}$.
 \end{enumerate}
 \end{claim}
 This finishes the proof of Lemma \ref{lem:Lpatkappa0inM}.
 \end{proof}

 We start internalizing $\cP_\infty^{0}$ into $\cM^{(0)}$ with the following observation, that still takes place in $\cW$.

 \begin{lemma}\label{lem:Ndelta*suitable}
  Let $\delta^* < \kappa_0$ be a cutpoint in $\cM^{(0)}$. Then, in $\cW$, $\cN_{\delta^*}^+$ is a $\Sigma_0$-iterate of $\cP^0$ and the iteration does not drop on the main branch.
\end{lemma}
\begin{proof}
  Work in $\cW$ and compare $\cM^{(0)}|\delta_0$ against the construction $(\cL_\eta \mid \eta \leq \eta^*)$ giving rise to $\cN_{\delta^*}$. By stationarity and universality of fully backgrounded constructions\footnote{Cf. \cite[Lemmas 2.11 and 2.13]{Sa15}, note that these results also hold in our setting where $\cN_{\delta^*}$ is the result of a fully backgrounded construction in $\cM^{(0)}|\delta_1$, that itself is an initial segment of the result of a fully backgrounded construction in $\cW$.}, $\cM^{(0)}|\delta_0$ iterates via $\Sigma_0$ without drops on the main branch to some level of the construction $(\cL_\eta \mid \eta \leq \eta^*)$. As $\cN_{\delta^*}^+ = Lp_\omega^\infty(\cL_\gamma)$ for $\gamma$ the least ordinal such that $\cL_\gamma \cap \Ord$ is Woodin in $Lp_\omega^\infty(\cL_\gamma)$, it follows by super fullness preservation of $\Sigma_0$ that in fact $\cP^0 = Lp_\omega^\infty(\cW|\delta_0)$ iterates via $\Sigma_0$ to $\cN_{\delta^*}^+$.
\end{proof}

 Similarly, we obtain the following lemma.

\begin{lemma}\label{lem:iterationNdelta*Ndelta**}
  Let $\delta^*$ and $\delta^{**}$ be cutpoints in $\cM^{(0)}$ with $\delta^* < \delta^{**} < \kappa_0$. Then, in $\cW$, $\cN_{\delta^*}$ iterates to $\cN_{\delta^{**}}$ without drops on the main branch and via the $\cP^0$-to-$\cN_{\delta^*}^+$-tail of $\Sigma_0$. 
\end{lemma}
\begin{proof}
  Let $(\cL_\eta \mid \eta \leq \eta^*)$ and $(\cL_\eta^\prime \mid \eta \leq \eta^\prime)$ be the constructions in $\cM^{(0)}$ giving rise to $\cN_{\delta^*}$ and $\cN_{\delta^{**}}$ respectively. Working in $\cW$, compare $\cN_{\delta^*}$ against the construction $(\cL_\eta^\prime \mid \eta \leq \eta^\prime)$, using the iteration strategy for $\cN_{\delta^*}$ induced by the background strategy. By stationarity and universality of fully backgrounded constructions, see \cite[Lemmas 2.11 and 2.13]{Sa15}, $\cN_{\delta^*}$ iterates to some level of the construction $(\cL_\eta^\prime \mid \eta \leq \eta^\prime)$ without drops on the main branch. As $\Sigma_0$ has branch condensation and fully backgrounded constructions absorb strategies with branch condensation (cf. \cite[Lemma 2.15]{Sa15}) the iteration of $\cN_{\delta^*}$ into a level of $(\cL_\eta^\prime \mid \eta \leq \eta^\prime)$ is in fact a $\Sigma_0$-iteration. Using that $\Sigma_0$ is super fullness preserving, we obtain that $\cN_{\delta^*}$ iterates to $\cN_{\delta^{**}}$. 
\end{proof}

Now, we argue that we can consider $\cP_\infty^{0}$ as a direct limit of premice such that densely many of these premice are of the form $\cN_{\delta^*}$ for cutpoints $\delta^*<\kappa_0$ in $\cM^{(0)}$. Let $\delta^*_0$ be the least cutpoint in $\cM^{(0)} | \kappa_0$ above $\delta_0$. 
Work in $\cM^{(0)}$. We will use the definitions of $\kappa_0$-short and $\kappa_0$-maximal trees $\cT$ referring to $(Lp^\infty(\cM(\cT)))^{\cW}$ as well as other notions introduced in Section \ref{subsec:suitablepm}. By Lemma \ref{lem:Lpatkappa0inM} these notions can be defined internally in $\cM^{(0)}$. Now let $\cN^*$ be the result of making $\cM^{(0)}|\kappa_0$ generically generic over an iterate of $\cN_{\delta^*_0}^+$ in the sense of \cite[Discussion after Definition 3.36]{Sa15} via the extender algebra with many generators (see \cite[Theorem 4.5]{Fa}). By considering pseudo-genericity iterations in the sense of \cite[Theorem 3.16]{StW16} we can obtain $\cN^* \in \cM^{(0)}$. Recall that as $\cW$ is a translatable structure, there is a set of term relations strongly guiding $\Sigma_0$. In particular, there is a set of term relations $\vec\tau \subseteq \cN^*$ such that    
whenever $G \subseteq \Col(\omega, |\cN^*|)$ is generic and $(\tau_k \mid k < \omega)$ is a fixed generic enumeration of $\vec\tau$, then $(\Sigma_{\cN^*})^G$ is strongly guided by $(\tau_k \mid k < \omega)$. For $l \leq \omega$, let $f_{l} \in F(\kappa_0, \Sigma_0)$ be the function given by \[ f_{l}(\cQ) = \bigoplus_{k<l} \tau_k^{\cQ} \] for $\cQ \in S(\kappa_0, \Sigma_0)$, where, for $\tau_k$ a $\Col(\omega, ((\delta^{\cN^*})^{+n})^{\cN^*})$-term relation for some $n<\omega$, 
\begin{align*}
     \tau_k^{\cQ} = \{ (p,\sigma) \mid p \in \Col(\omega, ((\delta^\cQ)^{+n})^{\cQ}), \sigma \in \cQ^{\Col(\omega, ((\delta^\cQ)^{+n})^{\cQ})} \text{ a standard name} \\
     \text{for a real and } p \Vdash^{\cQ} \sigma \in \tau_k\}. \;\;\;\; 
\end{align*}
Note that the definition of $\tau_k^{\cQ}$ makes sense by interpreting part of the term relation $\tau_k$ as a $\Col(\omega, ((\delta^{\cQ})^{+n})^{\cQ})$-term relation.
Then $f_l \in \cM^{(0)}$ for every $l<\omega$.
Now let
\begin{align*}
    \cI = \{ (\cQ,l) \mid \cQ \text{ is $\kappa_0$-suitable, strongly $f_l$-iterable and countable in } (\cM^{(0)})^{\Col(\omega, {<}\kappa_0)} \}
\end{align*}
and for $(\cQ,l), (\cQ',l') \in \cI$ let 
\[ (\cQ,l) \leq_\cI (\cQ',l') \text{ iff } \cQ \text{ is a $(\kappa_0, \Sigma)$-correct iterate of } \cQ' \text{ and } l \leq l'.  \]
Note that the models $\cN_{\delta^*}^+$ for cutpoints $\delta^* < \kappa_0$ are cofinal in the $\Sigma_0$-iterates of $\cP^0$ that are countable in $\cW[g]$. By Lemmas \ref{lem:Ndelta*suitable} and \ref{lem:iterationNdelta*Ndelta**} combined with standard arguments\footnote{See, for example, \cite[Section 3]{StW16} or \cite[Section 4.1]{Sa15}.} on the existence of strongly $f$-iterable premice for $f \in F(\kappa_0, \Sigma_0)$, $(\cI, \leq_\cI)$ is a directed system and, in $\cW$, there are $\Sigma$-iterations between the models. In $\cM^{(0)}$, there is, for each $(\cQ,l) \leq_\cI (\cQ',l')$ a natural embedding \[ \pi_{(\cQ,l),(\cQ',l')} \colon H^{\cQ}_{f_l} \rightarrow H^{\cQ'}_{f_{l'}}. \] We let $\cF$ be the directed system of models $H^{\cQ}_{f_l}$ indexed by $(\cQ,f_l) \in (\cI, \leq_\cI)$ together with the embeddings $\pi_{(\cQ,l),(\cQ',l')}$. Let $\cP_\infty^*$ be the direct limit of $\cF$ and let $\pi_{(\cQ,l),\infty} \colon H^{\cQ}_{f_l} \rightarrow \cP_\infty^*$ be the corresponding direct limit embedding.

The following lemma, which is our desired internalization of $\cP_\infty^{0}$, now follows as in standard HOD computation arguments, see, for example, \cite[Lemma 6.32]{StW16}. The second equality in Lemma \ref{lem:quasidirectlimitkappa0} is a consequence of fullness preservation.


\begin{lemma}\label{lem:quasidirectlimitkappa0}
  $\cP_\infty^* = \cP^{0}_\infty = (Lp_\omega^{\infty})^\cW(\cP_\infty^{0,-})$.
\end{lemma}

Recall that $\cP_\infty^0$ is a $\Sigma_0$-iterate of $\cP^0$ in $\cW$. Write $\Sigma_{\cP_\infty^{0}}$ for the corresponding tail strategy of $\Sigma_0$ and $\pi_\infty^0 \colon \cP^0 \rightarrow \cP_\infty^0$ for the iteration embedding. Moreover, recall that $\cM^{(0)}|\delta_1$ is obtained from a fully backgrounded construction in $\cW|\delta_1$. The proof of the successor case in the proof of \cite[Theorem 6.5]{Sa15} (cf. \cite[Lemma 2.15]{Sa17}) shows that $\Sigma_{\cP_\infty^{0}}$ up to $\delta_1$ is in $\cM^{(0)}$. More precisely:

\begin{lemma}\label{lem:TailSigma0InM}
  $\Sigma_{\cP_\infty^{0}} \upharpoonright (\cM^{(0)}|\delta_1)$ is amenable to $\cM^{(0)}|\delta_1$, i.e., for every $X \in \cM^{(0)}|\delta_1$, $\Sigma_{\cP_\infty^{0}} \upharpoonright X \in \cM^{(0)}|\delta_1$.
\end{lemma}

We will show in Lemma \ref{lem:TailSigmaiInM} that, if we suppose inductively that $(\mathsf{IH})_j$ holds, i.e., the $\cM$-construction in $\cW$ where Clause \eqref{item:strongext} is restricted to extenders with critical point in $\{\kappa_l \mid l \leq j\}$ is well-defined up to $\delta_{j+1}$ and, if $\cM^{(j)}|\delta_{j+1}$ is the result of this construction up to $\delta_{j+1}$, $\Sigma_{\cP_\infty^{0}} \upharpoonright (\cM^{(j)} | \delta_{j+1})$ is amenable to $\cM^{(j)} | \delta_{j+1}$.

\subsection{Direct limit systems at $\kappa_i$}\label{subsec:i>0}

We now define more general direct limit systems that will be used to characterize extenders with critical point $\kappa_i$ for $i>0$ and argue that their direct limits $\cP_\infty^i$ are in $\cM$. Recall that we fixed a translatable structure $\cW$. In this subsection we assume $(\mathsf{IH})_i$ and let $\cM^{(i)}$ denote the result of an $\cM$-construction where Clause \eqref{item:strongext} is restricted to extenders with critical point in $\{ \kappa_j \mid j \leq i \}$. Note that $\cM^{(i)}$ is not necessarily a proper class model at this point as the $\cM$-construction might break down for extenders with critical point $\kappa_i$ added by Clause \eqref{item:strongext} of Definition \ref{def:translationprocedure}. One of the applications of the results of this subsection is that this does not happen and $\cM^{(i)}$ is indeed a proper class model.
Recall that $\kappa_i$ is the least ${<}\delta_{i+1}$-strong cardinal in $\cW$.\footnote{Again, by a short argument in this setting $\kappa_i$ will also be the first cardinal above $\delta_i$ that is strong up to $\delta_{i+1}$ in $\cM^{(i)}$ since by an argument of Farmer Schlutzenberg the fully backgrounded construction will add extenders witnessing this.} To avoid confusion, we will emphasize where the inductive hypothesis is used. By our inductive hypothesis, $\cM^{(j)}$ is a well-defined proper class model for all $j<i$. Note that by construction $\cM^{(i)} |\delta_{i+1} = \cM^{(i-1)}|\delta_{i+1}$, so $\cM^{(i)}|\delta_{i+1}$ is well-defined.

Working in $\cW$, let $\cP_\infty^i$ be the direct limit of all $\Sigma_i$-iterates of $\cP^i = Lp_\omega^{\Sigma_{i-1},\infty}(\cW|\delta_i)$ that are countable in $\cW[g]$, where $g$ is $\Col(\omega, {<}\kappa_i)$-generic over $\cW$. 
By standard arguments this direct limit is well-defined and well-founded (see \cite{Sa15}) and we can let $\delta_\infty^i$ be the image of $\delta_i$ in $\cP_\infty^i$ under the direct limit embedding. Moreover, we write \[ \cP_\infty^{i,-} = \cP_\infty^i | \delta_\infty^i. \]

As in the case $i=0$, we aim to identify $\cP^{i}_\infty$ inside $\cM^{(i)}|\delta_{i+1}$ as the direct limit of premice where densely many of these premice result from certain backgrounded constructions, in this case relative to strategies.

Let $i > 0$. Working in $\cM^{(i)}$, we say an ordinal $\delta^* < \kappa_i$ is a \emph{$\kappa_i$-weak cutpoint} iff all extenders $E$ on the sequence of $\cM^{(i)}$ overlapping $\delta^*$, i.e., with $\crit(E) < \delta^* < \lh(E)$, have critical point $\kappa_j$ for some $j<i$. 

For any $\kappa_i$-weak cutpoint $\delta^*$ in $\cM^{(i)}|\kappa_i$ we define $\cN_{\delta^*}$ as the result of a hod pair construction\footnote{Similar as in \cite[Definition 2.6]{Sa15} but we are working in a fine structural model, not in a weak background triple.} above $\delta^*$ in $\cM^{(i)}|\delta_{i+1}$ relative to certain tail strategies of $\Sigma_k$ for $k<i$ until it reaches $i+1$ full Woodin cardinals. We make this precise in what follows. This part of the argument works in $\cM^{(i)}|\delta_{i+1}$. As already noted above, this segment of $\cM^{(i)}$ exists by $(\mathsf{IH}.2)_i$ as extender with critical point $\kappa_l$ for $l \geq i$ that are not fully backgrounded (i.e., that are added by Clause \eqref{item:strongext} of Definition \ref{def:translationprocedure}) are only added past stage $\delta_{i+1}$ of the construction.
Note that by Lemmas \ref{lem:TailSigma0InM} and \ref{lem:TailSigmaiInM} we may and will assume inductively that the relevant tail of $\Sigma_{k-1}$ restricted to $\cM^{(i)}|\delta_{i+1}$ is amenable to $\cM^{(i)}|\delta_{i+1}$ for all $k \leq i$.

To keep the notation simple, we illustrate the definition of the models $\cN_{\delta^*}$ at $\kappa_1$ and $\kappa_2$. The definition for $\kappa_i$, $i>2$, is then a straightforward generalization. In what follows, whenever we say that we ``add a strategy'' we mean that the strategy gets added to the model in a fine structural sense as in \cite{Sa15}. We decided to not go into the fine structural details here in order to be able to focus on the new ideas in this construction. We refer the interested reader to \cite[Chapter 1]{Sa15}.

So for the case $i=1$, let $\delta^* > \delta_1$ be a $\kappa_1$-weak cutpoint. Then $\cN_{\delta^*}$ is given by the following construction: Until it reaches the first Woodin cardinal a hod pair construction is just a fully backgrounded construction. So we start with performing a fully backgrounded construction as in \cite{MS94} (without the smallness assumption) using extenders with critical point above $\delta^*$ inside $\cM^{(1)}|\delta_2$. Let $(\cL_\eta \mid \eta \leq \eta^*)$ be the result of the construction and suppose it reaches an ordinal $\gamma$ such that \[ (Lp_\omega^{\infty}(\cL_\gamma))^\cW \models \text{``$\cL_\gamma \cap \Ord$ is Woodin.''} \] Then we let $\cR_0 = \cL_\gamma$ and write \[ \cR_0^+ = (Lp_\omega^{\infty}(\cL_\gamma))^\cW. \] If it exists, we say that $\cL_\gamma$ is the level where the construction reaches the first full Woodin cardinal. We stop the construction and say that it does not succeed if no such level is reached with $\cL_\gamma \cap \Ord < \kappa_1$.  We will show in Lemma \ref{lem:Ndeltasucceeds} that such a level $\cL_\gamma$ is indeed reached below $\kappa_1$ and hence all extenders in $\cR_0$ are backgrounded not only in $\cM^{(1)}$ but also in $\cW$.

 We will argue in Lemma \ref{lem:MinftyIteratesToConstructioni=1} that $\cR_0$ is a $\Sigma_{\cP_\infty^0}$-iterate of $\cP_\infty^0$ (that itself is a $\Sigma_0$-iterate of $\cP^0$). Write $\Sigma_{\cR_0}$ for the respective tail strategy of $\Sigma_0$. Then $(\cR_0, \Sigma_{\cR_0})$ is a hod pair. As $\Sigma_0$ has branch condensation, the fully backgrounded construction will absorb the tail strategy $\Sigma_{\cR_0}$ by \cite[Lemma 2.15]{Sa15}. More precisely, $\Sigma_{\cR_0}$ agrees with the iteration strategy $\cR_0$ inherits from $\cM^{(1)}$ (and ultimately from $\cW$) as the result of a backgrounded construction. 
 Therefore, it makes sense to define $\cR_1$ as the result of adding $\Sigma_{\cR_0}$ to $\cR_0$. That means, we proceed further in the fully backgrounded construction, now relative to $\Sigma_{\cR_0}$, and the construction will not project across $\cR_0 \cap \Ord$. Write $(\cL_\eta^1 \mid \eta \leq \eta^*_1)$ for the result of this construction.
 If it exists, let $\gamma_1$ be the least ordinal such that 
  \[ (Lp_\omega^{\Sigma_{\cR_0},\infty}(\cL^1_{\gamma_1}))^\cW \models \text{``$\cL^1_{\gamma_1} \cap \Ord$ is Woodin''.} \] Write $\cN_{\delta^*} = \cR_1 = \cL^1_{\gamma_1}$ and let \[ \cN_{\delta^*}^+ = (Lp_\omega^{\Sigma_{\cR_0},\infty}(\cR_1))^\cW. \] If it exists, we say that $\cL^1_{\gamma_1}$ is the level where the construction reaches the second full Woodin cardinal. Again, if no such level is reached below $\kappa_1$ we say that the construction does not succeed.
 We will argue in Lemma \ref{lem:Ndeltasucceeds} that the construction succeeds and in Lemma \ref{lem:LpatkappaiinM} that $\cN_{\delta^*}^+$ is definable in $\cM^{(1)}|\delta_2$.

\begin{lemma}\label{lem:Ndeltasucceeds}
    For each $i < \omega$ and each $\kappa_i$-weak cutpoint $\delta^*$ with $\delta_i < \delta^* < \kappa_i$ in $\cM^{(1)}$, the construction of $\cN_{\delta^*}^+$ does not fail, i.e., it reaches a level with $i+1$ full Woodin cardinals below $\kappa_i$. 
\end{lemma}
\begin{proof}
    As discussed in the previous subsection, for $i=0$ this follows directly from \cite[Lemma 2.14]{Sa17}. We give the argument for $i=1$ here, the general case $i>1$ is similar. 
    
    Fix a $\kappa_1$-weak cutpoint $\delta^*$ with $\delta_1 < \delta^* < \kappa_1$ in $\cM^{(1)}$ and consider the construction of $\cN_{\delta^*}^+$. The argument that this construction reaches the first full Woodin is an easier version of the argument to follow for the second full Woodin.
    So suppose that the construction reaches the first full Woodin at some level $\cR_0$. Then we proceed further in the fully backgrounded construction in $\cM^{(1)}|\delta_2$, now relative to $\Sigma_{\cR_0}$, and write $(\cL_\eta^1 \mid \eta \leq \eta^*_1)$ for the result of this construction. We first show the following claim.

    \begin{claim}\label{cl:delta2limitofLpWoodins}
        $\delta_2$ is a limit of ordinals $\xi$ such that for some $\eta \leq \eta^*_1,$ in $\cW$, \[ Lp^{\Sigma_{\cR_0}, \infty}(\cL_\eta^1) \vDash \text{``$\xi = \cL_\eta^1 \cap \Ord$ is Woodin.''} \]
    \end{claim}
    \begin{proof}
        Suppose this is not the case and let $\xi^*$ be the supremum of the ordinals $\xi$ in the statement of the claim. If there is no such $\xi$, we let $\xi^* = \delta_1$. 
        Let $\cQ$ be the result of a fully backgrounded construction in $\cW|\delta_2$ relative to $\Sigma$
        using only extenders with critical point above $\xi^*$. 

        Work in $\cW$ and compare $\cP^1 = Lp_\omega^{\Sigma_0, \infty}(\cW|\delta_1) = Lp_\omega^{\Sigma, \infty}(\cW|\delta_1)$ via the strategy $\Sigma$ against the construction of $\cQ$. Stationarity and universality of fully backgrounded constructions (see \cite[Lemmas 2.11 and 2.13]{Sa15}) imply that $\cP^1$ iterates to some $\cP^*$ with $\cP^* \unlhd \cQ$. Moreover, by universality, the iteration from $\cP^1$ to $\cP^*$ is non-dropping and, by choice of the strategy $\Sigma$, fullness preserving. Write $i \colon \cP^1 \rightarrow \cP^*$ for the iteration embedding and $\Sigma_{\cP^*}$ for the corresponding tail strategy of $\cP^*$. Let $\delta^* = i(\delta_1)$. Then, by fullness preservation, $\cP^* | ((\delta^*)^+)^{\cP^*} = Lp^{\Sigma_{\cP^*}, \infty}(\cP^*|\delta^*)$. As $\Sigma$ has branch condensation, it is by \cite[Lemma 2.15]{Sa15} absorbed by the fully backgrounded construction resulting in $\cQ$. Therefore, in $\cW$, 
        \[ Lp_\omega^{\Sigma, \infty}(\cQ|\delta^*) \vDash \text{``$\delta^*$ is Woodin.''} \]
        In fact, the argument we just gave shows that there is an unbounded set of ordinals $\delta^*$ such that for some $\cQ$, $Lp_\omega^{\Sigma, \infty}(\cQ|\delta^*) \vDash \text{``$\delta^*$ is Woodin.''}$ Pick an $\eta^*$ with $\cL^1_{\eta^*} \cap \Ord = \delta^*$ such that  $Lp_\omega^{\Sigma, \infty}(\cQ|\delta^*) \vDash \text{``$\delta^*$ is Woodin.''}$

            \begin{subclaim*}
                 In $\cW$, $Lp_\omega^{\Sigma, \infty}(\cL^1_{\eta^*}|\delta^*) \vDash \text{``$\delta^*$ is Woodin.''}$
            \end{subclaim*}
            \begin{proof}
                Suppose toward a contradiction that the subclaim does not hold. This implies that $\delta^*$ is not Woodin in $\cW$.
                If $\delta^*$ is a strong cutpoint cardinal in $\cW$, 
                we can consider the $\cP$-construction of $\cW|((\delta^*)^{+\omega})^\cW$ over $\cQ|\delta^*$ in the sense of \cite[Definition 3.41]{Sa15}\footnote{$\cP$-constructions are called $\cS$-constructions in \cite{Sa15} as they were introduced by Steel in \cite{SchSt09}.} By \cite[Lemma 3.43]{Sa15} this $\cP$-construction reaches a level where $\delta^*$ is not Woodin.
                In particular, $\delta^*$ is not Woodin in $Lp_\omega^{\Sigma, \infty}(\cQ|\delta^*)$, a contradiction.

                If $\delta^*$ is not a strong cutpoint cardinal in $\cW$, we can perform the $\cP$-construction inside $\Ult(\cW,F)$, where $F$ is the least extender overlapping $\delta^*$ in $\cW$ such that $\delta^*$ is a strong cutpoint in $\Ult(\cW,F)$. As before, this yields that $\delta^*$ is not Woodin in $(Lp_\omega^{\Sigma, \infty}(\cQ|\delta^*))^{\Ult(\cW,F)} = (Lp_\omega^{\Sigma, \infty}(\cQ|\delta^*))^\cW$.
            \end{proof}
             As $\xi^* < \delta^*$ and $Lp_\omega^{\Sigma, \infty}(\cL^1_{\eta^*}|\delta^*) = Lp_\omega^{\Sigma_{\cR_0}, \infty}(\cL^1_{\eta^*}|\delta^*)$, the subclaim immediately yields a contradiction.        
    \end{proof}
    
    We are left with showing that the previous claim implies Lemma \ref{lem:Ndeltasucceeds}. For this, we will perform a reflection argument using that $\kappa_1$ is ${<}\delta_1$-strong in $\cW$.

    \begin{claim}
        $\kappa_1$ is a limit of Woodin cardinals in $\cL^1_{\eta^*_1}$.
    \end{claim}
    \begin{proof}
        Let $\zeta < \kappa_1$ be arbitrary. By Claim \ref{cl:delta2limitofLpWoodins} there is some $\xi < \delta_2$ and some $\eta$ such that 
        $(Lp^{\Sigma_{\cR_0}, \infty}(\cL_\eta^1))^{\cW} \vDash \text{``$\xi = \cL_\eta^1 \cap \Ord$ is Woodin.''}$
        Let $E$ be an extender on the sequence of $\cW$ witnessing that $\kappa_1$ is ${<}\delta_2$-strong such that $\lh(E) = \str(E)$ is much larger than $\xi$. In particular, we pick $E$ sufficiently strong that, if  $(\cL_\eta^1)^{\Ult(\cW,E)}$ denotes the result of the $\cN_{\delta^*}^+$-construction in the result of an $\cM$-construction in $\Ult(\cW,E)$ instead of $\cW$, $(Lp^{\Sigma_{\cR_0}, \infty}(\cL_\eta^1))^{\Ult(\cW,E)} \vDash \text{``$\xi = \cL_\eta^1 \cap \Ord$ is Woodin.''}$
        Then, letting $i_E$ denote the $E$-ultrapower embedding, 
        \begin{eqnarray*}
            \Ult(\cW, E) \models \text{``} \exists \xi \text{ with } \zeta < \xi < i_E(\kappa_1) \text{ such that } \\ 
            Lp^{\Sigma_{\cR_0}, \infty}(\cL_\eta^1) \vDash \text{``$\xi = \cL_\eta^1 \cap \Ord$ is Woodin.''}
        \end{eqnarray*}
        By elementarity, 
        \begin{eqnarray*}
            \cW \models \text{``} \exists \xi \text{ with } \zeta < \xi < \kappa_1 \text{ such that } \;\;\;\;\;\; \\ 
            Lp^{\Sigma_{\cR_0}, \infty}(\cL_\eta^1) \vDash \text{``$\xi = \cL_\eta^1 \cap \Ord$ is Woodin.''}
        \end{eqnarray*}
        In particular, $\xi$ as in the above equation is a Woodin cardinal in $\cL^1_{\eta^*_1}$, as desired.
    \end{proof}
    This claim immediately implies Lemma \ref{lem:Ndeltasucceeds}.
\end{proof}

The next lemma will be applied in the inductive step from $i=0$ to $i=1$. So we will prove it under the corresponding inductive hypothesis.

\begin{lemma}\label{lem:MinftyIteratesToConstructioni=1onecase}
   Suppose $(\mathsf{IH})_1$. Then $\cR_0$ is a $\Sigma_{\cP_\infty^{0}}$-iterate of $\cP_\infty^{0}$ in $\cM^{(1)}|\delta_2 = \cM^{(0)}|\delta_2$ and if $\Sigma_{\cR_{0}}$ denotes the corresponding tail strategy, $\Sigma_{\cR_0} \upharpoonright \cM^{(1)}|\delta_2$ is amenable to $\cM^{(1)}|\delta_2$, i.e., for any $X \in \cM^{(1)}|\delta_2$, $\Sigma_{\cR_0}\upharpoonright X \in \cM^{(1)}|\delta_2$.
\end{lemma}
\begin{proof}
  Recall that $\cR_0$ is constructed in $\cM^{(1)} | \delta_2$. As we inductively assume in $(\mathsf{IH})_1$ that $\cM^{(0)}$ is a well-defined proper class model and $\kappa_0$ is fully strong in $\cM^{(0)}$, we have that $\Sigma_{\cP_\infty^0} \upharpoonright \cM^{(0)} |\delta_2$ is amenable to $\cM^{(0)}|\delta_2$ by Lemma \ref{lem:TailSigmaiInM} below (see also the comment after Lemma \ref{lem:TailSigma0InM}). Moreover, $\cM^{(0)}|\delta_2 = \cM^{(1)}|\delta_2$ since Clause \eqref{item:strongext} in Definition \ref{def:translationprocedure} only adds extenders with critical point $\kappa_1$ above stage $\delta_2$ of the construction.
  Working in $\cM^{(1)}|\delta_2$, compare $\cP_\infty^0$ (via $\Sigma_{\cP_\infty^0}$) against the construction $(\cL_\eta \mid \eta \leq \eta^*)$ used to define $\cR_0 = \cL_\gamma$. By stationarity and universality of fully backgrounded constructions (cf. \cite[Lemmas 2.11 and 2.13]{Sa15}), $\cP_\infty^0$ iterates to a level of the construction of $\cR_0$. As $\cR_0$ is the level where the construction reaches the first full Woodin cardinal, $\cP_\infty^0$ in fact iterates to $\cR_0$, as desired. This also immediately gives that $\Sigma_{\cR_0} \upharpoonright \cM^{(1)}|\delta_2$ is amenable to $\cM^{(1)}|\delta_2$.
\end{proof}

In fact, the proof straightforwardly generalizes to show the following lemma.

\begin{lemma}\label{lem:MinftyIteratesToConstructioni=1}
   Let $j<\omega$ and suppose $(\mathsf{IH})_{j}$, i.e., $\cM^{(j)}|\delta_{j+1}$ is well-defined. Then $\cR_0$ as constructed in $\cM^{(j)}|\delta_{j+1}$ above some $\delta^* > \delta_j$ is a $\Sigma_{\cP_\infty^{0}}$-iterate of $\cP_\infty^{0}$ in $\cM^{(j)}|\delta_{j+1}$ and if $\Sigma_{\cR_{0}}$ denotes the corresponding tail strategy, $\Sigma_{\cR_0} \upharpoonright \cM^{(j)}|\delta_{j+1}$ is amenable to $\cM^{(j)}|\delta_{j+1}$, i.e., for any $X \in \cM^{(j)}|\delta_{j+1}$, $\Sigma_{\cR_0}\upharpoonright X \in \cM^{(j)}|\delta_{j+1}$.
\end{lemma}

\begin{lemma}\label{lem:LpatkappaiinM}
    Suppose $(\mathsf{IH})_1$. Then the function \[ X \mapsto (Lp^{\Sigma_{\cR_0},\infty}(X))^\cW \] for $X \in \cM^{(1)}|\kappa_1$ with $\cR_0 \in X$ is definable over $\cM^{(1)}|\delta_2$.
\end{lemma}
\begin{proof}
    The proof is analogous to the proof of Lemma \ref{lem:Lpatkappa0inM}, replacing fully backgrounded constructions by fully backgrounded constructions relative to $\Sigma_{\cR_0}$ and using the comparison techniques from \cite{Sa15}. $\Sigma_{\cR_0} \upharpoonright \cM^{(1)}|\delta_2$ is amenable to $\cM^{(1)}|\delta_2$ by Lemma \ref{lem:MinftyIteratesToConstructioni=1onecase}.
\end{proof}



 For the next level $i=2$, suppose $(\mathsf{IH})_2$, work in $\cM^{(2)}|\delta_3$, and let $\delta^* \leq \kappa_2$ be a $\kappa_2$-weak cutpoint. Note that by $(\mathsf{IH})_2$, $\cM^{(2)}|\delta_3$ is well-defined and $\cM^{(0)}$ and $\cM^{(1)}$ are well-defined proper class models. Moreover, $\cM^{(2)}|\delta_3 = \cM^{(1)}|\delta_3$.
We obtain $\cN_{\delta^*}^+$ as follows: Start by performing a fully backgrounded construction above $\delta^*$ inside $\cM^{(2)}|\delta_3$ until the level $\cR_0$ where it reaches the first full Woodin cardinal. 
By Lemma \ref{lem:MinftyIteratesToConstructioni=1} we get that $\cR_0$ is a $\Sigma_{\cP_\infty^{0}}$-iterate of $\cP_\infty^{0}$ and $\Sigma_{\cR_0} \upharpoonright \cM^{(2)}|\delta_3$ is amenable to $\cM^{(2)}|\delta_3$. Therefore, we can add $\Sigma_{\cR_0}$ to $\cR_0$ without projecting across $\cR_0 \cap \Ord$. More precisely, if it exists, we let $\cR_1$ be the result of a fully backgrounded construction relative to $\Sigma_{\cR_0}$ on top of $\cR_0$ in $\cM^{(2)}|\delta_3$ up to a level $\cR_1$ where the construction reaches the second full Woodin cardinal. That means the construction will not project across $\cR_0 \cap \Ord$ and \[ (Lp_\omega^{\Sigma_{\cR_0},\infty}(\cR_1))^\cW \models \text{``$\cR_1 \cap \Ord$ is Woodin''.} \]
 All extenders in $\cM^{(2)}|\delta_3$ that serve as backgrounds for extenders in $\cR_0$ or $\cR_1$ are themselves backgrounded by extenders in $\cW$. This follows from the fact that the fully backgrounded constructions giving rise to $\cR_0$ and $\cR_1$ are performed in $\cM^{(2)}|\delta_3$, only use extenders above $\delta^* > \delta_2$, and stop before $\kappa_2$ by  Lemma \ref{lem:Ndeltasucceeds}.
 
 We will show in Lemma \ref{lem:MinftyIteratesToConstructioni>1} that $\cR_1$ is a $\Sigma_{\cP_\infty^{1}}$-iterate of $\cP_\infty^1$ and, if $\Sigma_{\cR_1}$ denotes the corresponding tail strategy of $\Sigma_{\cP_\infty^{1}}$ (and ultimately of $\Sigma_1$),  $\Sigma_{\cR_1} \upharpoonright \cM^{(2)}|\delta_3$ is amenable to $\cM^{(2)}|\delta_3$ (as we are inductively supposing $(\mathsf{IH})_2$). Moreover, as strategies with branch condensation are absorbed (see \cite[Lemma 2.15]{Sa15}), $\Sigma_{\cR_1} \upharpoonright \cM^{(2)}|\delta_3$ agrees with the strategy $\cR_1$ inherits from $\cM^{(2)}$ (and ultimately from $\cW$) as the result of a fully backgrounded construction. So we can let $\cR_2$ be the result of adding $\Sigma_{\cR_1}$ to $\cR_1$. That means, we again construct further relative to $\Sigma_{\cR_1}$ and, if it exists, obtain $\cR_2$ as a $\Sigma_{\cR_1}$-premouse above $\cR_1$ such that the construction does not project across $\cR_1 \cap \Ord$ and 
 \[ (Lp_\omega^{\Sigma_{\cR_1},\infty}(\cR_2))^\cW \models \text{``$\cR_2 \cap \Ord$ is Woodin''.} \]
Finally, we let $\cN_{\delta^*} = \cR_2$ and \[ \cN_{\delta^*}^+ = (Lp_\omega^{\Sigma_{\cR_1}, \infty}(\cR_2))^\cW. \]

As in Lemma \ref{lem:Ndeltasucceeds} the construction succeeds. An argument as in Lemma \ref{lem:LpatkappaiinM} shows that $\cN_{\delta^*}^+$ is definable in $\cM$. The next lemma provides a version of Lemma \ref{lem:MinftyIteratesToConstructioni=1} for $\cR_1$. The same argument inductively shows the result for constructions above $\kappa_i$-weak cutpoints $\delta^*$ for any $i>1$.

\begin{lemma}\label{lem:MinftyIteratesToConstructioni>1}
 Let $j<\omega$ and suppose $(\mathsf{IH})_{j}$, i.e., $\cM^{(j)}|\delta_{j+1}$ is well-defined. Then $\cR_1$ as constructed in $\cM^{(j)}|\delta_{j+1}$ above some $\delta^* > \delta_j$ is a $\Sigma_{\cP_\infty^{1}}$-iterate of $\cP_\infty^{1}$ in $\cM^{(j)}|\delta_{j+1}$ and if $\Sigma_{\cR_{1}}$ denotes the corresponding tail strategy, $\Sigma_{\cR_1} \upharpoonright \cM^{(j)}|\delta_{j+1}$ is amenable to $\cM^{(j)}|\delta_{j+1}$, i.e., for any $X \in \cM^{(j)}|\delta_{j+1}$, $\Sigma_{\cR_1}\upharpoonright X \in \cM^{(j)}|\delta_{j+1}$.
\end{lemma}
\begin{proof}
  Suppose inductively that $\cP_\infty^{1} \in \cM^{(j)}|\delta_{j+1}$ and using Lemmas \ref{lem:TailSigma0InM} and \ref{lem:TailSigmaiInM}, $\Sigma_{\cP_\infty^{1}} \upharpoonright \cM^{(j)}|\delta_{j+1}$ is amenable to $\cM^{(j)}|\delta_{j+1}$.
  Working in $\cM^{(j)}|\delta_{j+1}$, compare $\cP_\infty^{1}$ (via $\Sigma_{\cP_\infty^{1}}$) against the construction of $\cR_{1}$ (as the result of a fully backgrounded construction in $\cM^{(j)}|\delta_{j+1}$) using Sargsyans's comparison method  in \cite{Sa15}. 
  $\cR_{1}$ is the result of a fully backgrounded construction using extenders with critical point above $\delta^* > \delta_j > \cP_\infty^{1} \cap \Ord$, so (by stationarity and universality of fully backgrounded constructions) $\cR_1$ does not move in the comparison. Hence, $\cP_\infty^{1}$ iterates to an initial segment of $\cR_{1}$ and their strategies line up. By fullness, $\cP_\infty^{1}$ in fact iterates to $\cR_1$. It follows immediately that not only $\Sigma_{\cP_\infty^{1}} \upharpoonright \cM^{(j)}|\delta_{j+1}$ is amenable to $\cM^{(j)}|\delta_{j+1}$ but also its tail strategy $\Sigma_{\cR_1} \upharpoonright \cM^{(j)}|\delta_{j+1}$ is amenable to $\cM^{(j)}|\delta_{j+1}$. 
\end{proof}

Now returning to the general context for an arbitrary $i>0$, suppose $(\mathsf{IH})_i$. We have as in the case $i = 0$ that each $\cN_{\delta^*}^+$ for a $\kappa_i$-weak cutpoint $\delta^*$ is $\kappa_i$-suitable and we can prove analogues of Lemmas \ref{lem:Ndelta*suitable} and \ref{lem:iterationNdelta*Ndelta**}. As in the case $i=0$, we let $\cN^* \in \cM^{(i)}|\delta_{i+1}$ be the result of making $\cM^{(i)}|\kappa_i$ generically generic over an iterate of $\cN_{\delta_0^*}^+$, where $\delta_0^*$ is the least $\kappa_i$-weak cutpoint in $\cM^{(i)}|\kappa_i$ above $\delta_i$.
Recall that as $\cW$ is a translatable structure, there is a set of term relations $\vec\tau^{(i)} \subseteq \cN^*$ such that    
    whenever $G \subseteq \Col(\omega, |\cN^*|)$ is generic over $\cW$ and $(\tau_k^{(i)} \mid k < \omega)$ is a fixed generic enumeration of $\vec\tau^{(i)}$, then $(\Sigma_{\cN^*})^G$ is strongly guided by $(\tau_k^{(i)} \mid k < \omega)$. For $l \leq \omega$ let $f_{l}^{(i)} \in F(\kappa_i, \Sigma)$ be the function given by \[ f_{l}^{(i)}(\cQ) = \bigoplus_{k<l} (\tau_k^{(i)})^{\cQ} \] for $\cQ \in S(\kappa_i, \Sigma)$ and $(\tau_k^{(i)})^{\cQ}$ the restriction of $\tau_k^{(i)}$ to $\cQ$ as defined in the case $i=0$ at the end of Subsection \ref{subsec:i=0}. Now let
    \begin{align*}
    \cI^{(i)} = \{ (\cQ,l) \mid \cQ \text{ is $\kappa_i$-suitable, strongly $f^{(i)}_{l}$-iterable, and countable in $(\cM^{(i)})^{\Col(\omega,{<}\kappa_i)}$} \}
    \end{align*}
    and for $(\cQ,l), (\cQ',l') \in \cI^{(i)}$ let 
    \[ (\cQ,l) \leq_{\cI^{(i)}} (\cQ',l') \text{ iff } \cQ' \text{ is a $(\kappa_i, \Sigma)$-correct iterate of } \cQ \text{ and } l \leq l'.  \]
    Let $\cF^{(i)}$ be the system of models $H^{\cQ}_{f_l^{(i)}}$ indexed by $(\cQ,l) \in (\cI^{(i)}, \leq_{\cI^{(i)}})$ together with the associated embeddings. Note that $\cF^{(i)}$ is direct and internal to $\cM^{(i)}$. Let $(\cP_\infty^i)^*$ be its direct limit with direct limit embedding $\pi_{(\cQ,l),\infty} \colon H^{\cQ}_{f_l^{(i)}} \rightarrow (\cP_\infty^i)^*$.
As $\kappa_i$ is the least ${<}\delta_{i+1}$-strong cardinal above $\delta_i$ in $\cM^{(i)}$, it is not hard to see that there are cofinally many $\kappa_i$-weak cutpoints below $\kappa_i$ and the $\cN_{\delta^*}^+$'s appear cofinally in $\cF^{(i)}$.
The following lemma can be shown analogous to Lemma \ref{lem:quasidirectlimitkappa0}. The second equality in Lemma \ref{lem:quasidirectlimitkappai} is a consequence of fullness preservation.

\begin{lemma}\label{lem:quasidirectlimitkappai}
  $(\cP_\infty^i)^* = \cP_\infty^{i} = (Lp_\omega^{\Sigma_{(\cP_\infty^i)^-}, \infty})^{\cW}(\cP_\infty^{i,-})$.
\end{lemma}

This implies that we can define $\cP_\infty^{i}$ inside $\cM^{(i)}$. As in the case $i=0$, $\cP_\infty^{i}$ is an iterate of $\cP^i$ and we write $\pi_\infty^i \colon \cP^i \rightarrow \cP_\infty^i$ for the iteration embedding in $\cW$. We are left with showing that the corresponding tail iteration strategy $\Sigma_{\cP_\infty^{i}}$ can be internalized to $\cM^{(i)}$ as well. 
As in the previous section for $i=0$, ideas from the proof of the successor case in the proof of \cite[Theorem 6.5]{Sa15} (cf. \cite[Lemma 2.15]{Sa17}) can be used to argue that $\Sigma_{\cP_\infty^{i}}$ is in $\cM^{(i)}$. However some new ideas are necessary to deal with the non-fully backgrounded extenders in $\cM^{(i)}$. We outline the modifications that yield the internalization of $\Sigma_{\cP_\infty^{i}}$ to $\cM^{(i)}$ also in the case $i>0$.

\begin{lemma}\label{lem:TailSigmaiInM}
  Let $k\geq 0$ and suppose $(\mathsf{IH})_k$. Then, for any $i \leq k$, $\Sigma_{\cP_\infty^{i}} \upharpoonright \cM^{(k)}|\delta_{k+1}$ is amenable to $\cM^{(k)}|\delta_{k+1}$.
\end{lemma}

\begin{proof}
As we are no longer working in the result of a fully backgrounded construction to obtain $\cP_\infty^i$ but in the result of an $\cM$-construction, we sketch how to modify the proof of the successor case in the proof of \cite[Theorem 6.5]{Sa15} to obtain Lemma \ref{lem:TailSigmaiInM}. 
We assume that the reader is familiar with the proof of \cite[Theorem 6.5]{Sa15} and, for notational simplicity, restrict to the case for normal iteration trees. Sargsyan's argument uses universality in the sense of \cite[Lemma 2.13]{Sa15} or \cite[Lemma 11.1]{St08dm} to, for example, show that $\cQ$-structures for iteration trees according to the iteration strategy in question can be obtained via fully backgrounded constructions relative to segments of that strategy. This sort of universality goes back to Mitchell and Schindler's proof of universality for $K^c$ in \cite{MSch04}. To apply Sargsyan's argument in our setting we, for example, need the following claim. Before we can state the claim, recall the following notation from \cite{Sa15} (see also Subsection \ref{subsec:hodpm}): For a hod premouse $\cP$ with finitely many Woodin cardinals $\delta_0, \dots, \delta_n$, write, if $n>0$, $\cP^- = \cP(n-1)$. We will also use this notation for initial segments of hod premice and, for example, write $(\cP_\infty^{i,-})^- = (\cP_\infty^{i})^-$. Note that $(\cP_\infty^i)^-$ is strictly shorter than $\cP_\infty^{i,-} = \cP_\infty^i | \delta^{\cP_\infty^i}$. Recall that $(\mathsf{IH})_i$ implies that $\cM^{(i)}|\delta_{i+1}$ is well-defined.

\begin{claim}\label{cl:TailSigmaiInM_Qstructures}
  Let $\cT$ be an iteration tree on $\cP_\infty^{i}$ in $\cM^{(i)}|\delta_{i+1}$ of limit length and let $b = \Sigma_{\cP_\infty^{i}}(\cT)$. Suppose $\cT$ is not entirely based on $(\cP_\infty^{i})^-$ and $\cQ(b,\cT)$ exists. Let $\cS$ be the node in $\cT$ such that the fragment of $\cT$ from $\cP_\infty^{i}$-to-$\cS$ is based on $(\cP_\infty^{i})^-$ and the rest of $\cT$ is above $\cS^-$.
  Let $\cN$ be the result of a fully backgrounded construction relative to $(\Sigma_{(\cP_\infty^{i})^-})_{\cS^-}$ above $\cM(\cT)$ in $\cM^{(i)}|\delta_{i+1}$.
  Then 
  \begin{enumerate}
      \item the comparison of $\cQ(b,\cT)$ and $\cN$ in $\cM^{(i)}|\delta_{i+1}$ is successful, and \label{it:1comparisonsuccessful}
      \item $\cN$ is universal when compared with $\cQ(b,\cT)$, i.e., it wins the comparison. \label{it:2universal}
  \end{enumerate}
\end{claim}
\begin{proof}
Note that \eqref{it:2universal} follows from \eqref{it:1comparisonsuccessful} by universality of fully backgrounded constructions as in \cite[Lemma 2.13]{Sa15} and \cite[Lemma 11.1]{St08dm}. 
In particular, it suffices for universality of $\cN$ against $\cQ(b,\cT)$ that the comparison of $\cQ(b,\cT)$ and $\cN$ is successful in $\cM^{(i)}|\delta_{i+1}$. So we are left with showing \eqref{it:1comparisonsuccessful}.

Suppose toward a contradiction that the comparison is not successful and let $\cU_0$ and $\cU_1$ be iteration trees on $\cQ(b,\cT)$ and $\cN$ witnessing this. Take a countable substructure of a sufficiently large initial segment of $\cM$ reflecting this situation and write $\bar\cQ(b,\cT), \bar\cN, \bar\cU_0$, and $\bar\cU_1$ for the images of $\cQ(b,\cT), \cN, \cU_0$, and $\cU_1$ under the collapse. By the iterability proof for fully backgrounded constructions, the relevant iterates of $\bar\cN$ can be realized into $\cM$. 
Moreover, the $\bar\cN$-side of the comparison provides $\cQ$-structures for the $\bar\cQ(b,\cT)$-side, cf., for example \cite[Fact 3.6]{Sa15}. Therefore, the comparison of $\bar\cQ(b,\cT)$ and $\bar\cN$ is successful in $\cM$.
\end{proof}
A similar claim can be shown for $\cT \in \cM^{(k)}|\delta_{k+1}$ for the case $i<k$ in the statement of Lemma \ref{lem:TailSigmaiInM}, using that by $(\mathsf{IH})_k$, $\cM^{(k)}|\delta_{k+1}$ is a well-defined model.
Using the notation in Claim \ref{cl:TailSigmaiInM_Qstructures}, the cases in the proof of \cite[Theorem 6.5]{Sa15} when the tree $\cT$ is entirely based on $(\cP_\infty^{i})^-$ and when there is a fatal drop in the rest of $\cT$ above $\cS^-$ are dealt with similarly as the case where $\cQ(b,\cT)$ exists. So we are left with showing that branches for maximal trees are also computed correctly (this corresponds to Clause $2(c)$ in the proof of \cite[Theorem 6.5]{Sa15}).

We define the iteration strategy $\Lambda$ as in the proof of \cite[Theorem 6.5]{Sa15}. This means that given $\cT \in \dom(\Lambda)$ such that $\cT$ is not covered by one of the cases discussed above, we say $\Lambda(\cT)$ is defined and let $\Lambda(\cT) = b$ iff there is an extender $E$ on the sequence of $\cM^{(k)}|\delta_{k+1}$ with critical point $\kappa_i$ such that $\cT \in \cM^{(k)} | \lh(E)$ and there is an embedding \[ \sigma \colon \cM_b^\cT \rightarrow \pi_E(\cP_\infty^{i}) \] such that $\pi_E \upharpoonright \cP_\infty^{i} = \sigma \circ \pi_b^\cT$. As in the proof of \cite[Theorem 6.5]{Sa15} we have that $\cT \in \dom(\Lambda)$ if and only if $\cT \in \dom(\Sigma_{\cP_\infty^{i}})$. We are left with showing the following claim.

\begin{claim}
    Let $\cT \in \dom(\Lambda) \cap \dom(\Sigma_{\cP_\infty^{i}})$ be as above. Then
    \begin{enumerate}
        \item $\Lambda(\cT)$ is defined for $\cT \in \cM^{(k)} | \delta_{k+1}$, and \label{item:Lambdadefined}
        \item if it is defined, $\Lambda(\cT) = \Sigma_{\cP_\infty^{i}}(\cT)$.\label{item:LambdaequaltoSigma} 
    \end{enumerate}
\end{claim}
\begin{proof}
    For \eqref{item:LambdaequaltoSigma}, let $\cT \in \dom(\Lambda) \cap \dom(\Sigma_{\cP_\infty^{i}})$ be such that $\Lambda(\cT)$ is defined. So let $\Lambda(\cT) = b$ and let $E$ be an extender on the sequence of $\cM^{(k)}|\delta_{k+1}$ with critical point $\kappa_i$ and let $\sigma \colon \cM_b^\cT \rightarrow \pi_E(\cP_\infty^{i})$ be an embedding such that $\pi_E \upharpoonright \cP_\infty^{i} = \sigma \circ \pi_b^\cT$. In case the extender $E$ is added to $\cM^{(k)}|\delta_{k+1}$ by the fully backgrounded part of the construction (which will be the case if $\lh(E) < \delta_{i+1}$) the argument in the proof of \cite[Theorem 6.5]{Sa15} applies. So assume that is not the case and $E$ is generically countably complete.
    
    Let $\hat\cW$ be a countable elementary substructure of a sufficiently large initial segment $\cW|\Omega$ of $\cW$. Let $\tau \colon \hat\cW \rightarrow \cW|\Omega$ and suppose that all relevant objects are in the range of $\tau$. For notational simplicity write $\cM^{(k)}|\delta_{k+1}$ for the result of the $\cM$-construction up to $\delta_{k+1}$ where Clause \eqref{item:strongext} is restricted to extenders with critical point in $\{ \kappa_j \mid j<k \}$ in $\cW|\Omega$ and let $\hat\cM$ be the transitive collapse of $\cM^{(k)}|\delta_{k+1}$, i.e., the result of the same type of $\cM$-construction in $\hat\cW$. Say $\tau(\hat\kappa_i, \hat E, \hat\cP_\infty^i, \hat\Sigma^{i,\infty}) = (\kappa_i, E, \cP_\infty^i, \Sigma_{\cP_\infty^{i}})$.
    
    When comparing $\hat\cP_\infty^{i}$ via $\hat\Sigma^{i,\infty}$ and $\cP_\infty^{i}$ via $\Sigma_{\cP_\infty^{i}}$ inside $\cW$, $\hat\cP_\infty^{i}$ iterates to $\cP_\infty^{i}$. Moreover, \[ \tau \upharpoonright \hat\cP_\infty^{i} = \pi^{\hat\Sigma^{i,\infty}}_{\hat\cP_\infty^{i}, \tau(\hat\cP_\infty^{i})} \] as there are term relations that are strongly guiding $\hat{\Sigma}^{i,\infty}$. We have that $E = \tau(\hat E)$ is generically countably complete, so by Fact \ref{fact:realization} there is a realization map $r \colon \Ult(\hat \cM, \hat E) \rightarrow \cM$ such that $\tau \upharpoonright \hat\cM = r \circ \pi_{\hat E}$ where $\pi_{\hat E} \colon \hat \cM \rightarrow \Ult(\hat \cM, \hat E)$ denotes the ultrapower embedding. Therefore, $\pi_{\hat E}$ respects $\hat\Sigma^{i,\infty}$ and $\pi_{\hat E} \upharpoonright \hat\cP_\infty^{i} = \pi^{\hat\Sigma^{i,\infty}}_{\hat\cP_\infty^{i}, \pi_{\hat E}(\hat\cP_\infty^{i})}.$ By elementarity of $\tau$, \[ \pi_{E} \upharpoonright \cP_\infty^{i} = \pi^{\Sigma_{\cP_\infty^{i}}}_{\cP_\infty^{i}, \pi_{E}(\cP_\infty^{i})} \] and hence $\pi^{\Sigma_{\cP_\infty^i}}_{\cP_\infty^{i}, \pi_{E}(\cP_\infty^{i})} = \sigma \circ \pi_b^\cT$. As $\Sigma_{\cP_\infty^{i}}$ has branch condensation, $b = \Sigma_{\cP_\infty^{i}}(\cT)$, as desired.

    For \eqref{item:Lambdadefined}, given $\cT \in \dom(\Lambda) \cap \dom(\Sigma_{\cP_\infty^{i}})$ as above and $b = \Sigma_{\cP_\infty^{i}}(\cT)$ we aim to find $E$ and $\sigma$ such that $\pi_E \upharpoonright \cP_\infty^{i} = \sigma \circ \pi_b^\cT$. Let $r \leq k$ with $r \geq i+1$ be least such that $\cT \in \cM^{(r)} | \delta_{r}$.
    We start with arguing that $\cM_b^\cT$ is in fact an element of $\cM^{(r)}$. Let $\cN$ be the result of a fully backgrounded construction relative to $(\Sigma_{\cM_b^\cT})^-$ above $\cM(\cT)$ in $\cM^{(r)}|\delta_r$ using extenders with critical points above $\kappa_{r-1}$. Recall that inductively $(\Sigma_{\cM_b^\cT})^- \upharpoonright (\cM^{(r)}|\delta_r)$ is amenable to $\cM^{(r)}|\delta_r$. As the extenders backgrounding the construction of $\cN$ in $\cM^{(r)}|\delta_r$ have critical points above $\kappa_{r-1}$ they are in fact fully backgrounded in $\cW$. Therefore, when comparing $\cM_b^\cT$ against the construction of $\cN$ in $\cW$, by universality and stationarity of fully backgrounded constructions \cite[Lemmas 2.11 and 2.13]{Sa15}, $\cM_b^\cT$ iterates to a level of the construction of $\cN$. By fullness preservation of the iteration from $\cP^i$ via $\cP_\infty^i$ to $\cM_b^\cT$, the iteration from $\cM_b^\cT$ to the construction of $\cN$ would have to drop if it is non-trivial since $\cM_b^\cT$ and $\cN$ agree up to $\delta(\cT)$. Hence, $\cM_b^\cT \unlhd \cN$ and, in particular, $\cM_b^\cT \in \cM^{(r)}|\delta_r$.

    In case $i =k$, i.e., $\cT \in \cM^{(i)} | \delta_{i+1}$, recall that $\kappa_i$ is ${<}\delta_{i+1}$-strong in $\cM^{(i)}$ and pick $E$ with critical point $\kappa_i$ and $\cT \in \cM^{(i)}|\lh(E)$ on the sequence of $\cM^{(i)}$ added by the fully backgrounded part of the $\cM$-construction. In this case, $r = i+1$. For the case $k > i$, we suppose inductively that $(\mathsf{IH})_k$ holds, so $\kappa_i$ is strong in $\cM^{(k)}|\delta_{k+1}$. Therefore, we can fix an extender $E$ on the sequence of $\cM^{(k)}|\delta_{k+1}$ with $\lh(E) > \delta_{r}$. In both cases, we pick $E$ such that, in addition to the requirements above, $\cM^{(k)}|\delta_{k+1}$ and $\pi_E(\cM^{(k)}|\delta_{k+1})$ agree on the result $\cN$ of the hybrid fully backgrounded construction up to $\cM_b^\cT$. As most of what follows will be the same for both cases, $k=i$ and $k>i$, we will deal with them parallelly.


    Finally, we argue that there is an embedding $\sigma$ such that $\pi_E \upharpoonright \cP_\infty^{i} = \sigma \circ \pi_b^\cT$. Recall the  internal direct limit system giving rise to $\cP_\infty^i$ in $\cM^{(k)}|\delta_{k+1}$.
    $\cM_b^\cT$ is iterable via a tail of $\Sigma$ in $\cW$ and hence strongly $f_{l}^{(i)}$-iterable in $\cM^{(k)}|\delta_{k+1}$ for every $l<\omega$. Hence $(\cM_b^\cT, l) \in \cI^{(i)}$ for any $l<\omega$. Recall that $\cM_b^\cT$ can be obtained as an initial segment of a hybrid backgrounded construction in $\cM^{(r)}|\delta_r$. By agreement of $\cM^{(r)}|\delta_{r+1}$ and $\pi_E(\cM^{(r)}|\delta_{r+1})$ up to the strength of $E$ and hence on $\cM_b^\cT \unlhd \cN$, $\cM_b^\cT$ is strongly $\pi_E(f_{l}^{(i)})$-iterable in $\pi_E(\cM^{(r)}|\delta_{r+1})$, where $\pi_E(f_{l}^{(i)})$ is the canonical function obtained from the images of the terms $\tau_n^{(i)}$, $n<l$, under $\pi_E$.
    Let $\sigma_l$ be the corresponding direct limit embedding into $\pi_E(\cP_\infty^i)$. Then $\sigma = \bigcup_{l<\omega} \sigma_l \colon \cM_b^\cT \rightarrow \pi_E(\cP_\infty^i)$ is an embedding in $\cM^{(r)}|\delta_{r+1}$. We are left with showing that $\pi_E \upharpoonright \cP_\infty^{i} = \sigma \circ \pi_b^\cT$.

    In the case that $E$ is added by the fully backgrounded part of the $\cM$-construction, this follows as in the proof of the successor case in the proof of \cite[Theorem 6.5]{Sa15} (see also Lemma \ref{lem:backgroundedextenderscertified} below). In the case that $E$ is added by the generically countably complete part of the construction, i.e., in the case $k>i$, the argument for \eqref{item:LambdaequaltoSigma} above shows that \[ \pi_{E} \upharpoonright \cP_\infty^{i} = \pi^{\Sigma_{\cP_\infty^{i}}}_{\cP_\infty^{i}, \pi_{E}(\cP_\infty^{i})} = \sigma \circ \pi_b^\cT \] by our choice of $\sigma$ and $b = \Sigma_{\cP_\infty^i}(\cT)$.
\end{proof}
This finishes the proof of Lemma \ref{lem:TailSigmaiInM}.
\end{proof}

\subsection{Certified extenders}\label{sec:certifiedextenders}

The following definition is crucial for the rest of this paper as it describes how we want to identify extenders that can serve as witnesses for the strongness of $\kappa_i$ for $i<\omega$.

\begin{definition}
  Let $F$ be an extender with critical point $\kappa_i$ for some $i<\omega$. Then we say $F$ is \emph{certified} (in $\cW$) iff \[ \pi_F \upharpoonright \cP_\infty^{i,-} = \pi^{\Sigma^{i,\infty}}_{\cP_\infty^{i,-}, \pi_F(\cP_\infty^{i,-})}, \] for $\Sigma^{i,\infty} = \Sigma_{\cP_\infty^{i,-}}$ the tail strategy of $\Sigma_i$ as described in the previous subsection. Here $\cP_\infty^{i,-}$ denotes $\cP_\infty^i | \delta_\infty^i$, where $\delta_\infty^i$ denotes the largest Woodin cardinal of the $\kappa_i$-suitable premouse $\cP_\infty^i$.
\end{definition}

\begin{remark*}
  An extender $F$ on $\cM$ with critical point $\kappa_i$ is uniquely determined by $\Ult(\cW,F)|\pi_F(\kappa_i)$ together with $\pi_F \upharpoonright \cP_\infty^{i,-}$ as $\cP_\infty^{i,-} \cap \Ord = (\kappa_i^+)^{\cM^{(i)}}$ by the following lemma. Recall that under $(\mathsf{IH})_i$, $\cM^{(i)}|\delta_{i+1} = \cM |\delta_{i+1}$ is well-defined and well-founded and, in particular, $(\kappa_i^+)^{\cM^{(i)}|\delta_{i+1}} = (\kappa_i^+)^{\cM^{(i)}} = (\kappa_i^+)^\cM$. So we will write $(\kappa_i^+)^\cM$ in this situation to simplify the notation. 
\end{remark*}

\begin{lemma}
  Let $i \geq 0$ and suppose $(\mathsf{IH})_i$. Then \[ \cP_\infty^{i,-} \cap \Ord = (\kappa_i^+)^\cM. \]
\end{lemma}
\begin{proof}
  The proof uses similar ideas as \cite[Lemma 2.7(b)]{SaSch18} (cf., also \cite[Lemma 3.38(2)]{StW16}). For the easier inequality $\cP_\infty^{i,-} \cap \Ord = \delta_\infty^i \leq (\kappa_i^+)^\cM$, let $\alpha < \delta_\infty^i$. Then there is some $(\cQ,l) \in \cI^{(i)}$ and some $\beta < \gamma_{f_l^{(i)}}^\cQ$ such that $\pi_{(\cQ,l),\infty}(\beta) = \alpha$. Consider \[ B = \{ (\gamma, \cQ') \mid (\cQ,l) \leq_{\cI^{(i)}} (\cQ',l) \text{ and } \gamma < \pi_{(\cQ,l),(\cQ',l)}(\beta) \} \] and the function $g \colon B \rightarrow \Ord$ given by \[ g(\gamma,\cQ') = \pi_{(\cQ',l),\infty}(\gamma). \]
  Then $\alpha \subseteq \ran(g)$ and $B,g \in \cM$. As $B$ is of size at most $\kappa_i$ in $\cM$, $\alpha < (\kappa_i^+)^\cM$.

  For the other inequality, $(\kappa_i^+)^\cM \leq \delta_\infty^i$, let $\alpha < (\kappa_i^+)^\cM$. Let $f \colon \kappa_i \rightarrow \alpha$ be bijective, $f \in \cM$. In particular, $f \in \cW$ and we can fix a $\Sigma_1$ Skolem term $\tau$ and a finite set of indiscernibles $s$ such that for each $\eta < \alpha$ there is some $\beta < \kappa_i$  such that \[ \tau^{\cW|\max(s)}[\beta, s^-] = \eta, \] where $s^- = s \setminus \{\max(s)\}$. So fix some $\eta$ and $\beta$ such that $\eta = \tau^{\cW|\max(s)}[\beta, s^-] < \alpha$. 

  Recall that, in $\cW$, $\cN_{\delta_0^*}^+$ is a $\Sigma$-iterate of $\cP^i$, where $\delta_0^*$ is the least $\kappa_i$-weak cutpoint in $\cM|\kappa_i$ above $\delta_i$. Hence, we can lift the iteration $\cP^i \rightarrow \cN_{\delta_0^*}^+$ to an iteration $\cW \rightarrow \cN_0^\cW$.
  By standard arguments from $\HOD$ computations, see, for example, \cite{StW16}, and properties of the iteration strategy $\Sigma$, there is an iterate $\cN^\cW$ of $\cN_0^\cW$ such that, letting $\cN \unlhd \cN^\cW$ be the image of $\cN_{\delta_0^*}^+ \unlhd \cN_0^\cW$ under this iteration,
  \begin{enumerate}
      \item $(\cN,l) \in \cI^{(i)}$ for any $l<\omega$,\label{eq:1inkappapluscomputation}
      \item there is some $\zeta < \kappa_i$ that is a cutpoint cardinal in $\cW$, $\delta^\cN = (\zeta^+)^\cW$ and $\cW|\zeta$ is generic over $\cN$ for the $\delta^\cN$ generator version of the extender algebra $\mathbb{B}^\cN$,\label{eq:extenderalgebrainkappapluscomputation}
      \item $\beta$ is below the least measurable cardinal of $\cN$, 
      \item $\beta < \gamma_{f_l^{(i)}}^\cN$ for any sufficiently large $l<\omega$, and
      \item whenever $j \colon \cN^\cW \rightarrow \cN'$ is a $\Sigma$-iterate of $\cN^\cW$ acting only on $\cN$, $j(s) = s$ and $j(\eta) = \eta$.\label{eq:5inkappapluscomputation}
  \end{enumerate}
  By $\cP$-constructions\footnote{See \cite[Lemma 3.42]{Sa15} for $\cP$-constructions in this hybrid setting. For ordinary mice, $\cP$-constructions were introduced by John Steel. This is why Sargsyan calls them $S$-constructions in \cite{Sa15}.} in $\cW$, \eqref{eq:extenderalgebrainkappapluscomputation} implies that there is a proper class $\Sigma$-mouse $\cP(\cN)$ over $\cN$ such that $\cW|\zeta$ is generic over $\cP(\cN)$ and $\cP(\cN)[\cW|\zeta] = \cW$.

  Write $\nu^\cN$ for the least measurable cardinal of $\cN$ and let \[ P^\cN = \{ \xi \mid \exists \gamma < \nu^\cN \exists p \in \mathbb{B}^\cN \, p \Vdash^{\cP(\cN)} \tau^{\cP(\cN)[\dot{G}]|\max(\check{s})}[\check{\gamma}, \check{s}^-] = \check{\xi} \}. \]
  Then $\eta \in P^\cN$ and, by the $\delta^\cN$-c.c. of $\mathbb{B}^\cN$, the order type of $P^\cN$ is less than $\delta^\cN$. Therefore, we can let $\mu_\eta^\cN$ be the unique $\mu < \delta^\cN$ such that $\eta$ is the $\mu$-th element of $P^\cN$. Moreover, let $\mu_\eta = \pi_{(\cN,l), \infty}(\mu_\eta^\cN) < \delta_\infty^i$ for $l$ large enough such that this is defined.

  \begin{claim}
      The function $\alpha \rightarrow \delta_\infty^i$ given by $\eta \mapsto \mu_\eta$ is well-defined, i.e., independent of the choice of $\cN$, and order-preserving.
  \end{claim}
  \begin{proof}
      Let $\eta \leq \xi < \alpha$, say $\eta = \tau^{\cW|\max(s)}[\beta, s^-]$ and $\xi = \tau^{\cW|\max(s)}[\epsilon, s^-]$. Let $\cN$ and $\cQ$ be $\Sigma$-iterates of $\cP^i$ with properties \eqref{eq:1inkappapluscomputation} - \eqref{eq:5inkappapluscomputation} as above such that $\eta \in P^\cN$ and $\xi \in P^\cQ$. Let $\cR$ be a common $\Sigma$-iterate of $\cN$ and $\cQ$ with $(\cR,l) \in \cI^{(i)}$ for any $l<\omega$ and write $\cR^\cW$ for the common iterate of $\cN^\cW$ and $\cQ^\cW$, respectively, when the iterations are lifted from $\cP^i$ to $\cW$. Let $j_\cN \colon \cN^\cW \rightarrow \cR^\cW$ and $j_\cQ \colon \cQ^\cW \rightarrow \cR^\cW$ be the corresponding iteration embeddings. Then $j_\cN(\beta,\eta,s) = (\beta, \eta, s)$ and $j_\cQ(\epsilon, \xi, s) = (\epsilon, \xi, s)$. Therefore, $j_\cN(\mu_\eta^\cN) = \mu_\eta^\cR$ and $j_\cQ(\mu_\xi^\cQ) = \mu_\xi^\cR$. Hence, \[ \eta \leq \xi \text{ iff } \mu_\eta^\cR \leq \mu_\xi^\cR \text{ iff } \mu_\eta \leq \mu_\xi. \]
  \end{proof}
    This yields $\alpha < \delta_\infty^i$, as desired.
\end{proof}

\subsection{Preservation of certification under hulls}

We argue that taking hulls preserves the certification of extenders.

\begin{lemma}\label{lem:characterizationhull}
  Let $i < \omega$ and suppose $(\mathsf{IH})_i$, so $\cM^{(i)}|\delta_{i+1}$ is well-defined. Let $\bar\cW$ be the transitive collapse of a hull of $\cW | \Omega$ for some sufficiently large $\Omega$ such that $\cW|(\delta_i^{+\omega})^\cW \subseteq \bar\cW$. 
  Let \[\pi \colon \bar\cW \rightarrow \cW\] be the uncollapse embedding, $\bar\cM$ the collapse of $\cM^{(i)}|\Omega$, and $(\bar\kappa_j, \bar\delta_j \mid j<\omega)$ the preimages of $(\kappa_j, \delta_j \mid j < \omega)$.\footnote{So, in particular, $\bar\kappa_j = \kappa_j$ and $\bar\delta_j = \delta_j$ for $j < i$ and $\bar\delta_i = \delta_i$.} Then every extender on the sequence of $\bar\cM$ with critical point $\kappa_i$ is certified in $\bar\cW$.
\end{lemma}


\begin{proof}
  Let $E$ be an extender on the sequence of $\bar\cM$ with critical point $\kappa_i$ and let $F = \pi(E)$ be its preimage under the collapse. We aim to show that \[ \pi_E \upharpoonright \bar\cP_\infty^{i,-} = \pi^{\bar\Sigma^{i,\infty}}_{\bar\cP_\infty^{i,-},\pi_E(\bar\cP_\infty^{i,-})}, \] for $\bar\cP_\infty^{i,-}$ the collapse of $\cP_\infty^{i,-}$ and $\bar\Sigma^{i,\infty} = (\Sigma_i)_{\bar\cP_\infty^{i,-}}$. As $F$ is an extender on the sequence of $\cM^{(i)}$ with critical point $\kappa_i$, it is certified in $\cW$. Therefore, we can find a stack of iteration trees $\vec\cU$ on $\cP_\infty^{i,-}$ with last model $\cM^{\vec\cU}_\infty = \pi_F(\cP_\infty^{i,-})$ such that \[ \pi^{\Sigma^{i,\infty}}_{\cP_\infty^{i,-}, \pi_F(\cP_\infty^{i,-})} = \pi_\infty^{\vec\cU}, \] where $\pi_\infty^{\vec\cU} \colon \cP_\infty^{i,-} \rightarrow \cM_\infty^{\vec\cU}$ denotes the iteration embedding along $\vec\cU$. In particular, $\vec\cU$ is by $\Sigma^{i,\infty}$. Let $\vec\cT$ be the stack of iteration trees on $\cP^{i,-}$ with last model $\cP_\infty^{i,-}$ according to $\Sigma^i$.
  Let $\bar\cT$ and $\bar\cU$ be the images of $\vec\cT$ and $\vec\cU$ under the collapse, i.e., $\pi(\bar\cT, \bar\cU) = (\vec\cT, \vec\cU)$. Then by hull condensation (cf. Definition \ref{def:hullcondensation}) applied to the stacks $\vec\cT ^\frown \vec\cU$ and $\bar\cT ^\frown \bar\cU$ on $\cP^{i,-}$, the stack $\bar\cT ^\frown \bar\cU$ is by $\Sigma^{i}$ as well. In particular, $\bar\cU$ is by $\bar\Sigma^{i,\infty}$.
  By elementarity of $\pi$, we have $\pi_E(\bar\cP_\infty^{i,-}) = \cM_{\infty}^{\bar\cU}$. So \[ \pi_{\infty}^{\bar\cU} \colon \bar\cP_\infty^{i,-} \rightarrow \pi_E(\bar\cP_\infty^{i,-}). \] As $\bar\cU$ is by $\bar\Sigma^{i,\infty}$ this implies \[ \pi^{\bar\Sigma^{i,\infty}}_{\bar\cP_\infty^{i,-},\pi_E(\bar\cP_\infty^{i,-})} = \pi_{\infty}^{\bar\cU}. \]
  Note that as $F$ is certified, \[ \pi_F \pwimg \cP_\infty^{i,-} = \rng(\pi_\infty^{\vec\cU}). \] By elementarity of $\pi$, this implies \[ \pi_E \pwimg \bar\cP_\infty^{i,-} = \rng(\pi_{\infty}^{\bar\cU}). \]
  The following claim finishes the proof of Lemma \ref{lem:characterizationhull}.
  \begin{claim}
    \[ \pi_E \upharpoonright \bar\cP_\infty^{i,-} = \pi_{\infty}^{\bar\cU}. \] 
  \end{claim}
  \begin{proof}
    Suppose toward a contradiction that $\pi_E \upharpoonright \bar\cP_\infty^{i,-} \neq \pi_{\infty}^{\bar\cU}$ and let $a$ be minimal (in the canonical well-order of $\bar\cP_\infty^{i,-}$) with $\pi_E(a) \neq \pi_{\infty}^{\bar\cU}(a)$. Say $\pi_E(a) > \pi_{\infty}^{\bar\cU}(a)$. Then $\pi_{\infty}^{\bar\cU}(a) \not\in \rng(\pi_E \upharpoonright \bar\cP_\infty^{i,-}) = \pi_E \pwimg \bar\cP_\infty^{i,-}$ as there is no element of $\bar\cP_\infty^{i,-}$ that $\pi_E$ can map to $\pi_{\infty}^{\bar\cU}(a)$ while respecting the well-order of $\bar\cP_\infty^{i,-}$. This contradicts $\pi_E \pwimg \bar\cP_\infty^{i,-} = \rng(\pi_{\infty}^{\bar\cU})$. The argument in the case $\pi_E(a) < \pi_{\infty}^{\bar\cU}(a)$ is similar.
  \end{proof}
\end{proof}

\subsection{All extenders on the $\cM$-sequence are certified}

Similarly as in \cite[Lemma 2.25]{Sa17} we have that fully backgrounded extenders with critical point $\kappa_i$ for some $i \geq 0$ are certified.

\begin{lemma}\label{lem:backgroundedextenderscertified}
  Let $i \geq 0$, suppose $(\mathsf{IH})_i$ and let $E$ be an extender on the sequence of $\cM^{(i)}$ with critical point $\kappa_i$ that is added by the fully backgrounded part of the construction of $\cM^{(i)}$. Then $E$ is certified in $\cW$.
\end{lemma}

\begin{proof}
  Let $F$ be an extender with critical point $\kappa_i$ that gets added by the fully backgrounded part of the construction of $\cM^{(i)}$. Say $\cN_{\xi+1} = (J_\alpha^{\vec E}, \in, \vec E, F)$ and let $F^*$ be the background extender of $F$. Let $\cU$ be an iteration tree on $\cP_\infty^{i,-}$ of limit length such that if we let $b = \Sigma^{i,\infty}(\cU)$, \[ \pi^{\Sigma^{i,\infty}}_{\cP_\infty^{i,-}, \pi_F(\cP_\infty^{i,-})} = \pi_b^\cU. \] Moreover, let $\cU^*$ be an iteration tree on $\cP_\infty^{i,-}$ of limit length such that if we let $b^* = \Sigma^{i,\infty}(\cU^*)$, \[ \pi^{\Sigma^{i,\infty}}_{\cP_\infty^{i,-}, \pi_{F^*}(\cP_\infty^{i,-})} = \pi_{b^*}^{\cU^*}. \]
  Let \[ k \colon \Ult(\cN_{\xi+1}, F) \rightarrow \Ult(\cW,F^*) \] be the canonical realization map. Then $k \upharpoonright \pi_F(\cP_\infty^{i,-}) \colon \pi_F(\cP_\infty^{i,-}) \rightarrow \pi_{F^*}(\cP_\infty^{i,-})$ makes the diagram in Figure \ref{fig:backgroundedextenderscertified} commute.

  \begin{figure}[htb]
      \begin{tikzpicture}
        \draw[->] (0,-1.55) -- node[above] {$\pi_{F^*}$} (3.3,0) node[right]
        {$\pi_{F^*}(\cP_\infty^{i,-})$}; 

        \draw[->] (0,-1.75) node[left] {$\cP_\infty^{i,-}$}-- node[below] {$\pi_F$} (3.3,-1.75) node[right] {$\pi_F(\cP_\infty^{i,-})$};
       
        \draw[->] (4.1,-1.4) -- node[right] {$k$} (4.1,-0.3);
      \end{tikzpicture}
      \caption{Realization into the ultrapower by the background extender.}\label{fig:backgroundedextenderscertified}
    \end{figure}

    As $F^*$ is an extender on the sequence of $\cW$ and $(\Sigma_i)^g$ is strongly guided by term relations for all $g \subseteq \Col(\omega, \cP^i)$, \[ \pi_{F^*} \upharpoonright \cP_\infty^{i,-} = \pi_{b^*}^{\cU^*}. \] We now define a cofinal branch $c$ through $\cU$ such that $k(\cM_c^\cU) = \cM_{b^*}^{\cU^*}$ and $\pi_{F} \upharpoonright \cP_\infty^{i,-} = \pi_{c}^{\cU}$.

    Let \[ X = \{ \alpha < \lh(\cU) \mid \exists \beta \, \alpha = k(\pi_F(\beta)) \text{ and } b_\alpha^* \in \pi_{F^*}(\cN_{\xi+1}) \}. \]
    For $\alpha \in X$ let $c_\alpha \in \Ult(\cN_{\xi+1},F)$ be such that \[ k(c_\alpha) = b_\alpha^*. \]
Let $c = \bigcup_{\alpha \in X} c_\alpha$. Then $c$ is as desired and the following claim finishes the proof of Lemma \ref{lem:backgroundedextenderscertified}.

  \begin{claim*}
    $b=c$.
  \end{claim*}
  \begin{proof}
    Both $\cU^\frown c$ and $\cU^*{}^\frown b^*$ are iteration trees on $\cP_\infty^{i,-}$ and $\cU^\frown c$ is a hull of $\cU^*{}^\frown b^*$. As $\cU^*{}^\frown b^*$ is according to $\Sigma^{i,\infty}$, by hull condensation $\cU^\frown c$ is according to $\Sigma^{i,\infty}$ as well and hence \[ c = \Sigma^{i,\infty}(\cU) = b. \]
  \end{proof}
\end{proof}

Now we argue that generically countably complete extenders are certified as well.

\begin{lemma}\label{lem:genctblycompleteextenderscertified}
   Let $i \geq 0$, suppose $(\mathsf{IH})_i$ and let $E$ be an extender on the sequence of $\cM$ with critical point $\kappa_i$ that is added by the generically countably complete part of the construction of $\cM$.\footnote{We are not assuming that $\cM$ is a proper class model here. The construction of $\cM$ might break down at some stage after $E$ was added.} Then $E$ is certified in $\cW$.
 \end{lemma}

 \begin{proof}
   Let $E$ be an extender on the sequence of $\cM$ with critical point $\kappa_i$ that is added by the generically countably complete part of the construction of $\cM$. Let $\hat\cW$ be an elementary substructure of size ${<}\kappa_i$ of a sufficiently large initial segment $\cW|\Omega$ of $\cW$ with $\cW|(\delta_i^{+\omega})^\cW \subseteq \hat\cW$. Let \[ \sigma \colon \hat\cW \rightarrow \cW|\Omega \] and suppose that all relevant objects are in the range of $\sigma$. For notational simplicity write $\cM$ for the result of the $\cM$-construction in $\cW|\Omega$ and let $\hat\cM$ be the transitive collapse of $\cM$, i.e., the result of the $\cM$-construction in $\hat\cW$. Say $\sigma(\hat E, \hat\cP_\infty^i, \hat\Sigma^{i,\infty}) = (E, \cP_\infty^i, \Sigma^{i,\infty})$.

Then, by hull condensation, $\hat\cP_\infty^i$ is an iterate of $\cP^i$ via $\Sigma^i$. Hence, $\hat\cP_\infty^i$ is an element of the direct limit system giving rise to $\cP_\infty^i$. So when comparing $\hat\cP_\infty^{i}$ via $\hat\Sigma^{i,\infty}$ and $\cP_\infty^{i}$ via $\Sigma^{i,\infty}$ inside $\cW$, $\hat\cP_\infty^{i}$ iterates to $\cP_\infty^{i}$. Moreover, \[ \sigma \upharpoonright \hat\cP_\infty^{i,-} = \pi^{\hat\Sigma^{i,\infty}}_{\hat\cP_\infty^{i,-}, \sigma(\hat\cP_\infty^{i,-})} \] as there are term relations that are strongly guiding $\hat{\Sigma}^{i,\infty}$. We have that $E = \sigma(\hat E)$ is generically countably complete, so by Fact \ref{fact:realization} there is a realization map \[ \tau \colon \Ult(\hat \cM, \hat E) \rightarrow \cM \] such that $\sigma \upharpoonright \hat\cM = \tau \circ \pi_{\hat E}$ where $\pi_{\hat E} \colon \hat \cM \rightarrow \Ult(\hat \cM, \hat E)$ denotes the ultrapower embedding. Therefore, $\pi_{\hat E}$ respects $\hat\Sigma^{i,\infty}$ and \[ \pi_{\hat E} \upharpoonright \hat\cP_\infty^{i,-} = \pi^{\hat\Sigma^{i,\infty}}_{\hat\cP_\infty^{i,-}, \pi_{\hat E}(\hat\cP_\infty^{i,-})}. \]
So $\hat E$ is certified via $\hat\Sigma^{i,\infty}$. By elementarity of $\sigma$, this implies that $E$ is certified via $\Sigma^{i,\infty}$, as desired.
\end{proof}

\subsection{Characterization via operators}\label{subsec:operators}

Now we argue that we can also characterize extenders using operators in order to obtain the extenders in a simply definable way (e.g., in $M_1^{\#,\Sigma_i}$ of the current level).
In what follows, we define universal mice at $\kappa_i$ and $\kappa_i$-good universal mice as in \cite{Sa17}. The arguments in \cite{Sa17} generalize to $\kappa_i$, $i>0$, by considering $\cM$-constructions instead of fully backgrounded constructions. In particular, \cite[Lemmas 2.15 and 2.16]{Sa17} generalize and we can show that good universal mice exist. We give the details in what follows.

As usual, we write \[ \cO_\eta^\cQ = \bigcup \{ \cQ||\beta \mid \eta \leq \beta \leq \cQ \cap \Ord, \rho_\omega(\cQ||\beta) \leq \eta \text{ and } \eta \text{ is a cutpoint in } \cQ \} \] for premice $\cQ$ and ordinals $\eta < \cQ \cap \Ord$.

\begin{definition}
  Let $i<\omega$, suppose $(\mathsf{IH})_i$ and let $X \in \cW|\delta_{i+1}$ be a premouse such that $\cM^{(i)}|\delta_i \unlhd X$. We say $\cQ \subseteq \cW|\delta_{i+1}$ is a \emph{universal mouse over $X$ at $\kappa_i$} if $\cQ$ is an $X$-premouse such that $\cQ$ is $(\delta_{i+1}, \delta_{i+1})$-iterable with respect to the extenders on its sequence with critical point above $X \cap \Ord$, for every strong cutpoint $\eta \in (X \cap \Ord, \delta_{i+1})$ in $\cQ$, $\cO_\eta^\cQ = Lp^\infty(\cQ|\eta)$ and for a stationary set of $\xi < \delta_{i+1}$, $(\xi^+)^\cQ = (\xi^+)^\cW$. Moreover, in case $i>0$, $\cQ$ is \emph{$\kappa_j$-good} for all $j<i$, i.e., whenever $\cR^\cQ$ (defined inductively as described below) is a $\Sigma_j$-iterate of $\cP^j$ and whenever $g$ is generic over $\cW$ for a partial order in $\cQ$, $(\Sigma_j^g)_{\cR^\cQ} \upharpoonright \cQ[g] \in L_1(\cQ[g])$.
\end{definition}

Before we can define good universal mice (and the object $\cR^\cQ$ used in the previous definition), we need some additional notation. Let $i<\omega$, suppose $(\mathsf{IH})_i$ and let $X \in \cW|\delta_{i+1}$ be a premouse with $X \unrhd \cM^{(i)}|\delta_i$. For any $X$-premouse $\cQ \subseteq \cW|\delta_{i+1}$, let $\kappa^\cQ$ denote the least ${<}\delta_{i+1}$-strong cardinal of $\cQ$, if it exists. The reader might wish to recall the definition of being $\kappa_i$-suitable (see Definition \ref{def:kappaisuitable}) at this point. For any universal mouse $\cQ \subseteq \cW|\delta_{i+1}$ over $X$ at $\kappa_i$, consider, analogous to Section \ref{subsec:i>0}, the directed system given by 
 \begin{eqnarray*}
    \cI^{(i)}_{\cQ} = \{ (\cR,l) \mid \cR \text{ is $\kappa_i$-suitable, strongly $f^{(i)}_{l}$-iterable in $\cQ$, and} \\ \text{countable in $\cQ^{\Col(\omega,{<}\kappa^\cQ)}$} \}.
    \end{eqnarray*}
Note that for an $\cR$ as in the definition of $\cI^{(i)}_\cQ$ the notion of $\kappa_i$-suitability is well-defined inside $\cQ$ as $\cQ$ is sufficiently closed under the operation $X \mapsto Lp^{\Sigma_{\cR},\infty}(X)$ by the proof of Lemma \ref{lem:LpatkappaiinM}.
Write $\cP_\infty^\cQ$ for the corresponding direct limit and $\delta_\infty^\cQ$ for the largest Woodin cardinal in $\cP_\infty^\cQ$. Note that $\cP_\infty^\cQ \in \cQ$ and let \[ \cR^\cQ = \cP_\infty^\cQ | ((\delta_\infty^\cQ)^{+\omega})^{\cP_\infty^\cQ}. \]

\begin{definition}
  Let $i<\omega$, suppose $(\mathsf{IH})_i$ and let $X \in \cW|\delta_{i+1}$ be a premouse such that $\cM^{(i)}|\delta_i \unlhd X$. We say a universal mouse $\cQ$ over $X$ at $\kappa_i$ is a \emph{$\kappa_i$-good universal mouse over $X$} iff it has a $(\delta_{i+1}, \delta_{i+1})$-iteration strategy $\Phi$ with respect to extenders with critical point above $X \cap \Ord$ such that whenever $i \colon \cQ \rightarrow \cN$ is an embedding obtained from an iteration according to $\Phi$, $\cN$ is a universal mouse, $\cR^\cN$ is a $\Sigma_i$-iterate of $\cP^i$, and whenever $g$ is generic over $\cW$ for a partial order in $\cN$, \[ (\Sigma^g_i)_{\cR^\cN} \upharpoonright \cN[g] \in L_1(\cN[g]). \]
\end{definition}

The existence of good universal mice follows as in \cite[Lemmas 2.15 and 2.16]{Sa17}, cf. also the proof of Lemma \ref{lem:TailSigmaiInM}.

\begin{lemma}\label{lem:gooduniversalmiceexist}
  Let $i<\omega$, suppose $(\mathsf{IH})_i$ and let $X \unrhd \cM^{(i)}|\delta_i$ be a premouse with $X \in \cW|\delta_{i+1}$. Fix some ordinal $\eta$ with $X \cap \Ord < \eta < \delta_{i+1}$ and let $\cQ = \cQ^{X,\eta}$ be the result of an $\cM$-construction above $X$ inside $\cW|\delta_{i+1}$ where the fully backgrounded part only uses extenders with critical point above $\eta$. Let $\Phi$ be the iteration strategy for $\cQ$ induced by the background strategy\footnote{That means, $\Phi$ only acts on the extenders on the sequence of $\cQ$ that are backgrounded.} and let $\cN$ be an iterate of $\cQ$ (above $X$) according to $\Phi$ such that the iteration embedding $i \colon \cQ \rightarrow \cN$ exists. Then $\cN$ is a $\kappa_i$-good universal mouse over $X$.
\end{lemma}

\begin{proof}
    $\cR^\cN$ is a $\Sigma_i$-iterate of $\cP^i$ and $\cN$ is $\kappa_i$-good, i.e., whenever $g$ is generic over $\cW$ for a partial order in $\cN$, $(\Sigma^g_i)_{\cR^\cN} \upharpoonright \cN[g] \in L_1(\cN[g])$, by the argument for Lemma \ref{lem:TailSigmaiInM}. So we only sketch the argument that $\cN$ is universal and show that for a stationary set of $\xi < \delta_{i+1}$, $(\xi^+)^\cN = (\xi^+)^\cW$. Let $\cW^*$ be the iterate of $\cW$ obtained by applying the $\cP^i$-to-$\cR^\cN$ iteration to $\cW$. Let $\cW^{**}$ be the result of a fully backgrounded construction relative to $\Sigma_{\cR^\cN}$ inside $\cN$. We can now apply universality of background constructions to the comparison of $\cW^*$ and $\cW^{**}$ in $\cW$ and finish the proof as in \cite[Lemma 2.16]{Sa17}.
\end{proof}

Now we can define operators $S_i$ for each $i>0$ analogous to Sargsyan's $S$-operator in \cite{Sa17} that goes back to Steel \cite{St08dm} and show that the operators are independent of the choice of the good universal mice. This yields another way of characterizing extenders and we will argue that it gives rise to the same extenders as the certification defined in Section \ref{sec:certifiedextenders}.

Let $i<\omega$ and let $\cQ_i$ be a $\kappa_i$-good universal mouse. Write $\pi^{\cQ_i} \colon \cP^i \rightarrow \cR^{\cQ_i}$ for the iteration embedding according to $\Sigma_i$. Let $T_i \subseteq \cP^i$ be the collection of all terms that $\Sigma_i$ guides correctly. Write $\pi^{\cQ_i}_\tau = \pi^{\cQ_i} \upharpoonright H_\tau^{\cP^i}$ for $\tau \in T_i$ and \[\cQ_i^+ = (\cQ_i | ((\kappa^{\cQ_i})^+)^{\cQ_i}, \in, \{\pi_\tau^{\cQ_i}\}_{\tau \in T_i}).\]
For any $X \in \cW | \delta_{i+1}$ and some $\kappa_i$-good universal mouse $\cQ_i$ over $X$ let \[ S_i^{\cQ_i}(X) = Hull_1^{\cQ_i^+}(X). \]
As in \cite[Lemma 2.19]{Sa17} the operator $S_i^{\cQ_i}$ does not depend on the choice of $\cQ_i$, so we can simplify the notation and write $S_i(X)$ for $S_i^{\cQ_i}(X)$ for some $\kappa_i$-good universal mouse $\cQ_i$ over $X$.

\begin{definition}
  Let $i<\omega$ and suppose $(\mathsf{IH})_i$. We say an $\cM^{(i)}|\delta_{i+1}$-extender $E$ with critical point $\kappa_i$ is \emph{characterized via $S_i$} iff\footnote{Note that $\cM^{(i)}|\xi = \Ult(\cM^{(i)},E)|\xi$.} \[ Hull_1^{S_i(\cM^{(i)}|\xi)}(\cM^{(i)}|\kappa_i) = S_i(\cM^{(i)}|\kappa_i) \] where $\xi \geq \kappa_i^+$ is the least $i$-weak cutpoint\footnote{Recall that an ordinal $\alpha$ is an $i$-weak cutpoint above $\kappa_i$ if all extenders overlapping $\alpha$ have critical point $\kappa_j$ for some $j<i$.} of $\Ult(\cM^{(i)},E)$ above $\kappa_i$, and if $\alpha$ is 
  the successor of $\xi$ in $\Ult(\cM^{(i)},E)$, then $E$ is the $(\kappa_i,\alpha)$-extender\footnote{Recall that by our indexing convention, $E$ with critical point $\kappa_i$ will be indexed at $\alpha$.} derived from the uncollapse map \[ k \colon S_i(\cM^{(i)}|\kappa_i) \rightarrow S_i(\cM^{(i)}|\xi). \] Write $\xi(E)$ for the ordinal $\xi$ defined above.
\end{definition}

\begin{lemma}\label{lem:certifiedextendersaregivenbyoperators}
  Let $i< \omega$ and suppose $(\mathsf{IH})_i$. An $\cM^{(i)}|\delta_{i+1}$-extender $E$ with critical point $\kappa_i$ is certified iff it is characterized via the operator $S_i$.
\end{lemma}
\begin{proof}
  The forward direction, i.e., that any certified extender with critical point $\kappa_i$ is characterized via the operator $S_i$, follows from the fact that if an extender $E$ is according to $(\Sigma_i)_{\cM_\infty^{i,-}}$ it clearly moves term relations correctly.

  For the other direction, suppose $E$ is the extender derived from $k \colon S_i(\cM^{(i)}|\kappa_i) \rightarrow S_i(\cM^{(i)}|\xi(E))$. Let $\cQ$ be a $\kappa_i$-good universal mouse over $\cM^{(i)}|\kappa_i$ and let $\cQ^*$ be a $\kappa_i$-good universal mouse over $\cM^{(i)}|\xi(E)$. Consider the corresponding direct limits $\cR^\cQ$ and $\cR^{\cQ^*}$ of suitable premice. Both, $\cR^\cQ$ and $\cR^{\cQ^*}$, are $\Sigma_i$-iterates of $\cP^i$. Moreover,
    $\cR^\cQ$ is a point in the direct limit system giving rise to $\cR^{\cQ^*}$. 
  In fact, $\cR^{\cQ^*}$ is a $(\Sigma_i)_{\cR^\cQ}$-iterate of $\cR^\cQ$. Note that $\cM^{(i)}|\delta_{i+1}$ is a $\kappa_i$-good universal mouse over $\cM^{(i)}|\kappa_i$ with $\cR^{\cM^{(i)}|\delta_{i+1}} = \cP_\infty^i$. So we can pick $\cQ = \cM^{(i)}|\delta_{i+1}$ and the iteration from $\cR^\cQ = \cP_\infty^i$ to $\cR^{\cQ^*}$ is according to $(\Sigma_i)_{\cR^\cQ}$. Therefore, we obtain $\pi_E \upharpoonright \cP_\infty^{i,-} = \pi^{\Sigma^{i,\infty}}_{\cP_\infty^{i,-}, \pi_E(\cP_\infty^{i,-})}$, as desired.
\end{proof}



An important advantage of the characterization of extenders via operators is that it can be used to prove that the extenders that can be characterized this way are easily definable. This is made precise in the following lemma (cf. \cite[Lemma 2.24 and Lemma 2.25]{Sa17}). 

\begin{lemma}\label{lem:extenderdefinable}
  Let $i<\omega$ and suppose $(\mathsf{IH})_i$. Let $\xi$ be a cardinal in $\cM^{(i)}$ with $(\kappa_i^+)^{\cM^{(i)}} < \xi < \delta_{i+1}$. Then \[ S_i(\cM^{(i)}|\xi) \in M_1^{\#,\Sigma_i}(\cM^{(i)}|\xi). \]
  Moreover, if $E$ is an $\cM^{(i)}|\delta_{i+1}$-extender with critical point $\kappa_i$ that is characterized via the operator $S_i$, then \[ E \in M_1^{\#,\Sigma_i}(\Ult(\cM^{(i)},E)|\zeta), \] where $\zeta$ is the least $i$-weak cutpoint above $\kappa_i$ in $\Ult(\cM^{(i)},E)$.
\end{lemma}

The proof of Lemma \ref{lem:extenderdefinable} is analogous to the proofs of Lemmas 2.24 and 2.25 in \cite{Sa17}, so we leave it to the reader. The idea of proof is that, considering $\Ult(\cM^{(i)},E)$ for $E$ an extender on the sequence of $\cM^{(i)}$ such that $\xi$ is a cutpoint in $\Ult(\cM^{(i)},E)$, we can make $\cM^{(i)}|\xi$ generic for the extender algebra of $\cM'$ for some iteration $i \colon \cM^{(i)} \rightarrow \cM'$. Using $\cP$-constructions we can conclude that $S_i(\cM^{(i)}|\xi) \in M_1^{\#,\Sigma_i}(\cM^{(i)}|\xi).$ Note that $S_i(\cM^{(i)}|\kappa_i) = Hull_1^{S_i(\Ult(\cM^{(i)},E)|\zeta)}(\cM^{(i)}|\kappa_i)$ and let $\sigma \colon S_i(\cM^{(i)}|\kappa_i) \rightarrow S_i(\Ult(\cM^{(i)},E)|\zeta)$ be the uncollapse map. The argument for the first part of Lemma \ref{lem:extenderdefinable} sketched above also shows $\sigma \in M_1^{\#,\Sigma_i}(\Ult(\cM^{(i)},E)|\zeta)$. This yields $E \in M_1^{\#,\Sigma_i}(\Ult(\cM^{(i)},E)|\zeta)$ as $E$ can be derived from $\sigma$.

\section{Proving that it works}

\subsection{Iterability} \label{subsec:iterability}

We now prove that hulls of $\cM$ are iterable. The argument will combine the standard lifting to the background model for fully backgrounded constructions from \cite{MS94} with the almost linearity of the iteration tree when considering extenders with critical point (the image of) some $\kappa_i$, $i<\omega$, and index above (the image of) $\delta_{i+1}$ (cf. \cite{Sch02}).

Our concept of an almost linear iteration tree is a generalization of the one in \cite{Sch02} as we distinguish two different types of extenders -- with and without a background -- and allow arbitrary iteration trees for backgrounded extenders. We make this more precise now and use the notation from \cite{St10} for iteration trees $\cT$ and the associated tree order $T$. Lemma \ref{lem:iterationtreealmostlinear} will show that in our setting all normal iteration trees are almost linear on strongness witnesses.

\begin{definition}
  Let $\cN$ be the result of an $\cM$-construction in some translatable structure or an iterate of the result of such a construction. Write $(\kappa_i^\cN, \delta_i^\cN \mid i < \omega)$ for the sequence of the specific cardinals $\kappa_i$ and $\delta_i$, $i<\omega$, chosen in the definition of the $\cM$-construction. 
  \begin{enumerate}
      \item An extender $E$ on the $\cN$-sequence is called \emph{strongness witness} or \emph{non-backgrounded} iff the critical point of $E$ is $\kappa_i^\cN$ for some $i<\omega$ and the length of $E$ is above $\delta_{i+1}^\cN$. 
      Otherwise we say $E$ is \emph{backgrounded}.
      \item An iteration tree $\cT$ on $\cN$ is called \emph{almost linear on strongness witnesses} iff 
    \begin{enumerate}
          \item for all $\alpha+1 < \lh(\cT)$ such that $E_\alpha^\cT$ is a strongness witness we have $\pred_T^*(\alpha+1) \in [0,\alpha]_T$, 
          where $\pred_T^*(\alpha+1)$ denotes the unique ordinal $\beta+1$ such that $\cM_{\beta+1}^\cT$ is obtained as an ultrapower by a strongness witness $E_\beta^\cT$ and $i_{\beta+1,\pred_T(\alpha+1)} \colon \cM_\beta^{*,\cT} \rightarrow \cM_{\pred_T(\alpha+1)}^\cT$ is an iteration via backgrounded extenders,
          and  
          \item any $\alpha<\lh(\cT)$ has only finitely many $T$-successors $\beta+1$ such that $E_\beta^\cT$ 
          is a strongness witness.
    \end{enumerate}
  \end{enumerate}
\end{definition}

The following lemma uses similar ideas as the proof of \cite[Lemma 2.4]{Sch02}.

\begin{lemma}\label{lem:iterationtreealmostlinear}
  Let $\cN$ be the result of an $\cM$-construction in some translatable structure. Then any normal iteration tree $\cT$ on $\cN$ is almost linear on strongness witnesses.
\end{lemma}
\begin{proof}
    Analogous to the proof of \cite[Lemma 2.4]{Sch02}, we start with the following claim.
    
    \begin{claim}
      Let $\cU$ be a normal iteration tree on $\cN$ of successor length $\lambda+1$ and let $i+1 \in (0,\lambda]_U$ be such that $E = E_i^\cU$ is a strongness witness with critical point $\kappa_E$ and length $\lambda_E$. Then there is no strongness witness $F$ on the sequence of $\cM_\lambda^\cU$ with critical point $\kappa_F$ and length $\lambda_F$ such that \[ \kappa_E \leq \kappa_F < \lambda_E \leq \lambda_F. \] 
    \end{claim}
    \begin{proof}
    Suppose the claim does not hold and let $F$ be a strongness witness on the sequence of $\cM_\lambda^\cU$ with $\kappa_E \leq \kappa_F < \lambda_E \leq \lambda_F$.  
    Then it is not hard to see that $\rho_\omega(J_{\lambda_F}^{\cM_\lambda^\cU}) \geq \lambda_E$, cf. the subclaim in the proof of \cite[Lemma 2.4, Claim 1]{Sch02}.
      This implies that $F|\xi \in J_{\lambda_F}^{\cM_\lambda^\cU}$ for all $\kappa_F < \xi < \lambda_E$. In particular, $\kappa_F$ is $\lambda_E$-strong in $\cM_\lambda^\cU$. 
      Therefore, $\kappa_F$ is $\lambda_E$-strong in $\cM_{i+1}^\cU$.
    This yields a contradiction as $\cM_{i+1}^\cU$ is an ultrapower by $E$, so there is no extender with critical point between $\kappa_E$ and $\lambda_E$ that overlaps $\lambda_E$ in $\cM_{i+1}^\cU$.
    \end{proof}
    
    The claim implies that for every normal iteration tree $\cU$ on $\cN$ of length $\lambda+2$ and every $i<j+1<\lambda+1$ with $j+1 \in (i,\lambda]_U$ and $i = U-\pred(j+1) = U-\pred(\lambda+1)$ such that $E_j^\cU$ and $E_\lambda^\cU$ are both strongness witnesses \[ \crit(E_j^\cU) > \crit(E_\lambda^\cU). \] Using this observation, a short argument as in the proof of \cite[Lemma 2.4]{Sch02} yields that any normal iteration tree $\cT$ on $\cN$ is almost linear on strongness witnesses.
\end{proof}



The next lemma proves $(\mathsf{IH}.1)_{i+1}$ from $(\mathsf{IH})_{i}$.

\begin{lemma}\label{lem:iterability}
  Let $i < \omega$, suppose $(\mathsf{IH})_{i}$ and let \[\pi \colon \bar\cW \rightarrow \cW|\Omega \] for some sufficiently large\footnote{Note that we have not shown at this point that the construction of $\cM^{(i+1)}$ converges. In case $\cM^{(i+1)} \cap \Ord < \Ord$ pick $\Omega > \cM^{(i+1)} \cap \Ord$ and let $\cM^{(i+1)}|\Omega = \cM^{(i+1)}$.} $\Omega$ be such that letting $\crit(\pi) = \nu$, $\cW|\kappa_i \subseteq \bar\cW$, $(\kappa_j, \delta_j \mid j < \omega) \in \rng(\pi)$, the sequence of models of the construction of $\cM^{(i+1)}|\Omega$ is in the range of $\pi$, $\nu \in (\kappa_i, \delta_{i+1})$ is an inaccessible cardinal in $\cW$, and $\pi(\nu) = \delta_{i+1}$. Let $\bar\cM$ be the collapse of $\cM^{(i+1)}|\Omega$.
  Then $\bar\cM$ is $(\delta_{i+1},\delta_{i+1})$-iterable with respect to extenders with critical point above $\kappa_i$.
\end{lemma}

\begin{proof}
  Let $i<\omega$ and $\pi \colon \bar\cW \rightarrow \cW|\Omega$ be as in the statement of the lemma. Let $(\bar\kappa_j, \bar\delta_j \mid j<\omega)$ be the preimages of $(\kappa_j, \delta_j \mid j < \omega)$ under $\pi$ (so, in particular, $\bar\kappa_j = \kappa_j$ and $\bar\delta_j = \delta_j$ for $j \leq i$). 
  For notational simplicity we write $\cM$ and $\cW$ instead of $\cM^{(i+1)}|\Omega$ and $\cW|\Omega$ as it will be clear from the context if we refer to the proper class model or the restriction to $\Omega$ and suppressing the superscript ${(i+1)}$ simplifies the notation. We recursively describe an iteration strategy $\Lambda$ for $\bar\cM$ acting on iteration trees $\cT$ of length ${<}\delta_{i+1}$ which use extenders with critical point above $\kappa_i$. The recursive construction will depend on what type of extenders get used -- backgrounded extenders or non-backgrounded strongness witnesses. For notational simplicity, we describe the iteration strategy $\Lambda$ only for normal trees $\cT$ and leave it to the reader to generalize the construction to stacks of normal trees. Simultaneously to the construction of $\Lambda$, we define a \emph{realization of $\cT$ into a backgrounded iteration of $\cM$}. That means, for every normal iteration tree $\cT$ on $\bar\cM$ we recursively define a stack of iteration trees $\vec\cU = (\cU_\alpha \mid \alpha < \gamma)$ on $\cM$ and an embedding \[ k \colon \lh(\cT) \rightarrow \gamma \times \bigcup_{\alpha < \gamma}\lh(\cU_\alpha) \] such that $\vec\cU$ only uses fully backgrounded extenders in $\cM$ and its iterates and for every ordinal $\xi < \lh(\cT)$ if $k(\xi) = (\alpha, \eta)$ there is an elementary embedding \[ \iota_\xi \colon \cM_\xi^\cT \rightarrow \cM_{\eta}^{\cU_\alpha}. \] Here, $\pi \upharpoonright \bar\cM = \iota_0 \colon \bar\cM \rightarrow \cM$ and the other embeddings $\iota_\xi$ for $\xi > 0$ will be obtained as copy or realization maps. Note that for example in the special case that $\cT$ only uses non-backgrounded extenders, $\gamma = 1$, $\lh(\cU_0) = 1$ and all $\iota_\xi$ for $0<\xi<\lh(\cT)$ are realization maps of $\cM_\xi^\cT$ into $\cM = \cM_0^{\cU_0}$.
  

  Now we turn to the recursive definition of $\Lambda$ together with realizations into a backgrounded iteration of $\cM$. Let $\cT$ be a normal iteration tree of limit length $\lambda < \delta_{i+1}$ on $\bar\cM$ according to $\Lambda$ and recall that $\cT$ is almost linear on strongness witnesses by Lemma \ref{lem:iterationtreealmostlinear}. We distinguish the following cases.

  \begin{case}
    $\cT$ only uses fully backgrounded extenders.
  \end{case}
  
  The cofinal well-founded branches through $\cT \upharpoonright \eta$ for limit ordinals $\eta \leq \lambda$ that get picked by $\Lambda$ are given by the usual strategy of lifting the iteration tree $\cT$ to a tree on the background model $\bar\cW$ via the backgrounds of the extenders used in $\cT$ and choosing the cofinal branches according to the pullback strategy for $\bar\cW$ of the iteration strategy of $\cW$. This yields canonical cofinal well-founded branches through $\cT \upharpoonright \eta$ for all limit ordinals $\eta \leq \lambda$. Let $\iota_\xi$ for $\xi < \lh(\cT)$ be the copy maps according to $\pi \upharpoonright \bar\cM$ and let $\cU_0 = \pi^\cT$ be the corresponding copy tree. 

  \begin{case}\label{case:iterabilityoneextender}
    $\cT$ uses a single non-backgrounded extender $E$ in addition to fully backgrounded extenders.
  \end{case}
  
  The iteration tree $\cT$ naturally splits into an iteration tree $\cT_1$ before applying $E$ and an iteration tree $\cT_2$ after applying $E$ as backgrounded extenders after $E$ cannot be applied to a model before the application of $E$, i.e., a model in $\cT_1$, according to the rules of a normal iteration tree. Let $\bar\cM^1 = \cM_\eta^{\cT_1}$ denote the model in $\cT_1$ to which $E$ gets applied and write $\cT_1 \upharpoonright \bar\cM^1$ for the tree $\cT_1 \upharpoonright (\eta+1)$ with last model $\bar\cM^1$. Moreover, let $\iota_0\cT_1$ be the copy tree and $\iota_\eta$ the copy map resulting from copying $\cT_1$ up to $\bar\cM^1$ onto $\cM$, cf. Figure \ref{fig:RealizationSingleNonBackgroundedExtender} below.
  
  \begin{claim}\label{cl:realizationtau}
    There is a realization map $\tau \colon \Ult(\bar\cM^1,E) \rightarrow \cM^1$.
  \end{claim}
  \begin{proof}
  Let $\iota_\xi \colon \bar\cN \rightarrow \cN$ be the copy map resulting from copying $\cT_1$ onto $\cM$ via $\iota_0$. Let $\cW^1$ and $\cW_\cN$ be the results of lifting the iteration from $\cM$ to $\cM^1$ as well as the iteration from $\cM^1$ to $\cN$ onto an iteration of the background model $\cW$. 
  
      \begin{figure}[htb]
      \begin{tikzpicture}[decoration=snake]
       
        \draw[->] (2.65,-2.45) node[below=0.1cm] {$\bar\cM^1$} -- node[left] {$\iota_\eta$} (2.65,-0.25) node[above]
        {$\cM^1$};

        \draw[->,decorate] (3.1,-2.75) node[left] {}-- node[below=0.1cm] {} (6.6,-2.75) node[right] {$\bar\cN \ni E$};
        
        \draw[->,decorate] (3.1,-3.75) node[left] {$\bar\cW^1$}-- node[below=0.1cm] {} (6.6,-3.75) node[right] {$\bar\cW_\cN$};
        
        \draw[->,decorate] (3.1,-0.1) node[left] {}-- node[above=0.1cm] {} (6.6,-0.1) node[right] {$\cN$};
        
        \draw[->,decorate] (3.1,0.9) node[left] {$\cW^1$}-- node[above=0.1cm] {} (6.6,0.9) node[right] {$\cW_\cN$};
        
        \draw[->] (2.8,-2.5) node[left] {}-- node[right] {$\;\;\pi_E$} (4,-1.4) node[right]
        {$\Ult(\bar\cM^1,E)$}; 

        \draw[->,dashed] (4,-1.3) node[left] {}-- node[right] {$\;\tau$} (2.8,-0.25) node[right] {};
        
        \draw[->] (6.85,-2.45) node[below=0.1cm] {} -- node[right] {$\iota_\xi$} (6.85,-0.3) node[above] {};
        
      \end{tikzpicture}
      \caption{The realization map $\tau$.}\label{fig:RealizationSingleNonBackgroundedExtenderClaim}
    \end{figure}
  
  As $E$ is non-backgrounded with critical point above $\kappa_i$, $\iota_\xi(E)$ is generically countably complete in $\cW_\cN$ and has critical point at least $\kappa_{i+1}$. Recall that $|\bar\cN| < \delta_{i+1} < \kappa_{i+1}$ as $\lambda = \lh(\cT) < \delta_{i+1}$.
  Therefore, $\iota_\xi(E)$ is countably complete in $\cW_\cN[G]$, where $G$ is $\Col(\omega,\delta_{i+1})$-generic over $\cW_\cN$ and the length $\lambda_E$ of $E$ is countable in $\cW_\cN[G]$. So if we write $\kappa = \crit(\iota_\xi(E))$ we can fix a function
  \[ \rho \colon \iota_\xi \pwimg [\lambda_E]^{{<}\omega} \rightarrow \kappa \] such that for each $a \in [\lambda_E]^{{<}\omega}$ and $A \subseteq [\crit(E)]^{|a|}$, \[ A \in E_a \Leftrightarrow \iota_\xi(A) \in \iota_\xi(E)_{\iota_\xi(a)} \Leftrightarrow \rho \pwimg \iota_\xi(a) \in \iota_\xi(A). \] 
  Here, note that $\iota_\xi(A) \in \iota_\xi(E)_{\iota_\xi(a)} \Rightarrow \rho \pwimg \iota_\xi(a) \in \iota_\xi(A)$ holds by definition of countable completeness, cf. Definition \ref{def:genctblecompleteness}, and $\iota_\xi(A) \in \iota_\xi(E)_{\iota_\xi(a)} \Leftarrow \rho \pwimg \iota_\xi(a) \in \iota_\xi(A)$ follows easily as $\iota_\xi(A) \notin \iota_\xi(E)_{\iota_\xi(a)} \Leftrightarrow [\kappa]^{|a|} \setminus \iota_\xi(A) \in \iota_\xi(E)_{\iota_\xi(a)} \Rightarrow \rho \pwimg \iota_\xi(a) \in [\kappa]^{|a|} \setminus \iota_\xi(A) \Leftrightarrow
  \rho \pwimg \iota_\xi(a) \notin \iota_\xi(A).$
  
    Let $\tau \colon \Ult(\bar\cM^1,E) \rightarrow \cM^1$ be given by \[ [a,f] \mapsto \iota_\eta(f) (\rho \pwimg \iota_\xi(a)), \] for $a \in [\lambda_E]^{{<}\omega}$ and $f \in \bar\cM^1$, cf. Figure \ref{fig:RealizationSingleNonBackgroundedExtenderClaim}. Note that $\tau$ is well-defined, i.e., does not depend on the choice of $a$ and $f$ representing $[a,f]$. As $\bar\cM^1$ and $\bar\cN$ agree on their subsets of $\crit(E)$ as well as $\cM^1$ and $\cN$ agree on their subsets of $\kappa = \crit(\iota_\xi(E))$, we have $\iota_\xi(A) = \iota_\eta(A)$ for all $A \in \cP(\crit(E)) \cap \bar\cM^1$.
    
    We now argue that $\tau$ is elementary to finish the proof of Claim \ref{cl:realizationtau}. Let $[a,f] \in \Ult(\bar\cM^1,E)$ and let $\varphi$ be a formula. Let \[ A = \{ u \in \bar\cM^1 \mid \bar\cM^1 \vDash \varphi(f(u)) \}. \] Then 
    \begin{align*}
        \Ult(\bar\cM^1, E) \vDash \varphi([a,f]) & \text{ iff } 
        A \in E_a \\
        & \text{ iff } \iota_\xi(A) \in \iota_\xi(E)_{\iota_\xi(a)}\\
        & \text{ iff } \rho \pwimg \iota_\xi(a) \in \iota_\xi(A)\\
        & \text{ iff } \rho \pwimg \iota_\xi(a) \in \iota_\eta(A)\\
        & \text{ iff } \rho \pwimg \iota_\xi(a) \in \{ u \in \cM^1 \mid \cM^1 \vDash \varphi(\iota_\eta(f)(u)) \}\\
        & \text{ iff } \cM^1 \vDash \varphi(\iota_\eta(f)(\rho \pwimg \iota_\xi(a)))\\
        & \text{ iff } \cM^1 \vDash \varphi(\tau([a,f])).
    \end{align*}
    
  \end{proof}

    Let $\iota_{\xi+1}$ be the realization map $\tau$ in Claim \ref{cl:realizationtau}. Then we consider the copy tree $\iota_{\xi+1}\cT_2$ and obtain a realization of $\cT$ into a backgrounded iteration of $\cM$ as in Figure \ref{fig:RealizationSingleNonBackgroundedExtender}. Here $\bar\cN$ is the last model of $\cT_1$ and $\cN$ is the last model of the copy tree $\iota_0\cT_1$. Moreover, we write $\iota_0\cT_1 \upharpoonright \cM^1$ for the copy tree of $\cT_1 \upharpoonright \bar\cM^1$ with last model $\cM^1$.
    
    \begin{figure}[htb]
      \begin{tikzpicture}[decoration=snake]
        \draw[->,decorate] (0,0) node[left] {$\cM$}-- node[above=0.1cm] {$\iota_0\cT_1$} (2.3,0) node[right]
        {$\cM^1$}; 

        \draw[->,decorate] (0,-2.75) node[left] {$\bar\cM$}-- node[below=0.1cm] {$\cT_1$} (2.3,-2.75) node[right] {$\bar\cM^1$};
       
        \draw[->] (-0.3,-2.4) -- node[left] {$\iota_0$} (-0.3,-.25);
        \draw[->] (2.65,-2.4) -- node[left] {$\iota_\eta$} (2.65,-0.25);
        
        \draw[->,decorate] (3.1,0) node[left] {}-- node[above=0.1cm] {$\;\;\;\;\iota_{\xi+1}\cT_2$} (5.7,0) node[right] {}; 

        \draw[->,decorate] (3.1,-2.75) node[left] {}-- node[below=0.1cm] {$\cT_1\;$} (4.5,-3.5) node[right] {$\bar\cN \ni E$};
        
        \draw[->,decorate] (3.1,0.1) node[left] {}-- node[above=0.1cm] {$\iota_0\cT_1\;\;\;\;\;\;$} (4.5,1.1) node[right] {$\cN$};
        
        \draw[->] (2.8,-2.5) node[left] {}-- node[right] {$\;\;\pi_E$} (4,-1.4) node[right]
        {$\Ult(\bar\cM^1,E)$}; 

        \draw[->] (4,-1.3) node[left] {}-- node[right] {$\;\iota_{\xi+1}$} (2.8,-0.25) node[right] {};
        
        \draw[->,decorate] (6.0,-1.4) node[left] {}-- node[below=0.1cm] {$\cT_2$} (7.2,-1.4) node[right] {}; 
        
      \end{tikzpicture}
      \caption{Realization into a backgrounded iteration of $\cM$ when applying a single non-backgrounded extender $E$.}\label{fig:RealizationSingleNonBackgroundedExtender}
    \end{figure}
    
    Finally, we define $\Lambda(\cT)$ to be the canonical branch through $\cT$ (or more precisely through $\cT_2$) obtained from the iteration strategy for the background model $\cW$ by lifting the stack of the two trees $(\iota_0\cT_1\upharpoonright\cM^1, \iota_{\xi+1}\cT_2)$ to a stack of two trees on $\cW$. Note that $\Lambda(\cT)$ depends on our choice of realization map $\tau$ from Claim \ref{cl:realizationtau}.


\setcounter{ca}{2}
  \begin{case}\label{case:iterabilityfinitelymanyextenders}
    $\cT$ uses more than one extender that is not fully backgrounded but these overlapping extenders are not used cofinally often in the tree.
  \end{case}
  
  This case is a straightforward generalization of the previous case. Suppose for example that $\cT$ uses two non-backgrounded extenders $E$ and $F$. If the second non-backgrounded extender $F$ gets applied to an earlier model in the tree, before the first non-backgrounded extender $E$ got used, we realize the ultrapower into the corresponding iterate of $\cM$ and continue there. This case is illustrated in Figure \ref{fig:RealizationTwoNonBackgroundedExtendersA}, where we for simplicity assume that $F$ gets applied to $\bar\cM$. We leave the technical details to the reader.
  
  \begin{figure}[htb]
      \begin{tikzpicture}[decoration=snake]
        \draw[->,decorate] (-3,0) node[left] {$\cM$}-- node[above=0.1cm] {$\iota_0\cT_1$} (2.3,0) node[right]
        {$\cM^1$}; 

        \draw[->,decorate] (-3,-2.75) node[left] {$\bar\cM$}-- node[below=0.1cm] {$\cT_1$} (2.3,-2.75) node[right] {$\bar\cM^1$};
       
        \draw[->] (-3.3,-2.4) -- node[left] {$\iota_0$} (-3.3,-.25);
        \draw[->] (2.65,-2.4) -- node[left] {} (2.65,-0.25);
        
        \draw[->,decorate] (3.1,0) node[left] {}-- node[above=0.1cm] {$\;\;\;\;\iota_{\xi+1}\cT_2$} (5.7,0) node[right] {$\cM^{2}$}; 

        \draw[->,decorate] (3.1,-2.75) node[left] {}-- node[below=0.1cm] {$\cT_1\;$} (4.5,-3.5) node[right] {$\bar\cN \ni E = E_\xi^\cT$};
        
        \draw[->,decorate] (3.1,0.1) node[left] {}-- node[above=0.1cm] {$\iota_0\cT_1\;\;\;\;\;\;$} (4.5,1.1) node[right] {$\cN$};
        
        \draw[->] (2.8,-2.5) node[left] {}-- node[right] {$\;\;\pi_E$} (4,-1.4) node[right]
        {$\Ult(\bar\cM^1,E)$}; 

        \draw[->] (4,-1.3) node[left] {}-- node[right] {$\;\iota_{\xi+1}$} (2.8,-0.25) node[right] {};
        
        \draw[->,decorate] (6.0,-1.4) node[left] {}-- node[below=0.1cm] {$\cT_2$} (7.2,-1.4) node[right] {$\bar\cM^{2} \ni F$};
        
        \draw[->] (7.4,-1.2) node[left] {}-- node[right] {$\;\iota_\zeta$} (6.2,-0.25) node[right] {};
        
        \draw[->] (-3,-2.5) node[left] {}-- node[right] {$\;\;\pi_F$} (-1.8,-1.4) node[right]
        {$\Ult(\bar\cM,F)$}; 

        \draw[->] (-1.8,-1.3) node[left] {}-- node[right] {$\;\iota_{\zeta+1}$} (-3,-0.25) node[right] {};
        
        \draw[->,decorate] (0.05,-1.4) node[left] {}-- node[below=0.1cm] {$\cT_3$} (1.3,-1.4) node[right] {}; 
        
        \draw[->,decorate] (-2.95,0.15) node[left] {}-- node[above=0.1cm] {$\iota_{\zeta+1}\cT_3\;\;\;\;\;\;\;\;$} (-1.5,1.1) node[right] {}; 
      \end{tikzpicture}
      \caption{Realization into a backgrounded iteration of $\cM$ when applying two non-backgrounded extenders $E$ and $F$ with $\crit(E) > \crit(F)$.}\label{fig:RealizationTwoNonBackgroundedExtendersA}
    \end{figure}
    
    
    The cofinal well-founded branch $\Lambda(\cT)$ through $\cT$ (or more precisely through $\cT_3$) is as in the previous case obtained by lifting the iteration tree $\iota_{\zeta+1}\cT_3$ of $\cM$ to the background $\cW$.
  
  If we do not go back to an earlier model in the tree, e.g. if the critical point of the second non-backgrounded extender $F$ is larger than the critical point of the first non-backgrounded extender $E$ that got used, we continue similarly as in Case \ref{case:iterabilityoneextender} and copy to a stack of length three on $\cM$, cf. Figure \ref{fig:RealizationTwoNonBackgroundedExtendersB}. We again leave the technical details to the reader.
  
  \begin{figure}[htb]
      \begin{tikzpicture}[decoration=snake]
        \draw[->,decorate] (0,1) node[left] {$\cM$}-- node[above=0.1cm] {$\iota_0\cT_1$} (2.3,1) node[right]
        {$\cM^1$}; 

        \draw[->,decorate] (0,-2.75) node[left] {$\bar\cM$}-- node[below=0.1cm] {$\cT_1$} (2.3,-2.75) node[right] {$\bar\cM^1$};
       
        \draw[->] (-0.3,-2.4) -- node[left] {$\iota_0$} (-0.3,0.75);
        \draw[->] (2.65,-2.4) -- node[left] {} (2.65,0.75);
        
        \draw[->,decorate] (3.05,1) node[left] {}-- node[above=0.1cm] {$\;\;\;\;\iota_{\xi+1}\cT_2$} (7.2,1) node[right] {$\cM^{2}$}; 

        \draw[->,decorate] (3.1,-2.75) node[left] {}-- node[below=0.1cm] {$\cT_1\;$} (4.5,-3.5) node[right] {$\bar\cN \ni E = E_\xi^\cT$};
        
        \draw[->,decorate] (3.1,1.1) node[left] {}-- node[above=0.1cm] {$\iota_0\cT_1\;\;\;\;\;\;$} (4.5,2.1) node[right] {$\cN$};
        
        \draw[->] (2.8,-2.5) node[left] {}-- node[right] {$\;\;\pi_E$} (4,-1.4) node[right]
        {$\Ult(\bar\cM^1,E)$}; 

        \draw[->] (4,-1.3) node[left] {}-- node[right] {$\;\iota_{\xi+1}$} (2.8,0.75) node[right] {};
        
        \draw[->,decorate] (6.0,-1.4) node[left] {}-- node[below=0.1cm] {$\cT_2$} (7.2,-1.4) node[right] {$\bar\cM^{2}$};
        
        \draw[->] (7.5,-1.2) node[left] {}-- node[right] {} (7.5,0.75) node[right] {};
        
        \draw[->,decorate] (7.75,-1.65) node[left] {}-- node[below] {$\cT_2\;\;\;\;$} (9.1,-2.5) node[right] {$\bar\cN^* \ni F = E_\zeta^\cT$};
        
        \draw[->,decorate] (7.95,1.1) node[left] {}-- node[above=0.1cm] {$\iota_{\xi+1}\cT_2\;\;\;\;\;\;\;\;$} (9.35,2.1) node[right] {$\cN$};
        
        \draw[->] (7.8,-1.2) node[left] {}-- node[below] {$\;\;\;\;\pi_F$} (9,-0.55) node[right]
        {$\Ult(\bar\cM^2,F)$}; 

        \draw[->] (9,-0.45) node[left] {}-- node[right] {$\;\iota_{\zeta+1}$} (7.75,0.75) node[right] {};
        
        \draw[->,decorate] (11.0,-0.55) node[left] {}-- node[below=0.1cm] {$\cT_3$} (12.2,-0.55) node[right] {}; 
        
        \draw[->,decorate] (8,1) node[left] {}-- node[above=0.1cm] {$\;\;\;\;\iota_{\zeta+1}\cT_3$} (11,1) node[right] {}; 
        
      \end{tikzpicture}
      \caption{Realization into a backgrounded iteration of $\cM$ when applying two non-backgrounded extenders $E$ and $F$ with $\crit(E) < \crit(F)$.}\label{fig:RealizationTwoNonBackgroundedExtendersB}
    \end{figure}
  
   The cofinal well-founded branch $\Lambda(\cT)$ through $\cT$ (or more precisely through $\cT_3$) is as in the previous case obtained by lifting the stack of iteration trees $(\iota_0\cT_1\upharpoonright\cM^1, \iota_{\xi+1}\cT_2\upharpoonright\cM^2,\iota_{\zeta+1}\cT_3)$ on $\cM$ to the background $\cW$. Here $\iota_0\cT_1\upharpoonright\cM^1$ and $\iota_{\xi+1}\cT_2\upharpoonright\cM^2$ denote the copy trees of $\cT_1 \upharpoonright \bar\cM^1$ and $\cT_2 \upharpoonright \bar\cM^2$ with last models $\cM^1$ and $\cM^2$ respectively. 

   The case when applying two non-backgrounded extenders $E$ and $F$ with $\crit(E) = \crit(F)$ is again similar.

  \begin{case}
    $\cT$ uses cofinally many non-backgrounded extenders.
  \end{case}
  
  By Lemma \ref{lem:iterationtreealmostlinear}, $\cT$ is almost linear on strongness witnesses. Therefore, there is a canonical cofinal branch $b$ through $\cT$ in this case and we can let $\Lambda(\cT) = b$. To continue, we can consider the tree $\cT^\frown b$ of length $\lambda+1$ and let $\iota_\lambda$ be the limit of $\iota_\xi$ for $\xi \in b$.
  
  This finishes the definition of $\Lambda$. It is now straightforward to check that $\Lambda$ is an iteration strategy for $\bar\cM$ above $\kappa_i$ for trees of length ${<}\delta_{i+1}$.
\end{proof}




\subsection{Comparison against a hull}

The following lemma will be crucial in the proof that the construction of $\cM$ converges and all $\kappa_i$ for $i<\omega$ are fully strong in $\cM$. 

\begin{lemma}\label{lem:CompAgainstHull}
    Let $i < \omega$, suppose $(\mathsf{IH})_{i}$ and let \[\pi \colon \bar\cW \rightarrow \cW|\Omega \] for some sufficiently large\footnote{Note that we have not shown at this point that the construction of $\cM^{(i+1)}$ converges. In case $\cM^{(i+1)} \cap \Ord < \Ord$ pick $\Omega > \cM^{(i+1)} \cap \Ord$ and let $\cM^{(i+1)}|\Omega = \cM^{(i+1)}$.} $\Omega$ be such that letting $\crit(\pi) = \nu$, $\cW|\kappa_i \subseteq \bar\cW$, $(\kappa_j, \delta_j \mid j < \omega) \in \rng(\pi)$, the sequence of models of the construction of $\cM^{(i+1)}|\Omega$ is in the range of $\pi$, $\nu \in (\kappa_i, \delta_{i+1})$ is an inaccessible cardinal in $\cW$, and $\pi(\nu) = \delta_{i+1}$. Let $\bar\cM$ be the collapse of $\cM^{(i+1)}|\Omega$. 
    Then there is an iterate $\cM^*$ of $\bar\cM$ such that $\bar\cM$-to-$\cM^*$ is non-dropping, $\cM^* = \cM_\chi$ for some level $\cM_\chi$ in the construction of $\cM^{(i+1)}$, and if \[i \colon \bar\cM \rightarrow \cM^*\] denotes the iteration embedding, the critical point of $i$ is above $\kappa_i$.
\end{lemma}

\begin{proof}
  Let $(\bar\kappa_j, \bar\delta_j \mid j<\omega)$ be the preimages of $(\kappa_j, \delta_j \mid j < \omega)$ under $\pi$ (so, in particular, $\bar\kappa_j = \kappa_j$ and $\bar\delta_j = \delta_j$ for $j \leq i$). For notational simplicity we write $\cM$ for $\cM^{(i+1)}|\Omega$ and $\cW$ for $\cW|\Omega$.
  By Lemma \ref{lem:iterability} we have that $\bar\cM$ is $(\delta_{i+1}, \delta_{i+1})$-iterable with respect to extenders with critical point above $\kappa_i$ and we can compare its construction against the construction of $\cM$ (cf., e.g. \cite{SaZe} or \cite{St22}). We have $|\bar\cM| < \delta_{i+1}$ and we compare the construction of $\bar\cM$ against models $\cM_\xi$ of the construction of $\cM$ with $|\cM_\xi| < \delta_{i+1}$. We will see below that the comparison against the construction of $\cM$ will finish successfully before reaching models $\cM_\xi$ of size larger than $\delta_{i+1}$. So the iteration trees we consider on models in the construction of $\bar\cM$ will have length at most $\max(|\bar\cM|, |\cM_\xi|)^+ < \delta_{i+1}$ for ordinals $\xi$ with $|\cM_\xi| < \delta_{i+1}$. In addition, we will argue in Claim \ref{cl:noextenderleqi} below that all extenders used in this comparison on the $\bar\cM$-side have critical point above $\kappa_i$, so the iterability given by Lemma \ref{lem:iterability} suffices for $\bar\cM$. Moreover, we start with arguing that no disagreement is caused by the $\cM$-side, so the construction of $\cM$ does not move in the comparison process. This also shows that we do not need iterability for $\cM$ itself for the comparison argument.

  \begin{claim}\label{cl:Mdoesnotmove}
    The construction of $\cM$ does not move in the comparison.
  \end{claim}
  \begin{proof}
    Let $(\cM_\xi,\cN_\xi \mid \xi \leq \Omega)$ be the models in the construction of $\cM$ and suppose toward a contradiction that the construction of $\cM$ does move in the comparison.
    Let $E$ be the first extender on some $\cM_\xi$ 
 in the construction of $\cM$ that is used. Write $\cM^*$ for the iterate of the $\bar\cM$-side we have obtained up to this stage in the comparison and $\cT$ for the corresponding iteration tree of length $\lambda+1$ for some ordinal $\lambda$, i.e., $\cM_\lambda^\cT = \cM^*$. As the construction of $\cM$ wins the comparison, there is no drop on the main branch of the $\bar\cM$-side and we have an elementary embedding $i = i_{0\lambda}^\cT \colon \bar\cM_\xi \rightarrow \cM^*$ for $\bar\cM_\xi$ a model in the construction of $\bar\cM$.

 It is a standard argument that fully backgrounded extenders with critical point above $\kappa_i$ do not move in such a comparison (see for example \cite[Lemma 2.11]{Sa15} or \cite{SaZe}), so we can suppose that $\crit(E) \in \{\kappa_j \mid j < \omega\}$.
 $\bar\cM$ cannot iterate up to $\delta_{i+1} < \kappa_{i+1}$ as, by the choice of $\bar\cW$, $\bar\delta_{i+1}$ is not Woodin in $\cW$ and (as we will argue below) the iteration of the construction of $\bar\cM$ takes place above $\bar\delta_{i+1}$.

 \medskip

\noindent\textbf{Extenders on both sides:}
    We start by arguing that if there is an extender $F$ in $\cM^*$ with the same index as $E$ then $F$ is certified via the same tail strategy of $\Sigma_j$ as $F$ and hence $E = F$, contradicting the assumption that $E$ is used because of a disagreement. To see that $F$ is certified by the same strategy certifying $E$, we consider different cases depending on how the iteration from the construction of $\bar\cM$ to $\cM^*$ looks like. Let $\bar\Sigma$ denote the iteration strategy of $\bar\cW$ obtained as the pullback of $\Sigma$ and $\bar\cM_\xi$ the current model in the construction of $\bar\cM$.

    

    \begin{case}\label{case:allfullybackgrdd}
      On the main branch in the iteration from $\bar\cM_\xi$ to $\cM^*$ all extenders used are fully backgrounded, i.e., for every $E_\alpha^\cT$, $\alpha \in [0,\lambda]_T$, and every $l<\omega$, \[ \crit(E_\alpha^\cT) \neq i_{0\alpha}^\cT(\kappa_l^{\bar\cM_\xi}). \] 
  \end{case}

  As the construction of $\cM$ did not move up to this point of the comparison, this means that disagreements were only caused by fully backgrounded extenders. 
  All extenders used in the iteration from $\bar\cM_\xi$ to $\cM^*$ have critical point above $\kappa_i$ as $\bar\cM | \kappa_i = \cM | \kappa_i$ and none of the backgrounded extenders overlaps $\kappa_i$. 
  Write $\kappa_k^*$ for the image of $\bar\kappa_k$ in $\cM^*$ for all $k < \omega$ and note that $\kappa_k^* = \kappa_k$ for all $k \leq i$. Recall that $\crit(E) = \kappa_j < \kappa_i$, so the extender $F$ on $\cM^*$ has critical point $\kappa_j^* = \kappa_j$.
  
  Since we are comparing the constructions of $\bar\cM$ and $\cM$, we can at the same time consider the iterate $\cW^*$ of $\bar\cW$ via the background extenders and maintain the property that $\cM^*$ is the result of an $\cM$-construction inside $\cW^*$. In particular, $F$ is generically countably complete and hence by Lemma \ref{lem:genctblycompleteextenderscertified} certified in $\cW^*$ as it has critical point $\kappa_j = \kappa_j^* < \kappa_i$ and index above $\kappa_i$.

  By \cite[Theorem 3.3]{Sa15}, inside $\cW$, each $\cP^i$ has a unique iteration strategy $\Sigma_i$. So, since $\bar\cW$ is a hull inside $\cW$, $\bar\Sigma_i$ and $\Sigma_i$ agree in $\cW$ for all $i<\omega$. The $\bar\cW$-to-$\cW^*$ iteration can also be performed inside $\cW$ and by strategy coherence in $\cW$, the tail strategy $\Sigma_i^*$ of $\bar\Sigma_i$ along the $\bar\cW$-to-$\cW^*$ iteration is equal to $\Sigma_i$ for all $i<\omega$. Strategy coherence holds for $\Sigma_i$ by the argument\footnote{This argument works for Mitchell-Steel indexing without the additional fine structural assumptions made in \cite{St22}. We thank Benjamin Siskind for pointing this out.} for \cite[Corollary 7.6.9]{St22} or as a consequence of positionality, which holds for $\Sigma_i$ as shown in \cite{Sa15}. Therefore, not only $\cM^*$ and $\cM$ agree below the index of $F$ but also the strategies of their backgrounds $\cW^*$ and $\cW$ on $\cP^j$. So, in particular, they have the same tail strategy $(\Sigma_j)_{\cM_\infty^{j,-}}$ of $\Sigma_j$. As $E$ and $F$ are both certified via this strategy, this implies $E=F$. 

   \begin{case}\label{case:oneextender}
    The iteration from $\bar\cM_\xi$ to $\cM^*$ uses in addition to fully backgrounded extenders a single extender $G$ that is not fully backgrounded.\footnote{This case might actually occur for extenders $G$ with critical point the image of some $\bar\kappa_l$ for $l>i$.}
  \end{case}

  We again aim to show that $F$ is certified via the same tail strategy of $\Sigma_j$ that certifies $E$ and hence $E$ does not cause a disagreement in the comparison. 
  For notational simplicity we suppose that the iteration from $\bar\cM_\xi$ to $\cM^*$ only uses $G$, i.e., $\cM^* = \Ult(\bar\cM_\xi, G)$. Taking additional iterations via fully backgrounded extenders into account is a straightforward generalization of the argument below so we leave that to the reader. Write $(\kappa_i^*, \delta_i^* \mid i < \omega)$ for the image of $(\bar\kappa_i, \bar\delta_i \mid i < \omega)$ under $\pi_G \colon \bar\cM_\xi \rightarrow \Ult(\bar\cM_\xi, G)$, if these cardinals are defined in $\bar\cM_\xi$.





  Let $\dbar\cW$ be a countable hull of $\bar\cW$ and hence in particular a countable hull of $\cW|\Omega$, let \[ \bar\pi \colon \dbar\cW \rightarrow \bar\cW \] and \[ \sigma \colon \dbar\cW \rightarrow \cW \] be the hull embeddings and $k$ an embedding such that $\sigma = k \circ \bar\pi$. Suppose all relevant objects in $\bar\cW$ and $\cW$ are in the range of $\bar\pi$ and $\sigma$ respectively. Let $\dbar\cM$ be the $\bar\pi$-collapse of $\bar\cM$, $\bar\cM^*$ the $\bar\pi$-collapse of $\cM^*$ and let $\bar\pi(\bar G, \bar F) = (G,F)$. Write $\dbar\Sigma$ for the strategy of $\dbar\cW$, i.e., $\sigma(\dbar\Sigma) = \Sigma$.
  
   \begin{figure}[htb]
      \begin{tikzpicture}
        \draw (0,0) node[left] {$\cM$};

        \draw[->] (0,-1.75) node[left] {$\bar\cM$}-- node[below] {} (1.2,-1.75) node[right] {$\Ult(\bar\cM,G) = \cM^*$};
       
       \draw[->] (0,-3.5) node[left] {$\dbar\cM$}-- node[below] {} (4.5,-3.5) node[right] {$\Ult(\dbar\cM,\bar G) = \bar\cM^*$};
       
       \draw[->] (7.5,-3.5) node[left] {}-- node[below] {} (9,-3.5) node[right] {$\Ult(\bar\cM^*,\bar F)$};

        \draw[->] (-0.35,-1.4) -- node[left] {$k$} (-0.35,-0.25);
        \draw[->] (-0.35,-3.15) -- node[left] {$\bar\pi$} (-0.35,-2);
        \draw[<-, bend right=40] (-0.7,-0.1) to node[left] {$\sigma$} (-0.7,-3.45);
        
        \draw[->] (5,-3.15) -- node[below] {$\tau_{\bar G}^{\bar\cM}\;\;\;$} (0,-2);
        
        \draw[->, bend right=30] (5.5,-3.1) to node[right] {$\;\;\;\tau_{\bar G}^\cM$} (0,-0.15);
        
        \draw[->, bend right=30, dashed] (10,-3.1) to node[right] {$\;\;\;\tau$} (0,0);
      \end{tikzpicture}
      \caption{Realizing $\Ult(\bar\cM^*,\bar F)$.}\label{fig:RealizationDoubleHull}
    \end{figure}

  \begin{subclaim}
    There is a realization map $\tau \colon \Ult(\bar\cM^*,\bar F) \rightarrow \cM$.
  \end{subclaim}
  \begin{proof}
    By case assumption, $G$ is generically countably complete in $\bar\cW$. In particular, $\Ult(\dbar\cM,\bar G)$ can be realized into $\bar\cM$. Let $\tau_{\bar G}^{\bar\cM} \colon \Ult(\dbar\cM,\bar G) \rightarrow \bar\cM$ be the realization map and let \[\tau_{\bar G}^{\cM} = k \circ \tau_{\bar G}^{\bar\cM}\] be a realization of $\Ult(\dbar\cM,\bar G)$ into $\cM$, cf. Figure \ref{fig:RealizationDoubleHull}. Write $H = \tau_{\bar G}^{\cM}(\bar F)$. Recall that $E$ is on the sequence of $\cM$ and $\crit(E) = \kappa_j$ for $j \leq i$. As $\cM^*$ results from a comparison against the construction of $\cM$, $\cM^*$ and $\cM$ agree below the index of $E$. In particular, $\kappa_j^* = \kappa_j$. Therefore, in $\cM^* = \Ult(\bar\cM, G)$, $\crit(F) = \kappa_j^*$ and $F$ is on the sequence of $\cM^*$. By elementarity, the analogous statement holds for $\bar F$ in $\Ult(\dbar\cM,\bar G)$.
    Therefore, the critical point of $H$ is some $\kappa_j$, $j<\omega$, and $H$ is on the sequence of $\cM$. This yields that $H$ is generically countably complete so there is a realization map $\tau \colon \Ult(\bar\cM^*,\bar F) \rightarrow \cM$, as desired.
  \end{proof}

  Using the previous subclaim we can argue as in the proof of Lemma \ref{lem:genctblycompleteextenderscertified} to obtain that $\bar F$ is certified via a tail of $\dbar\Sigma_j$. By elementarity and an argument as in Case \ref{case:allfullybackgrdd} we can obtain that $F$ is certified via the same tail strategy of $\Sigma_j$ that certifies $E$.

 \begin{case}\label{case:finitelymanyextenders}
    The iteration from $\bar\cM$ to $\cM^*$ uses in addition to fully backgrounded extenders finitely many extenders that are not fully backgrounded.
  \end{case}

  This is a straightforward generalization of the previous case using the fact that we can still find realization maps for iterates of countable substructures as in the iterability proof, cf. Lemma \ref{lem:iterability}. Therefore, we also get in Case \ref{case:finitelymanyextenders} that $E$ does not cause a disagreement.

  \begin{case}
    $\cM^*$ results from a limit stage of the iteration and cofinally often below this stage an extender with critical point the image of some $\bar\kappa_i$, $i<\omega$, was applied.
  \end{case}
  Let $(E_l \mid l<\gamma)$ be the sequence of extenders used that have critical point the image of some $\bar\kappa_l$ for $l<\omega$. Recall that the extenders need to be of increasing length and note that if there is some $l < k < \gamma$ such that $\crit(E_l) > \crit(E_k)$ and $E_k$ is the first extender where this happens after stage $l$, then $E_k$ gets applied to $\bar\cM_\xi$. As we are only interested in the extenders on the main branch leading from $\bar\cM_\xi$ to $\cM^*$, we might as well assume that the sequence $(\crit(E_l) \mid l<\gamma)$ is nondecreasing.

  \begin{subcase}\label{subcase:samecriticalpoint}
    The critical points of all $E_l$ are images of the same $\bar\kappa_k$ for some fixed $k<\omega$.
  \end{subcase}

  Note that even in the case that $\crit(F) = \kappa_k^*$ we can define the direct limit system that gives rise to the direct limit $(\cM_\infty^k)^*$ to characterize $F$ in $\cM^*$ as in Section \ref{sec:directlimitsystems} via the internalized limit of $\cN_{\delta^*}$'s for $\kappa_k^*$-weak cutpoints $\delta^*$ in $\cM^*|\kappa_k^*$ using elementarity between $\bar\cM_\xi$ and $\cM^*$. By agreement in the comparison process, $\cM^*|\kappa_k^* = \cM | \kappa_k$ and $(\cM_\infty^k)^* = \cM_\infty^k$.
  
  
  Now we can finish the argument in this subcase by arguing as in Case \ref{case:oneextender}: Consider a countable hull $\dbar\cW$ of $\bar\cW$ and let $\bar\cM^*$, $\bar F$ be the collapses of $\cM^*$, $F$. By the iterability proof, cf. the proof of Lemma \ref{lem:iterability}, there is an elementary embedding $\iota \colon \bar\cM^* \rightarrow \cM^*$. Hence, we can use the same argument as in Case \ref{case:oneextender} to argue that $\Ult(\bar\cM^*,\bar F)$ can be realized and therefore $\bar F$ is certified.

  \begin{subcase}
    The critical points of the $E_l$ are increasing, i.e., not all $E_l$ have the same critical point, but the critical points are not cofinal in the images of the $\bar\kappa_l$.
  \end{subcase}
  
  Then the sequence $(E_l \mid l<\gamma)$ will on a tail arise from extenders with critical points that are images of the same $\kappa_k$ for some $k<\omega$. In fact, we can split the sequence $(E_l \mid l<\gamma)$ into finitely many blocks, each block consisting of extenders with critical points that are images of the same $\bar\kappa_k$ for some $k<\omega$.
  Now the argument is a straightforward generalization of Subcase \ref{subcase:samecriticalpoint} and Case \ref{case:finitelymanyextenders}.

  \begin{subcase}
    The critical points of the $E_l$ are cofinal in the images of the $\bar\kappa_l$, i.e., in $\cM^*$ the supremum of the critical points of the $E_l$ is equal to the supremum of the images of $\bar\kappa_l$.
  \end{subcase}
  
  In this case the index of $F$ is above the supremum of $\kappa_k^*$ for $k<\omega$ and the critical point of $F$ is $\kappa_j^* = \kappa_j$ by agreement between $\cM^*$ and $\cM$. As in Case \ref{case:oneextender} we can consider a countable hull $\dbar\cW$ of $\bar\cW$ and let $\bar\cM^*$, $\bar F$ be the collapses of $\cM^*$, $F$. By the iterability proof, cf. the proof of Lemma \ref{lem:iterability}, there is an elementary embedding $\iota \colon \bar\cM^* \rightarrow \cM^*$. Hence, we can use the same argument as in Case \ref{case:oneextender} to argue that $\Ult(\bar\cM^*,\bar F)$ can be realized and therefore $\bar F$ is certified.
  
 This finishes the argument that $E$ does not cause a disagreement in the case that there is an extender $F$ with the same index in the $\bar\cM$-side of the comparison.
  
  \medskip

  \noindent\textbf{Extender only on the $\cM$ side:}
  We finally argue that the reason $E$ is used on the $\cM$-side of the comparison cannot be that there is no extender $F$ with the same index as $E$ on $\cM^*$. So say $E = E_\alpha^\cM$ and suppose toward a contradiction that $E_\alpha^{\cM^*} = \emptyset$. Recall that $\crit(E_\alpha^\cM) = \kappa_j \leq \kappa_i$. We start by showing that the iteration on the $\bar\cM$-side does not use any extenders with critical point $\leq \kappa_j$.
  
  \begin{subclaim}
    The iteration from $\bar\cM_\xi$ to $\cM^*$ does not use any extenders with critical point ${\leq}\kappa_j$.
  \end{subclaim}
  \begin{proof}
  Suppose toward a contradiction that there is some extender $F$ with critical point the image of $\kappa_l = \bar\kappa_l \leq \bar\kappa_j = \kappa_j$ that gets used in the iteration from $\bar\cM_\xi$ to $\cM^*$.
  
  \setcounter{ca}{0}
  \begin{case}
      $F$ is used as there is no extender with the same index on the corresponding model $\cM_\zeta$ in the construction of $\cM$.
    \end{case}
    
    In case $\kappa_j < \kappa_i$, recall that we suppose that $\kappa_l$ is fully strong in $\cM$. So there is some extender $E_\beta^\cM$ on the sequence of $\cM$ with critical point $\kappa_l$ and index $\beta$ above the index of $F$ on the $\cM^*$-sequence. If $\kappa_l = \kappa_j = \kappa_i$, $F$ is indexed on the $\cM^*$-sequence below $\delta_{i+1}$ as $\cM^* \cap \Ord < \delta_{i+1}$. Since $\kappa_i$ is ${<}\delta_{i+1}$-strong in $\cM$, there is an extender $E_\beta^\cM$ on the $\cM$-sequence with critical point $\kappa_i$ and $\ind^\cM(E_\beta^\cM) > \ind^{\cM^*}(F)$. In both cases, let $\beta$ be minimal with this property and write $\alpha$ for the $\cM^*$-index of $F$. Then $\alpha$ is a cutpoint in $\Ult(\cM, E_\beta^\cM)$ and the $(\kappa_l,\alpha)$-extender $E$ derived from the ultrapower embedding $\pi^\cM_{E_\beta^\cM}$ is on the $\cM$-sequence by the initial segment condition. So $E_\alpha^\cM = E$ and $E$ is on the sequence of $\cM_\gamma$ for some $\gamma \geq \zeta$. Therefore, even if $F$ is used in the comparison against $\cM_\zeta$ it will not be used in the iteration of $\bar\cM_\xi$ into the construction of $\cM$ as this part of the comparison gets collapsed when moving on to the comparison with $\cM_\gamma$. So we might and will assume that the iteration from $\bar\cM_\xi$ to $\cM^*$ involves comparisons against $\cM_\gamma$ or larger models from the construction of $\cM$.

    \begin{case}
      $F$ is used as there is an extender $G$ with the same index on the corresponding model $\cM_\zeta$ in the construction of $\cM$ and $F \neq G$.
    \end{case}
    
    In this case we can argue that $F$ and $G$ are both certified with the same critical point $\kappa_l$ and via the same strategies. Therefore, $F=G$ and hence this case cannot occur. 
  \end{proof}

  We will now argue that there is an extender $F$ with the same index as $E$ on $\cM^*$. For notational simplicity, we write $\bar\cM_\xi = \bar\cM$ in what follows.
  
  Suppose first that $\kappa_j < \kappa_i$. Then, by our hypothesis, $\kappa_j$ is strong in $\bar\cM$. Hence by the subclaim $\kappa_j$ is strong in $\cM^*$, so there is some extender $E_\beta^{\cM^*}$ with critical point $\kappa_j$ and $\beta > \alpha$. Let $\beta>\alpha$ be minimal such that $E_\beta^{\cM^*}$ is an extender with critical point $\kappa_j$. Then $\alpha$ is a cutpoint in $\Ult(\cM^*,E_\beta^{\cM^*})$ and the $(\kappa_j,\alpha)$-extender $F$ derived from the ultrapower embedding $\pi_{E_\beta^{\cM^*}}^{\cM^*}$ is on the sequence of $\cM^*$ by the initial segment condition. I.e., $F = E_\alpha^{\cM^*}$, contradicting the assumption that $E_\alpha^{\cM^*} = \emptyset$.

  Now consider the case $\kappa_j = \kappa_i$. In this case we use a more complicated argument to derive a contradiction. Let, as before, $E$ be the extender causing the disagreement on the $\cM$-side or any  long enough extender with critical point $\kappa_i$ on the model $\cM_\zeta$ in the construction of $\cM$ that is currently compared against $\cM^*$. Consider the extender $F$ with critical point $\kappa_i$ given by \[ (a,A) \in F \text{ iff } (i(a),A) \in E \] for $a \in [\Ord]^{<\omega}$ and $A \subseteq [\kappa_i]^{|a|}$. By the last subclaim, $\bar\cM \cap \Pot(\kappa_i) = \cM^* \cap \Pot(\kappa_i) = \cM_\zeta \cap \Pot(\kappa_i)$ and this defines an extender $F$ over $\bar\cM$. Let $\bar\xi$ be the least $i$-weak cutpoint above $\kappa_i$ in $\Ult(\bar\cM,F)$. Recall that an ordinal $\alpha$ is an $i$-weak cutpoint above $\kappa_i$ if all extenders overlapping $\alpha$ have critical point $\kappa_j$ for some $j<i$. Then $\bar\eta = (\bar\xi^+)^{\Ult(\bar\cM,F)}$ is the ordinal where $F$ would be indexed if it is on the sequence of $\bar\cM$.

  We argue that $F$ coheres $\bar\cM | \bar\eta$ and is generically countably complete in $\cW$.
    Consider the commuting diagram in Figure \ref{fig:CoherenceDiagram}, where $\sigma$ is given by \[ [a,f]_F \mapsto [i(a),i(f)]_E. \]

    \begin{figure}[htb]
      \begin{tikzpicture}
        \draw[->] (0,0) node[left] {$\cM_\zeta|i(\bar\xi)$}-- node[above] {$j_E$} (3.3,0) node[right]
        {$j_E(\cM_\zeta|i(\bar\xi))$}; 

        \draw[->] (-0.25,-1.75) node[left] {$\bar\cM | \bar\xi$}-- node[below] {$j_F$} (3.55,-1.75) node[right] {$j_F(\bar\cM|\bar\xi)$};
       
        \draw[->] (-0.75,-1.4) -- node[left] {$i$} (-0.75,-0.25);
        \draw[->] (4.2,-1.4) -- node[right] {$\sigma$} (4.2,-0.25);
      \end{tikzpicture}
      \caption{Coherence.}\label{fig:CoherenceDiagram}
    \end{figure}

    For $a \in [\bar\eta]^{<\omega}$, we have $a = [a, \id]_F^{\bar\cM}$, so in particular for all ordinals $\chi \leq \bar\xi$, \[ \sigma(\chi) = \sigma([\chi, \id]_F^{\bar\cM}) = [i(\chi), \id]_E^{\cM_\zeta} = i(\chi). \]
    
    \begin{subclaim}\label{subcl:AgreementBetweenMbarandUltMbarE}
      $\bar\cM|\bar\xi = \Ult(\bar\cM,F)|\bar\xi.$
    \end{subclaim}
    \begin{proof}
      Let $X = \sigma \pwimg (\Ult(\bar\cM, F)|\bar\xi)$ and $Y = i \pwimg (\bar\cM | \bar\xi)$. We will argue that $X = Y$ using that they are both elementary substructures of the same model and have the same ordinals. 
      
      First, note that $X$ and $Y$ are both (uncollapsed) elementary substructures of $\Ult(\cM_\zeta,E)|i(\bar\xi)$ as $\cM_\zeta | i(\bar\xi) = \Ult(\cM_\zeta,E)|i(\bar\xi)$ using that $E$ coheres $\cM_\zeta | \eta$.
      
      $\Ult(\cM_\zeta,E)$ is a fine structural premouse, so every set in $\Ult(\cM_\zeta,E)$ is ordinal definable (with the extender sequence as a predicate). Let $x \in X$ be arbitrary and let $\varphi_x(v, v_0, \dots, v_n)$ be a formula defining $x$ in $\Ult(\cM_\zeta,E)|i(\bar\xi)$. Then \[ \Ult(\cM_\zeta,E)|i(\bar\xi) \vDash \exists\alpha_0, \dots, \alpha_n \, \varphi_x(x, \alpha_0, \dots, \alpha_n) \wedge \forall z (\varphi_x(z, \alpha_0, \dots, \alpha_n) \rightarrow x = z). \] By elementarity, the same statement holds in $X$. So there are ordinals $\xi_0, \dots, \xi_n \in X$ such that  \[ X \vDash  \varphi_x(x, \xi_0, \dots, \xi_n) \wedge \forall z (\varphi_x(z, \xi_0, \dots, \xi_n) \rightarrow x = z). \] As $X \cap \Ord = Y \cap \Ord$, in particular $\xi_0, \dots, \xi_n \in Y$ and we have by elementarity between $X$ and $\Ult(\cM_\zeta,E)|i(\bar\xi)$ as well as between $Y$ and $\Ult(\cM_\zeta,E)|i(\bar\xi)$, \[ Y \vDash \exists y \,
      \varphi_x(y, \xi_0, \dots, \xi_n) \wedge \forall z (\varphi_x(z, \xi_0, \dots, \xi_n) \rightarrow y = z). \] 
      As $Y \preceq \Ult(\cM_\zeta,E)|i(\bar\xi)$ and $x$ is the unique element in $\Ult(\cM_\zeta,E)|i(\bar\xi)$ satisfying $\varphi_x(x, \xi_0, \dots, \xi_n)$, this in fact implies $x \in Y$. So we have shown that $X \subseteq Y$. The same argument shows $Y \subseteq X$ and we have $X = Y$, as desired.
    \end{proof}
    
    The subclaim implies $\bar\cM|\bar\xi \unlhd j_F(\bar\cM|(\kappa_i^+)^{\bar\cM})$. Recall that for coherence we aim to show that in fact \[ \bar\cM|(\bar\xi^+)^{\Ult(\bar\cM,F)} =  j_F(\bar\cM|(\kappa_i^+)^{\bar\cM})|(\bar\xi^+)^{\Ult(\bar\cM,F)}. \] We will use the operators for this argument.
    In the following subclaim we show that $F$ is given by operators and hence is the extender derived from the hull embedding $k \colon S_i(\bar\cM|\kappa_i) \rightarrow S_i(\bar\cM|\bar\xi)$.

    \begin{subclaim}
      $F$ is given by operators.
    \end{subclaim}
    \begin{proof}
    Recall that $E$ is given by operators by Lemmas \ref{lem:backgroundedextenderscertified} and \ref{lem:certifiedextendersaregivenbyoperators}. We will use this fact together with ideas from the argument for condensation of the operators in \cite[Lemma 2.27]{Sa17} to show that $F$ is given by operators as well. Let $\sigma \colon \Ult(\bar\cM,F) \rightarrow \Ult(\cM_\zeta,E)$ be the canonical factor map as defined above and let $\nu < \kappa_0$ be the least Woodin cardinal above $\delta_0$ in $\cM_\zeta$ and $\bar\cM$. 
    Note that $Lp^\infty_\omega(\bar\cM|\nu) = Lp^\infty_\omega(\cM_\zeta|\nu)$ is suitable at $\kappa_0$ and $(\Ord,\Ord)$-iterable via the canonical iteration strategy obtained from $\Sigma$ by lifting the iteration trees to the background model $\cW$ using that all extenders in $Lp^\infty_\omega(\cM_\zeta|\nu)$ are fully backgrounded. Via this iteration strategy we can make $\bar\cM|\bar\xi$ and $\cM_\zeta|\xi$ for $\xi = i(\bar\xi)$ generic over iterates $\bar\cN$ and $\cN$ of $Lp^\infty_\omega(\bar\cM|\nu) = Lp^\infty_\omega(\cM_\zeta|\nu)$. Note that $\cM_\zeta|\xi \in \Ult(\cM_\zeta,E)$ by coherence of $E$ and $\bar\cM|\bar\xi \in \Ult(\bar\cM,F)$ by Subclaim \ref{subcl:AgreementBetweenMbarandUltMbarE}. So we in fact have $\bar\cN \in \Ult(\bar\cM,F)$ and $\cN \in \Ult(\cM_\zeta,E)$. Let $\bar\pi$ and $\pi$ denote the corresponding iteration embeddings, cf. Figure \ref{fig:CoherenceProofGenericityIteration}.

     \begin{figure}[htb]
      \begin{tikzpicture}
        \draw[->] (0,-1.65) node[left] {}-- node[above] {$\pi$} (3.3,0) node[right]
        {$\cN$ \text{ with } $\cN[G] \ni \cM_\zeta | \xi$}; 

        \draw[->] (0,-1.75) node[left] {$Lp^\infty_\omega(\bar\cM|\nu) = Lp^\infty_\omega(\cM_\zeta|\nu)$}-- node[below] {$\bar\pi$} (3.3,-1.75) node[right] {$\bar\cN$ \text{ with } $\bar\cN[\bar G] \ni \bar\cM | \bar\xi$};
       
        \draw[->] (3.6,-1.4) -- node[right] {$\sigma \upharpoonright \bar\cN$} (3.6,-0.25);
      \end{tikzpicture}
      \caption{Genericity iterations.} \label{fig:CoherenceProofGenericityIteration}
    \end{figure}
    
    
    By branch condensation, not only $\pi$ but also $\bar\pi$ is given by the strategy obtained from $\Sigma$ via lifting the iteration trees to the background model. 
    So as $\Sigma$ restricted to trees on $Lp^\infty_\omega(\cM_\zeta|\nu)$ is strongly guided by term relations, we have $\pi = (\sigma \upharpoonright \bar\cN) \circ \bar\pi$.

    Using the ideas from the proof of \cite[Lemma 2.24]{Sa17}, $S_i(\bar\cN) \in \Ult(\bar\cM,F)$ and analogously $S_i(\cN) \in \Ult(\cM_\zeta,E)$. In particular, $\sigma$ is defined on $S_i(\bar\cN)$. As $\bar\pi$ and $\pi$ are via the canonical iteration strategy obtained from $\Sigma$ by lifting trees to the background model, we have the same for their extensions to $S_i(Lp^\infty_\omega(\bar\cM|\nu))$. By \cite[Lemma 2.23]{Sa17}, \[ \bar\pi(S_i(Lp^\infty_\omega(\bar\cM|\nu))) = S_i(\bar\cN) \] and similarly \[ \pi(S_i(Lp^\infty_\omega(\cM_\zeta|\nu))) = S_i(\cN). \] 
      
      We still have $\pi = (\sigma \upharpoonright S_i(\bar\cN)) \circ \bar\pi$ for the lifted embeddings and in particular $\sigma: S_i(\bar\cN) \rightarrow S_i(\cN)$.
      As in addition $\sigma \upharpoonright \bar\cM|\bar\xi \colon \bar\cM|\bar\xi \rightarrow \cM_\zeta|\xi$, standard lifting arguments for elementary embeddings through generic extensions yield that we can lift $\sigma$ to \[ \sigma^+ \colon S_i(\bar\cN)[\bar\cM|\bar\xi] \rightarrow S_i(\cN)[\cM_\zeta|\xi]. \]

      Now, by the argument in \cite[Lemma 2.23]{Sa17}, $S_i(\bar\cM|\bar\xi)$ is the result of a $\cP$-construction\footnote{$\cP$-constructions were introduced by Steel, cf. \cite{SchSt09} or \cite{St08b} for the definition and basic facts about $\cP$-constructions.} over $\bar\cM|\bar\xi$ in $S_i(\bar\cN)[\bar\cM|\bar\xi]$. Analogously, $S_i(\cM_\zeta|\xi)$ is the result of a $\cP$-construction over $\cM_\zeta|\xi$ in $S_i(\cN)[\cM_\zeta|\xi]$. Therefore, by definability of the $\cP$-construction from the extender sequence, \[ \sigma^+ \upharpoonright S_i(\bar\cM|\bar\xi) \colon S_i(\bar\cM|\bar\xi) \rightarrow S_i(\cM_\zeta|\xi). \]

      Write $m \colon S_i(\cM_\zeta|\kappa_i) \rightarrow S_i(\cM_\zeta|\xi)$ and $l \colon S_i(\cM_\zeta | \kappa_i) \rightarrow S_i(\bar\cM|\bar\xi)$ for the hull embeddings. Then by the definition of the hull embeddings and since $\sigma^+ \upharpoonright \cM_\zeta|\kappa_i = \sigma \upharpoonright \cM_\zeta|\kappa_i = \id$, \[ m = (\sigma^+ \upharpoonright S_i(\bar\cM|\bar\xi)) \circ l. \]

      Finally, this implies that $F$ is the extender derived from $l$ and hence given by operators by the following computation. Let $a \in [\bar\eta]^{{<}\omega}$ and $A \subseteq [\kappa_i]^{|a|}$. Then, as the extender $E$ is given by operators and hence equal to the extender derived from the embedding $m$,
      \begin{align*}
          (a,A) \in F &\text{ iff } a \in \pi_F(A)\\
          &\text{ iff } \sigma(a) \in \pi_E(A)\\
          &\text{ iff } \sigma(a) \in m(A)\\
          &\text{ iff } \sigma^+(a) \in \sigma^+(l(A))\\
          &\text{ iff } a \in l(A).
      \end{align*}
    \end{proof}

    As $F$ is given by operators, it agrees with the extender derived from the hull embedding $k \colon S_i(\bar\cM | \kappa_i) \rightarrow S_i(\bar\cM | \bar\xi)$. This implies $\bar\cM | (\bar\xi^+)^{\Ult(\bar\cM,F)} \unlhd j_F(\bar\cM | (\kappa_i^+)^\cM)$ and hence finishes the argument that $F$ coheres $\bar\cM | \bar\xi$.

    To show that $F$ is generically countably complete, let $G$ be generic over $\cW$ for a forcing of size ${<}\kappa_i$ and let $\tau_E \colon \bigcup_k a_k \rightarrow \kappa_i$ be a function witnessing that $E$ is countably complete in $\cW[G]$. Then $\tau_F$ given by $\tau_F(a) = x$ if and only if $\tau_E(i(a)) = x$ witnesses that $F$ is countably complete in $\cW[G]$ by the following argument: Let $\langle (a_k,A_k) \mid k < \omega\rangle$ be a sequence such that $(a_k,A_k) \in F$ for all $k<\omega$. Then $(i(a_k),A_k) \in E$ and therefore $\tau_E \pwimg (i \pwimg a_k) = \tau_E \pwimg i(a_k) \in A_k$. Hence, $\tau_F \pwimg a_k \in A_k$, as desired.

    As $F$ is given by operators and $\Ult(\cM_\zeta, F)|\bar\xi = \Ult(\bar\cM, F) | \bar\xi = \bar\cM | \bar\xi$ by the previous subclaims, Lemma \ref{lem:extenderdefinable} yields $F \in M_1^{\#,\Sigma_i}(\bar\cM|\bar\xi)$. In particular, $F \in \bar\cW$.

Finally, a straightforward search tree argument shows that $F$ is also generically countably complete in $\bar\cW$ by absoluteness of well-foundedness. So $F$ must be added to $\bar\cM$, the result of an $\cM$-construction in $\bar\cW$. But then we can argue that $E_\alpha^{\cM^*} \neq \emptyset$. 

  This finishes the proof that the construction of $\cM$ does not move in the comparison.
  \end{proof}

The following claim is another key observation in the comparison argument.
  
  \begin{claim}\label{cl:noextenderleqi}
    On both sides of the iteration, no extender with critical point ${\leq}\kappa_i$ is used.
  \end{claim}
  \begin{proof}
    By Claim \ref{cl:Mdoesnotmove} it suffices to show this for $\bar\cM$. Suppose the $\bar\cM$-side of the coiteration uses an extender with critical point ${\leq}\kappa_i$ and let $F$ be the first such extender that is used. Let $\cM^*$ be the current stage of the iteration on the $\bar\cM$-side, i.e., $F \in \cM^*$, and let $\cU$ be the iteration tree of length $\lambda+1$ on some model $\bar\cM_\xi$ in the construction of $\bar\cM$  such that $\cM^* = \cM_\lambda^\cU$. As $\cM|\kappa_i = \bar\cM|\kappa_i = \bar\cM_\xi | \kappa_i$, $F$ is indexed above $\kappa_i$. 
    As in $\cM$ no fully backgrounded extender overlaps $\kappa_j$ for any $j<\omega$, $F$ has critical point $\kappa_j$ for some $j \leq i$. We can argue as in the proof of Claim \ref{cl:Mdoesnotmove}, cf. Case \ref{case:oneextender}, to obtain that $F$ is certified.
 
    \begin{case}
      $F$ is used as there is no extender with the same index on the corresponding model $\cM_\zeta$ in the construction of $\cM$.
    \end{case}
    
    \begin{subcase}
      $\crit(F) = \kappa_j$ for $j<i$.
    \end{subcase}
    
    Recall that we suppose that $\kappa_j$ is fully strong in $\cM$ and argue as in the case $\kappa_j < \kappa_i$ in the proof of the previous claim. So there is some extender $E_\beta^\cM$ on the sequence of $\cM$ with critical point $\kappa_j$ and $\beta$ above the index of $F$ (in the $\cM^*$-sequence). Let $\beta$ be minimal with this property and write $\alpha$ for the $\cM^*$-index of $F$. Then $\alpha$ is a cutpoint in $\Ult(\cM, E_\beta^\cM)$ and the $(\kappa_j,\alpha)$-extender $E$ derived from the ultrapower embedding $\pi^\cM_{E_\beta^\cM}$ is on the $\cM$-sequence by the initial segment condition. So $E_\alpha^\cM = E$ and $E$ is on the sequence of $\cM_\nu$ for some $\nu \geq \zeta$. Therefore, even if $F$ is used in the comparison of the constructions of $\bar\cM$ and $\cM$ it will not be used in the iteration from $\bar\cM_\xi$ to $\cM^*$ as the part of the coiteration using $F$ gets collapsed when moving on to the comparison with $\cM_\nu$.
    
    \begin{subcase}
      $\crit(F) = \kappa_i$.
    \end{subcase}
    
    As $\cM^* \cap \Ord < \delta_{i+1}$, there is some extender $E_\beta^\cM$ on the sequence of $\cM$ with critical point $\kappa_i$ and index $\beta$ above the $\cM^*$-index of $F$. Therefore, we can argue as in the previous subcase.

    \begin{case}
      $F$ is used as there is an extender $G$ with the same index on the corresponding model $\cM_\zeta$ in the construction of $\cM$ and $F \neq G$.
    \end{case}
    
    As $F$ and $G$ are both certified with the same critical point $\kappa_j$, $F=G$. Hence, this case cannot occur. 

  \end{proof}
  
  So all in all there is an iterate $\cM^*$ of $\bar\cM$ such that the main branch from $\bar\cM$ to $\cM^*$ does not drop and if we let $i \colon \bar\cM \rightarrow \cM^*$ denote the corresponding iteration embedding, then $\cM^* = \cM_\gamma$ for some $\cM_\gamma$ in the construction of $\cM$ and the critical point of $i$ is above $\kappa_i$.
  \end{proof}

\subsection{The construction converges and the $\kappa_i$ are strong all the way} \label{subsec:constructionworks}

Now we can prove $(\mathsf{IH}.2)_{i+1}$ from $(\mathsf{IH})_i$.

\begin{lemma}\label{lem:constructionconverges}
    Let $i<\omega$ and suppose $(\mathsf{IH})_i$. Then $\cM^{(i+1)}$ is a proper class premouse, i.e., the $\cM$-construction in $\cW$ does not break down when restricting Clause \eqref{item:strongext} to extender with critical points in $\{\kappa_j \mid j \leq i\}$.
\end{lemma}

  \begin{proof}
    Let $(\cM_\xi,\cN_\xi \mid \xi < \zeta), \cN_\zeta$ be the models in the construction of $\cM^{(i+1)}$ and suppose toward a contradiction that the $\cM$-construction breaks down at stage $\zeta$. That means $\cM_\xi = \cC_\omega(\cN_{\xi})$ for $\xi < \zeta$, $\zeta$ is a successor ordinal and $\cC_\omega(\cN_\zeta)$ is not defined. Say $\cM^{(i+1)} = \cN_\zeta = (J_\alpha^{\vec E}, \in, \vec E, F)$ for some generically countably complete extender $F$ with critical point $\kappa_i$ for some $i<\omega$ such that $\delta_{i+1}<\alpha$. Let \[\pi \colon \bar\cW \rightarrow \cW|\Omega \] for some sufficiently large $\Omega > \cM^{(i+1)} \cap \Ord$ be such that letting $\crit(\pi) = \nu$, $\cW|\kappa_i \subseteq \bar\cW$, $(\kappa_j, \delta_j \mid j < \omega) \in \rng(\pi)$, the sequence of models of the construction of $\cM^{(i+1)}$ is in the range of $\pi$, $\nu \in (\kappa_i, \delta_{i+1})$ is an inaccessible cardinal in $\cW$, and $\pi(\nu) = \delta_{i+1}$. Let $\bar\cM$ be the collapse of $\cM^{(i+1)}$ and let $(\bar\kappa_j, \bar\delta_j \mid j<\omega)$ be the preimages of $(\kappa_j, \delta_j \mid j < \omega)$ (so in particular, $\bar\kappa_j = \kappa_j$ and $\bar\delta_j = \delta_j$ for $j \leq i$). By Lemma \ref{lem:CompAgainstHull} there is an iterate $\cM^*$ of $\bar\cM$ such that the main branch from $\bar\cM$ to $\cM^*$ does not drop and if we let \[i \colon \bar\cM \rightarrow \cM^*\] denote the corresponding iteration embedding, then $\cM^* = \cM_\gamma$ for some $\cM_\gamma$ in the construction of $\cM$ and the critical point of $i$ is above $\kappa_i$.
    But this implies that $\cC_\omega(\cM^*)$ and hence $\cC_\omega(\bar\cM)$ is defined, a contradiction.
  \end{proof}

The next lemma proves $(\mathsf{IH}.3)_{i+1}$ from $(\mathsf{IH})_i$, finishing the inductive argument for $(\mathsf{IH})_{i+1}$ from $(\mathsf{IH})_i$.

\begin{lemma}\label{lem:kappaiarestrong}
  Let $(\delta_i \mid i < \omega)$ be the sequence of Woodin cardinals in $\cW$ and for each $i<\omega$, let $\kappa_i$ be the least ${<}\delta_{i+1}$-strong cardinal above $\delta_i$. Let $i<\omega$ and assume $(\mathsf{IH})_i$. Then $\kappa_j$ is a strong cardinal in $\cM^{(i+1)}$ for all $j \leq i$.
\end{lemma}

\begin{proof}
  Suppose there is a $j \leq i$ such that $\kappa_j$ is not $\gamma$-strong in $\cM^{(i+1)}$ for some $\gamma$. We may assume that $\gamma$ is a cardinal. Let $j$ be minimal with this property, i.e., suppose that all cardinals $\kappa_l$ with $l<j$ are fully strong in $\cM^{(i+1)}$. By Lemma \ref{lem:constructionconverges} $\cM^{(i+1)}$ is a proper class model. We may assume $j=i$ as the argument for $j<i$ is similar. Let \[\pi \colon \bar\cW \rightarrow \cW|\Omega \] for some sufficiently large $\Omega$ be such that letting $\crit(\pi) = \nu$, $\cW|\kappa_i \subseteq \bar\cW$, $(\kappa_j, \delta_j \mid j < \omega) \in \rng(\pi)$, the sequence of models of the construction of $\cM^{(i+1)}|\Omega$ is in the range of $\pi$, $\nu \in (\kappa_i, \delta_{i+1})$ is an inaccessible cardinal in $\cW$, and $\pi(\nu) = \delta_{i+1}$. Let $\bar\cM$ be the collapse of $\cM^{(i+1)}|\Omega$ and $(\bar\kappa_j, \bar\delta_j, \bar\gamma \mid j<\omega)$ the preimages of $(\kappa_j, \delta_j, \gamma \mid j < \omega)$ (so in particular, $\bar\kappa_j = \kappa_j$ and $\bar\delta_j = \delta_j$ for $j \leq i$). By Lemma \ref{lem:CompAgainstHull} there is an iterate $\cM^*$ of $\bar\cM$ such that the main branch from $\bar\cM$ to $\cM^*$ does not drop and if we let \[ i \colon \bar\cM \rightarrow \cM^* \] denote the corresponding iteration embedding, then $\cM^* = \cM_\chi$ for some model $\cM_\chi$ in the construction of $\cM^{(i+1)}$ and the critical point of $i$ is above $\kappa_i$.
  As in the proof of Claim \ref{cl:Mdoesnotmove} in the proof of Lemma \ref{lem:CompAgainstHull} we have $\cM^* \cap \Ord < \delta_{i+1}$. Recall that $\kappa_i$ is ${<}\delta_{i+1}$-strong in $\cM^{(i+1)}$, so there are extenders on the $\cM^{(i+1)}$-sequence with arbitrarily large length $\eta$ between $i(\bar\gamma)$ and $\delta_{i+1}$. Let $E$ be an extender on the $\cM_\chi$-sequence with critical point $\kappa_i$ and length $\eta$ between $i(\bar\gamma)$ and $\delta_{i+1}$. 
  Now consider the extender $F$ with critical point $\kappa_i$ given by \[ (a,A) \in F \text{ iff } (i(a),A) \in E \] for $a \in [\Ord]^{<\omega}$ and $A \subseteq [\kappa_i]^{|a|}$. As the critical point of $i$ is above $\kappa_i$, we have $\cM \cap \Pot(\kappa_i) = \bar\cM \cap \Pot(\kappa_i)$ and this defines an extender $F$ over $\bar\cM$. Let $\bar\xi$ be the least $i$-weak cutpoint above $\kappa_i$ in $\Ult(\bar\cM,F)$. Recall that an ordinal $\alpha$ is an $i$-weak cutpoint above $\kappa_i$ if all extenders overlapping $\alpha$ have critical point $\kappa_j$ for some $j<i$. So $\bar\eta = (\bar\xi^+)^{\Ult(\bar\cM,F)}$ is the ordinal where $F$ would be indexed if it is on the sequence of $\bar\cM$. As in the proof of Claim \ref{cl:Mdoesnotmove} in the proof of Lemma \ref{lem:CompAgainstHull} we obtain the following claim.

  \begin{claim}
    $F$ coheres $\bar\cM | \bar\eta$ and is generically countably complete in $\bar\cW$.
   \end{claim}

Hence, $F$ must be added to $\bar\cM$, the result of an $\cM^{(i+1)}$-construction in $\bar\cW$, contradicting the fact that $\kappa_i$ is not $\bar\gamma$-strong in $\bar\cM$.
\end{proof}

As all $\delta_i$, $i<\omega$, remain Woodin in $\cM$, Lemmas \ref{lem:constructionconverges} and \ref{lem:kappaiarestrong} inductively yield that $\delta_\omega = \sup_{i<\omega} \delta_i$ is a limit of Woodin cardinals and a limit of strong cardinals in the proper class model $\cM$. This finishes the proof of Theorem \ref{thm:main}.

\bibliographystyle{plain}
\bibliography{References}

\end{document}